\documentclass[10pt]{amsart}

\usepackage{amssymb}
\usepackage{amsthm}
\usepackage{amsmath}

\usepackage[foot]{amsaddr}

\usepackage{hyperref}

\theoremstyle{plain}
\newtheorem{proposition}{Proposition}[section]
\newtheorem{theorem}[proposition]{Theorem}
\newtheorem{lemma}[proposition]{Lemma}
\newtheorem{corollary}[proposition]{Corollary}
\theoremstyle{definition}
\newtheorem{example}[proposition]{Example}
\newtheorem{definition}[proposition]{Definition}
\newtheorem{observation}[proposition]{Observation}
\theoremstyle{remark}
\newtheorem{remark}[proposition]{Remark}

\newtheorem{question}[proposition]{Question}

\DeclareMathOperator{\Aff}{Aff}

\DeclareMathOperator{\Arg}{Arg}
\DeclareMathOperator{\dimension}{dim}
\DeclareMathOperator{\diam}{diam}
\DeclareMathOperator{\Real}{Re}
\DeclareMathOperator{\Imaginary}{Im}

\DeclareMathOperator{\GL}{GL}

\DeclareMathOperator{\PGL}{PGL}

\DeclareMathOperator{\Hol}{Hol}

\DeclareMathOperator{\arcosh}{arcosh} 
\DeclareMathOperator{\arctanh}{tanh^{-1}} 
\DeclareMathOperator{\degree}{deg}

\DeclareMathOperator{\Cc}{\mathcal{C}}

\DeclareMathOperator{\Hc}{\mathcal{H}}

\DeclareMathOperator{\Lc}{\mathcal{L}}
\DeclareMathOperator{\Nc}{\mathcal{N}}
\DeclareMathOperator{\Oc}{\mathcal{O}}

\DeclareMathOperator{\Tc}{\mathcal{T}}

\DeclareMathOperator{\Vc}{\mathcal{V}}

\DeclareMathOperator{\Cb}{\mathbb{C}}

\DeclareMathOperator{\Nb}{\mathbb{N}}
\DeclareMathOperator{\Pb}{\mathbb{P}}
\DeclareMathOperator{\Rb}{\mathbb{R}}
\DeclareMathOperator{\Xb}{\mathbb{X}}

\newcommand{\abs}[1]{\left|#1\right|}

\newcommand{\norm}[1]{\left\|#1\right\|}
\newcommand{\wt}[1]{\widetilde{#1}}
\newcommand{\wh}[1]{\widehat{#1}}
\newcommand{\ip}[1]{\left\langle #1\right\rangle}

\begin{document}

\title[Gromov hyperbolicity and the Kobayashi metric]{Gromov hyperbolicity and the Kobayashi metric on convex domains of finite type}
\author{Andrew M. Zimmer}\address{Department of Mathematics, University of Michigan, Ann Arbor, MI 48109.}
\email{aazimmer@uchicago.edu}
\date{\today}

\begin{abstract} In this paper we prove necessary and sufficient conditions for the Kobayashi metric on a convex domain to be Gromov hyperbolic.  In particular we show that for convex domains with $C^\infty$ boundary being of finite type in the sense of D'Angelo is equivalent to the Gromov hyperbolicity of the Kobayashi metric. We also show that bounded domains which are locally convexifiable and have finite type in the sense of D'Angelo have Gromov hyperbolic Kobayashi metric. The proofs use ideas from the theory of the Hilbert metric. 
\end{abstract}

\maketitle

{\small
\setcounter{tocdepth}{1}
\tableofcontents
}
\section{Introduction}

In this paper we investigate the asymptotic geometry of the Kobasyashi metric on convex domains. Following Balogh and Bonk~\cite{BB2000} we are particularly interested in when the Kobayashi metric is Gromov hyperbolic. For convex domains with smooth boundary we will show that Gromov hyperbolicity of the Kobayashi metric is equivalent to having finite type in the sense of D'Angelo. Our main strategy for this investigation is to consider the action of affine transformation group on the space of convex sets endowed with the local Hausdorff topology.  This approach is motivated by Benoist's recent work on the Hilbert metric~\cite{B2003}. It is also related to the scaling methods of Pinchuk~\cite{P1991} and Frankel~\cite{F1989b, F1989, F1991} (for an overview see~\cite{KK2008}). 

More precisely, we call an open convex set \emph{$\Cb$-proper} if it does not contain any complex affine lines. These are exactly the convex sets for which the Kobayashi metric is complete~\cite{B1980}. Let $\Xb_d$ be the space of $\Cb$-proper convex open sets in $\Cb^d$ endowed with the local Hausdorff topology. Let $\Aff(\Cb^d)$ denote group of complex affine transformations of $\Cb^d$. By studying the closure of $\Aff(\Cb^d)$-orbits in $\Xb_d$ we will establish necessary and sufficient conditions for the Gromov hyperbolicity of the Kobayashi metric. The main application of these results is the following:

\begin{theorem}\label{thm:main}
Suppose $\Omega$ is a bounded convex set with $C^\infty$ boundary. Then $(\Omega, d_{\Omega})$ is Gromov hyperbolic if and only if $\Omega$ has finite type in the sense of D'Angelo.
\end{theorem}

This answers a conjecture of Gaussier and Seshadri~\cite{GS2013}. This also provides a partial answer to a question of Balogh and Bonk~\cite[Section 6]{BB2000} who asked if the Kobayashi metric is Gromov hyperbolic for a general domain with finite type in the sense of D'Angelo. 

We will also show that the Gromov boundary $\Omega(\infty)$ can be identified with the topological boundary $\partial \Omega$:

\begin{proposition}\label{prop:main}
Suppose $\Omega$ is a bounded convex set of finite type in the sense of D'Angelo. Then the identity map $\Omega \rightarrow \Omega$ extends to a homeomorphism $\Omega \cup \Omega(\infty) \rightarrow \Omega \cup \partial\Omega$.
\end{proposition}

Although the necessary and sufficient conditions developed in this paper are for convex domains, the techniques can be applied to domains which are locally convexifiable:

\begin{theorem}\label{thm:loc_main_i}
Suppose $\Omega$ is locally convexifiable and has finite type in the sense of D'Angelo. Then $(\Omega, d_{\Omega})$ is Gromov hyperbolic. Moreover, the identity map $\Omega \rightarrow \Omega$ extends to a homeomorphism $\Omega \cup  \Omega(\infty) \rightarrow \Omega \cup \partial\Omega$.
\end{theorem}

We actually only require that the boundary $\partial \Omega$ is a $C^L$ hypersurface and near each $\xi \in \partial \Omega$ there exists holomorphic coordinates where $\partial \Omega$ is a convex hypersurface and has line type at most $L$ near $\xi$. Since every strongly pseudo-convex domain with $C^2$ boundary satisfies this hypothesis with $L=2$, we obtain a result of Balogh and Bonk as a corollary: 

\begin{corollary}\cite[Theorem 1.4]{BB2000}
Suppose $\Omega$ is a bounded strongly pseudo convex domain with $C^2$ boundary. Then $(\Omega, d_{\Omega})$ is Gromov hyperbolic.  Moreover, the identity map $\Omega \rightarrow \Omega$ extends to a homeomorphism $\Omega \cup \partial \Omega \rightarrow \Omega \cup \Omega(\infty)$.
\end{corollary}

\begin{remark}  Balogh and Bonk also show that the Carnot-Carath{\'e}odory metric $d_H$ on $\partial \Omega$ lies in the canonical class of snowflake equivalent metrics on $\Omega(\infty)$. It is unclear if our approach can be used to prove this.\end{remark}

We now describe the necessary and sufficient conditions for Gromov hyperbolicity established in this paper. Our first necessary condition is:

\begin{theorem}\label{thm:no_affine_disk}
Suppose $\Omega \subset \Cb^d$ is a  $\Cb$-proper convex open set. If $(\Omega, d_{\Omega})$ is Gromov hyperbolic then $\partial \Omega$ does not contain any non-trivial holomorphic disks.
\end{theorem}

\begin{remark} 
When $\Omega$ is convex, $\partial \Omega$ contains a non-trivial homolmorphic disk if and only if it contains a non-trivial complex affine disk (see~\cite{FS1998}). Moreover special cases of the above theorem are already known:
\begin{enumerate} \item Gaussier and Seshadri~\cite[Theorem 1.1]{GS2013} gave an argument when $\Omega$ is a bounded convex set with $C^\infty$ boundary. 
\item Nikolov, Thomas, and Trybula~\cite[Theorem 1]{NTT2014} gave an argument when $d=2$ and $\Omega$ has $C^{1,1}$ boundary. 
\end{enumerate}
In our proof we use the elementary estimates on the infinitesimal Kobayashi metric
\begin{align*}
\frac{\norm{v}}{2\delta_{\Omega}(p;v)} \leq K_{\Omega}(p;v) \leq \frac{\norm{v}}{\delta_{\Omega}(p;v)} 
\end{align*}
(valid for any convex set $\Omega$, point $p \in \Omega$, and vector $v \in \Cb^d$) to give a general condition on when a line segment in $\Omega$ can be parametrized to be a quasi-geodesic. We then use these quasi-geodesics to construct ``fat'' triangles near any complex affine disk in the boundary. 
\end{remark}

Theorem~\ref{thm:no_affine_disk} becomes an useful tool for demonstrating non-hyperbolicity when combined with the second necessary condition we establish:

\begin{theorem}\label{thm:nec_i}
Suppose $\Omega \subset \Cb^d$ is a $\Cb$-proper convex open set and $(\Omega, d_{\Omega})$ is Gromov hyperbolic. If $\wh{\Omega} \in \overline{\Aff(\Cb^d) \Omega} \cap \Xb_d$ then $(\wh{\Omega}, d_{\wh{\Omega}})$ is Gromov hyperbolic.\end{theorem}

\begin{remark} \ \begin{enumerate}
\item Here we taking the closure of $\Aff(\Cb^d) \Omega$ in the space of open convex sets endowed with the local Hausdorff topology.
\item Clearly one can blow up any open convex set by affine transformations to be all of $\Cb^d$. Thus it is important to only consider limits in $\Xb_d$. 
\item The main step in the proof is showing that the Kobayashi metric is continuous in the local Hausdorff topology. Then Theorem~\ref{thm:nec_i} follows from the Gromov product formulation of hyperbolicity. 
\end{enumerate}
\end{remark}

Using Theorem~\ref{thm:no_affine_disk} and Theorem~\ref{thm:nec_i}  we will demonstrate:

\begin{proposition}\label{prop:infinite_type}
Suppose $\Omega \subset \Cb^d$ is a $\Cb$-proper convex open set with $0 \in \partial \Omega$ and 
\begin{align*}
\Omega \cap \Oc = \{(z_1, \dots, z_d) \in \Cb^d :  \Imaginary(z_1) > f(\Real(z_1), z_2, \dots, z_d)\}
\end{align*}
where $\Oc$ is a neighborhood of the origin and $f:\Rb \times \Cb^{d-1} \rightarrow \Rb$ is a convex non-negative function. If 
\begin{align*}
\lim_{z \rightarrow 0} \frac{f(0,z,0,\dots,0)}{\abs{z}^n} = 0
\end{align*}
for all $n>0$ then $(\Omega, d_{\Omega})$ is not Gromov hyperbolic.
\end{proposition}

\begin{remark} \
\begin{enumerate}
\item If $\Omega \subset \Cb^d$ is  a bounded convex open set with $C^{\infty}$ boundary and a point of infinite linear type then, up to an affine transformation, $\Omega$ satisfies the hypothesis of Proposition~\ref{prop:infinite_type}. 
\item In the proof we will show that it possible to find affine maps $A_n \in \Aff(\Cb^d)$ such that $A_n\Omega$ converges in the local Hausdorff topology to a $\Cb$-proper convex open set whose boundary contains a non-trivial complex affine disk. Theorem~\ref{thm:no_affine_disk} and Theorem~\ref{thm:nec_i} then imply that $(\Omega, d_{\Omega})$ is not Gromov hyperbolic. Using this approach, we avoid the need to develop precise estimates for the Kobayashi metric near the point of infinite linear type. 
\item Nikolov, Thomas, and Trybula~\cite[Theorem 2]{NTT2014} gave an argument for the above proposition when $d=2$ and $f$ satisfies some additional conditions (including being $C^{1,1}$). 
\end{enumerate}
\end{remark}

Delaying the definition of locally m-convex sequences to Section~\ref{sec:m_convex} and the definition of well behaved geodesics to Section~\ref{sec:suff}, our sufficient condition for the Kobasyashi metric to be Gromov hyperbolic is:

\begin{theorem}\label{thm:suff_i}
Suppose $\Omega \subset \Cb^d$ is a $\Cb$-proper convex open set. If for every sequence $u_n \in \Omega$ there exists $n_k \rightarrow \infty$, affine maps $A_k \in \Aff(\Cb^d)$, and a $\Cb$-proper convex open set $\wh{\Omega}$ such that 
\begin{enumerate}
\item $A_k \Omega \rightarrow \wh{\Omega}$ in the local Hausdorff topology,
\item $A_k u_{n_k} \rightarrow u_{\infty} \in \wh{\Omega}$,
\item $(A_k \Omega )_{k \in \Nb}$ is a locally m-convex sequence, and
\item geodesics in $\wh{\Omega}$ are well behaved,
\end{enumerate}
then $(\Omega, d_{\Omega})$ is Gromov hyperbolic.
\end{theorem}

\begin{remark}\ \begin{enumerate}
\item Using language of Frankel~\cite{F1989b, F1989, F1991}, Theorem~\ref{thm:suff_i} says that $(\Omega, d_{\Omega})$ is Gromov hyperbolic given some conditions on every ``affine blow up'' of $\Omega$. 
\item The idea of the proof is to assume that $(\Omega, d_{\Omega})$ is not Gromov hyperbolic. Then for each $n$ there exists  a geodesic triangle $\Tc_n$ with vertices $x_n,y_n,z_n$, edges $\sigma_{x_ny_n}, \sigma_{y_nz_n}, \sigma_{z_nx_n}$ joining them, and a point $u_n \in \sigma_{x_n y_n}$ such that 
\begin{align*}
d_{\Omega}(u_n, \sigma_{y_nz_n} \cup \sigma_{z_nx_n} ) > n.
\end{align*}
Now assume $n_k \rightarrow \infty$ and $A_k \in \Aff(\Cb^d)$ are as in the statement of the theorem. The goal will be to show that the geodesic triangle $A_k \Tc_{n_k}$ in $A_k\Omega$ converges to a geodesic triangle $\Tc_{\infty}$ in $\wh{\Omega}$. Once this is established, the continuity of the Kobayashi metric in the local Hausdorff topology implies that
\begin{align*}
\liminf_{n \rightarrow \infty} d_{\Omega}(u_n, \sigma_{y_nz_n} \cup \sigma_{z_nx_n} ) < \infty.
\end{align*}
which is a contradiction.
\end{enumerate}
\end{remark}

To show that a bounded convex domain $\Omega$ of finite type is Gromov hyperbolic we will first use an argument of Gaussier~\cite{G1997}  to deduce that for any sequence $u_n \in \Omega$ there exists $n_k \rightarrow \infty$, affine maps $A_k \in \Aff(\Cb^d)$, and a $\Cb$-proper convex open set $\wh{\Omega}$ such that 
\begin{enumerate}
\item $A_k \Omega \rightarrow \wh{\Omega}$ in the local Hausdorff topology,
\item $A_k u_{n_k} \rightarrow u_{\infty} \in \wh{\Omega}$,
\item $(A_k \Omega )_{k \in \Nb}$ is a locally m-convex sequence.
\end{enumerate}
Moreover, $\wh{\Omega}$ has the form 
\begin{align*}
\wh{\Omega} = \{ (z_1 \dots, z_d) \in \Cb^{d} : \Real(z_1)  > P(z_2,z_3, \dots, z_d) \}
\end{align*}
where $P$ is a non-negative non-degenerate convex polynomial with $P(0)=0$. In Section~\ref{sec:gromov_prod} and Section~\ref{sec:multi_infty} we will show that geodesics in such a domain are well behaved. Thus, by Theorem~\ref{thm:suff_i}, $(\Omega, d_{\Omega})$ is Gromov hyperbolic. 

\subsection{Motivation from Hilbert geometry} Every proper open convex set $\Omega \subset \Pb(\Rb^{d+1})$ has a projectively invariant metric $H_{\Omega}$ called the \emph{Hilbert metric}. This metric is usually defined using cross ratios, but it has an equivalent formulation which makes it a real projective analogue of the Kobayashi metric (see for instance~\cite{K1977},~\cite{L1986}, or~\cite{G2009}). Thus results about the Hilbert metric can serve as guide to understanding the Kobayashi metric. 

The convex domains for which the Hilbert metric is Gromov hyperbolic are completely understood: Karlsson and Noskov~\cite{KN2002} showed that if $(\Omega, H_{\Omega})$ is Gromov hyperbolic then $\partial \Omega$ is a $C^1$ hypersurface and Benoist~\cite{B2003} characterized the convex domains for which the Hilbert metric is Gromov hyperbolic in terms of the first derivatives of local defining functions for $\partial \Omega$.

A key step in Benoist's characterization is the following:

\begin{theorem}\cite[Proposition 1.6]{B2003}
Suppose $\Omega \subset \Pb(\Rb^{d+1})$ is a proper open convex set. Let $H_{\Omega}$ be the Hilbert metric on $\Omega$. Then $(\Omega, H_{\Omega})$ is Gromov hyperbolic if and only if every proper convex open set in $ \overline{\PGL(\Rb^{d+1})\Omega}$ is strictly convex.
\end{theorem}

This paper can be seen as an attempt to find analogues of Benoist's results for the Kobayashi metric. In particular, Theorem~\ref{thm:no_affine_disk}, Theorem~\ref{thm:nec_i}, and Theorem~\ref{thm:suff_i} can be seen as an analogue of the above theorem. 

We should observe that the Hilbert metric has several important features that the Kobayashi metric lacks. First there is an explicit formula for the Hilbert distance between two points. Second straight lines are geodesics and so the behavior of some geodesics is easy to understand. Finally, convexity is invariant under real projective transformations. This implies that every proper convex set in $\Pb(\Rb^{d+1})$ can be realized as a bounded convex set in some affine chart. Convexity is not invariant under complex projective transformations and this creates many problems. In particular, one is forced to consider unbounded convex sets in many of the arguments in this paper.

Finally, motivated by Benoist's characterization of Gromov hyperbolicity for the Hilbert metric we ask the following question:

\begin{question}
Suppose $\Omega$ is a bounded convex open set with $C^1$ boundary. If $\partial \Omega$ is quasi-symmetric in the sense of Benoist~\cite{B2003} is the Kobayashi metric Gromov hyperbolic?
\end{question}

\subsection*{Acknowledgments} 

I would like to thank my advisor Ralf Spatzier for many helpful conversations and also Yves Benoist for explaining to me some of the motivation behind his work on the Hilbert metric. This material is based upon work supported by the National Science Foundation under Grant Number NSF 1045119.

\section{Preliminaries}\label{sec:prelim}

\subsection{Some Notation}

\begin{itemize}
\item Let $\Delta: = \{ z \in \Cb : \abs{z} < 1\}$ and $\Hc = \{ z \in \Cb : \Imaginary(z) > 0\}$. 
\item For $z \in \Cb^d$ let $\norm{z}$ denote the standard Euclidean norm of $z$. 
\item For $z_0 \in \Cb^d$ and $R>0$ let $B_R(z_0) := \{ z \in \Cb^d : \norm{z-z_0} < R\}$.
\item Given a open set $\Omega \subset \Cb^d$ and $p \in \Omega$ let 
\begin{align*}
\delta_{\Omega}(p):= \inf \left\{ \norm{q-p} : q \in  \Cb^d \setminus \Omega \right\}.
\end{align*}
\item Given a open set $\Omega \subset \Cb^d$, $p \in \Omega$, and $v \in \Cb^d$ nonzero let 
\begin{align*}
\delta_{\Omega}(p;v):= \inf \left\{ \norm{q-p} : q \in (p+\Cb \cdot v)  \cap (\Cb^d \setminus \Omega) \right\}.
\end{align*}
\end{itemize}

\subsection{The Kobayashi Metric} In this section we will review some basic properties of the Kobayashi metric. A nice introduction to the Kobayashi metric and its properties can be found in~\cite{K2005} or~\cite{A1989}.

Given a domain $\Omega \subset \Cb^d$ the \emph{(infinitesimal) Kobayashi metric} is the pseudo-Finsler metric
\begin{align*}
K_{\Omega}(x;v) = \inf \left\{ \abs{\xi} : f \in \Hol(\Delta, \Omega), \ f(0) = x, \ df(\xi) = v \right\}
\end{align*}
and the \emph{Kobayashi pseudo-distance} is 
\begin{align*}
d_{\Omega}(x,y) = \inf \left\{ \int_{0}^1 K_{\Omega}(\gamma(t); \gamma'(t)) dt : \gamma \in C^\infty([0,1],\Omega),  \gamma(0)=x, \text{ and } \gamma(1)=y\right\}.
\end{align*}
 Directly from the definitions one obtains that:
 
 \begin{proposition}
 \label{prop:basic_kob} \
 \begin{enumerate}
  \item Suppose $\Omega \subset \Cb^d$ is an open domain then
 \begin{align*}
 K_{\Omega}(p;v) \leq \frac{\norm{v}}{\delta_{\Omega}(p;v)}.
 \end{align*}
\item  Suppose $f: \Omega_1 \rightarrow \Omega_2$ is a holomorphic map then 
 \begin{align*}
 K_{\Omega_2}\left(f(p); df(v)\right) \leq K_{\Omega_1}\left(p;v\right) 
 \end{align*}
 and
 \begin{align*}
 d_{\Omega_2}\left(f(p_1),f(p_2)\right) \leq d_{\Omega_1}\left(p_1,p_2\right).
 \end{align*}
 \end{enumerate}
 \end{proposition}
 
 Using the Kobayashi pseudo-metric one can establish the following compactness result for holomorphic maps of $\Delta$ into $\Omega$. 

 \begin{proposition}\label{prop:taut}
 Suppose $\Omega \subset \Cb^d$ is open and $d_{\Omega}$ is a complete metric on $\Omega$. If $\varphi_n: \Delta \rightarrow \Omega$ is a sequence of holomorphic maps then either 
 \begin{enumerate}
 \item there exists a subsequence which converges uniformly on compact sets to a holomorphic function $\varphi: \Delta \rightarrow \Omega$ or 
 \item for all $x \in \Delta$ and all compact subsets $K \subset \Omega$ there exists $N>0$ such that  $\varphi_n(x) \notin K$ for all $n >N$.
 \end{enumerate}
 \end{proposition}

For a proof see~\cite[Theorem 2.3.18]{A1989}.

\subsection{The disk and the upper half plane}

Let $\Delta = \{ z \in \Cb : \abs{z} < 1\}$. Then 
\begin{align*}
K_{\Delta}(\zeta; v) = \frac{\abs{v}}{1-\abs{\zeta}^2}\end{align*}
and 
\begin{align*}
d_{\Delta}(\zeta_1,\zeta_2) = \arctanh \abs{\frac{\zeta_1-\zeta_2}{1-\zeta_1\overline{\zeta_2}}}.
\end{align*}
Next let $\Hc = \{ z \in \Cb : \Imaginary(z) >0\}$. Then 
\begin{align*}
K_{\Hc}(\zeta; v) = \frac{\abs{v}}{2\Imaginary(\zeta)}
\end{align*}
and 
\begin{align*}
d_{\Hc}(\zeta_1,\zeta_2) = \frac{1}{2} \arcosh \left( 1+ \frac{\abs{\zeta_1-\zeta_2}^2}{2\Imaginary(\zeta_1)\Imaginary(\zeta_2) } \right).
\end{align*}

 \subsection{Convex sets}
 
For convex sets there is a nice characterization of when $d_{\Omega}$ is a complete metric on $\Omega$.
 
 \begin{definition} An open convex set $\Omega \subset \Cb^d$ is called \emph{$\Cb$-proper} if $\Omega$ does not contain any complex affine lines. \end{definition}
 
The term ``proper'' is motivated by its use in the study of convex sets in real vector spaces (see for instance~\cite{B2005}). Other authors have used other language, for instance Frankel calls such sets \emph{affine hyperbolic}~\cite{F1989} but we use the word $\Cb$-proper to avoid confusion with the other meaning of hyperbolic.
 
\begin{proposition}\cite{B1980}
For an open convex set $\Omega$, the following are equivalent:
\begin{enumerate}
\item $d_{\Omega}$ is a complete metric on $\Omega$,
\item $\Omega$ is a $\Cb$-proper convex set.
\end{enumerate}
\end{proposition}

We also have a well known estimate of the Kobayashi metric on convex domains (see for instance~\cite[Theorem 4.1]{BP1994}, \cite[Theorem 5]{G1991}, or~\cite[Theorem 2.2]{F1991}):

\begin{lemma}\label{lem:convex_inf}
Suppose $\Omega \subset \Cb^d$ is an open convex set. If $p \in \Omega$ and $v \in \Cb^d$ is nonzero then 
\begin{align*}
\frac{\norm{v}}{2\delta_{\Omega}(p;v)} \leq K_{\Omega}(p;v).
\end{align*}
\end{lemma}

Since the proof is short we include it:

\begin{proof}
Let $L:= p + \Cb v$ and $\xi \in L \setminus \Omega \cap L$ such that $\norm{\xi-p} = \delta_{\Omega}(p;v)$. Let $H$ be a real hyperplane through $\xi$ which does not intersect $\Omega$. By rotating and translating we may assume $\xi=0$, $p=(p_1,0,\dots, 0)$, $H=\{ (z_1,\dots, z_d) \in \Cb^d : \Imaginary(z_1) =0\}$, and $\Omega \subset\{ (z_1,\dots,z_d) \in \Cb^d : \Imaginary(z_1) >0\}$. With this choice of normalization $v=(v_1,0,\dots, 0)$ for some $v_1 \in \Cb$.

Then if $P : \Cb^d \rightarrow \Cb$ is the projection onto the first component we have 
\begin{align*}
K_{\Omega}(p;v) \geq K_{P(\Omega)}(p_1; v_1) \geq K_{\Hc}(p_1; v_1) = \frac{\abs{v_1}}{2\Imaginary(p_1)} \geq \frac{\abs{v_1}}{2\abs{p_1}}.
\end{align*}
Since $\abs{p_1} = \norm{\xi-p} =  \delta_{\Omega}(p;v)$ and $\abs{v_1}=\norm{v}$ this completes the proof.
\end{proof}

We also have a global version of the above lemma:

\begin{lemma}\label{lem:convex_lower_bd_2}
Suppose $\Omega \subset \Cb^d$ is an open convex set and $p,q \in \Omega$. If $L$ is the complex line containing $p, q$ and $\xi \in L \setminus L \cap \Omega$ then 
\begin{align*}
 \frac{1}{2}\log \left( \frac{\norm{p-\xi}}{\norm{q-\xi}} \right) \leq d_{\Omega}(p,q).
 \end{align*}
 \end{lemma}
 
 \begin{remark} For our purposes the above estimate suffices, but the more precise estimate 
 \begin{align*}
  \frac{1}{2} \log \left(1+  \frac{\norm{p-q}}{\min\{ \delta_\Omega(p; q-p), \delta_\Omega(q; p-q)\} } \right) \leq d_{\Omega}(p,q)
  \end{align*}
  follows from the proof of Proposition 2 part (ii) in~\cite{NT2015}. 
 \end{remark}
  
 \begin{proof}
Since $p,q,\xi$ are all co-linear both sides of the desired inequality are invariant under affine transformations, in particular we can replace $\Omega$ by $A\Omega$ for some affine map $A$. Now let $H$ be a real hyperplane through $\xi$ which does not intersect $\Omega$. Using an affine transformation we may assume $\xi=0$, $p=(p_1,0,\dots, 0)$, $q=(q_1,0,\dots, 0)$, $H=\{ (z_1,\dots, z_d) \in \Cb^d : \Imaginary(z_1) =0\}$, and $\Omega \subset\{ (z_1,\dots, z_d) \in \Cb^d : \Imaginary(z_1) >0\}$. Then if $P : \Cb^d \rightarrow \Cb$ is the projection onto the first coordinate then we have 
 \begin{align*}
 d_{\Omega}(p,q) 
  & \geq d_{P(\Omega)}(p_1, q_1) \geq d_{\Hc}(p_1,q_1) 
 = \frac{1}{2} \arcosh \left( 1+ \frac{ \abs{p_1-q_1}^2 }{2\Imaginary(p_1)\Imaginary(q_1)} \right)\\
 & \geq  \frac{1}{2} \arcosh \left( 1+ \frac{ (\abs{p_1}-\abs{q_1})^2  }{2\abs{p_1}\abs{q_1}} \right) =  \frac{1}{2} \arcosh \left( \frac{\abs{p_1}}{2\abs{q_1}} +\frac{\abs{q_1}}{2\abs{p_1}} \right) \\
 & = \frac{1}{2} \abs{\log \left( \frac{\abs{p_1}}{\abs{q_1}} \right) }.
 \end{align*}
 Since $\norm{p-\xi} = \abs{p_1}$ and $\norm{q-\xi}=\abs{q_1}$ the lemma follows.
 \end{proof}
 
\subsection{Complex Geodesics}

 \begin{definition}
Suppose $\Omega \subset \Cb^d$ and  $\varphi : \Delta \rightarrow \Omega$ is a holomorphic map. If 
\begin{align*}
d_{\Delta}(p,q) = d_{\Omega}(\varphi(p),\varphi(q))
\end{align*}
for all $p,q\in \Delta$ then $\varphi$ is called a \emph{complex geodesic}.
\end{definition}

For bounded convex sets, a fundamental result of Royden and Wong~\cite{RW1983} states that every two points are contained in a complex geodesic. This result was recently generalized to $\Cb$-proper convex sets by Bracci and Saracco:

\begin{proposition}\label{prop:connecting}\cite[Lemma 3.3]{BS2009}
Suppose $\Omega \subset \Cb^d$ is a $\Cb$-proper open convex set. Then every two points are contained in a complex geodesic. 
\end{proposition}

\subsection{Finite type}

Given a function $f: \Cb \rightarrow \Rb$ with $f(0)=0$ let $\nu(f)$ denote the order of vanishing of $f$ at $0$. Suppose that $\Omega = \{ z \in \Cb^d : r(z) < 0\}$ where $r$ is a $C^\infty$ function with $\nabla r \neq 0$ near $\partial \Omega$. We say that a point $x \in \partial \Omega$ has \emph{finite line type $L$} if 
\begin{align*}
\sup \{ \nu( r \circ \ell) | \ell : \Cb \rightarrow \Cb^d \text{ is a non-trivial affine map and $\ell(0)=x$} \} = L.
\end{align*}
Notice that $\nu(r\circ \ell) \geq 2$ if and only if $\ell(\Cb)$ is tangent to $\Omega$. McNeal~\cite{M1992} proved that if $\Omega$ is convex  then $x \in \partial \Omega$ has finite line type if and only if it has finite type in the sense of D'Angelo (see also~\cite{BS1992}). In this paper, we say a convex domain $\Omega$ with $C^\infty$ boundary has \emph{finite line type $L$} if the line type of all $x \in \partial \Omega$ is at most $L$ and this bound is realized at some boundary point.

\subsection{Gromov hyperbolic metric spaces}\label{sec:prelim_gromov}

Suppose $(X,d)$ is a metric space. A curve $\sigma: [a,b] \rightarrow X$ is a \emph{geodesic} if $d(\sigma(t_1),\sigma(t_2)) = \abs{t_1-t_2}$ for all $t_1, t_2 \in [a,b]$.  A \emph{geodesic triangle} in a metric space is a choice of three points in $X$ and geodesic segments  connecting these points. A geodesic triangle is said to be \emph{$\delta$-thin} if any point on any of the sides of the triangle is within distance $\delta$ of the other two sides. 

\begin{definition}
A proper geodesic metric space $(X,d)$ is called \emph{$\delta$-hyperbolic} if every geodesic triangle is $\delta$-thin. If $(X,d)$ is $\delta$-hyperbolic for some $\delta\geq0$ then $(X,d)$ is called \emph{Gromov hyperbolic}.
\end{definition}

Bridson and Haefliger's book~\cite{BH1999} is one of the standard references for Gromov hyperbolic metric spaces. 

In this paper we will also use an equivalent formulation of Gromov hyperbolicity. Given $o,y,z \in X$ the \emph{Gromov product} is 
\begin{align*}
(x|y)_o = \frac{1}{2}(d(o,x)+d(o,y)-d(x,y)).
\end{align*}
Using the Gromov product it is possible to give an alternative definition of Gromov hyperbolicity (for a proof see~\cite[Proposition 2.1.2, Proposition 2.1.3]{BS2007}):

\begin{theorem}
A proper geodesic metric space $(X,d)$ is Gromov hyperbolic if and only if there exists $\delta \geq 0$ such that 
\begin{align*}
(x|y)_o \geq \min \{ (x|z)_o, (z|y)_o\} - \delta
\end{align*}
for all $o,x,y,z \in X$.
\end{theorem}

A curve $\sigma: [a,b] \rightarrow X$ is an $(A,B)$-\emph{quasi-geodesic} if 
\begin{align*}
\frac{1}{A}\abs{t_1-t_2}-B \leq d(\sigma(t_1),\sigma(t_2)) \leq A \abs{t_1-t_2}+B
\end{align*}
for all $t_1, t_2 \in [a,b]$. An important property of $\delta$-hyperbolic spaces is that every quasi-geodesic is close to an actual geodesic (see for instance~\cite[Theorem 1.3.2]{BS2007}) which implies:

\begin{proposition}
\label{prop:thin_tri}
For $A\geq 1$, $B\geq 0$, and $\delta \geq 0$ there exists $M>0$ such that if $(X,d)$ is $\delta$-hyperbolic then every $(A,B)$-quasi-geodesic triangle is $M$-thin.
\end{proposition}

A proper geodesic Gromov hyperbolic metric space $(X,d)$ also has a natural boundary $X(\infty)$ called the Gromov boundary. Two geodesic rays $\sigma_1,\sigma_2:[0,\infty) \rightarrow X$ are said to \emph{asymptotic} if 
\begin{align*}
\sup_{t \geq 0} d(\sigma_1(t),\sigma_2(t)) < \infty.
\end{align*}
Let the \emph{Gromov boundary} $X(\infty)$ be the set of the equivalence classes of asymptotic geodesic rays in $X$. 

The set $\overline{X} = X \cup X(\infty)$ has a natural topology making it a compactification of $X$ (see for instance~\cite[Chapter III.H.3]{BH1999}). To understand this topology we introduce some additional notation: given a geodesic ray $\sigma:[0,\infty) \rightarrow X$ let $\sigma(\infty)$ denote the equivalence class of $\sigma$ and given a geodesic segment $\sigma:[0,R] \rightarrow X$ let $\sigma(\infty)$ denote the point $\sigma(R)$. Now $\overline{X} = X \cup X(\infty)$ has a topology where $\xi_n \rightarrow \xi$ if and only if for every choice of geodesics $\sigma_n$ with $\sigma_n(0)=o$ and $\sigma_n(\infty)=\xi_n$ every subsequence of $\{\sigma_n\}$ has a subsequence which converges locally uniformly to a geodesic $\sigma$ with $\sigma(\infty)=\xi$.

\part{Necessary conditions}

\section{Holomorphic disks in the boundary}\label{sec:no_disks}

In this section we prove Theorem~\ref{thm:no_affine_disk} whose statement we recall:

\begin{theorem}\label{thm:no_affine_disk_body}
Suppose $\Omega$ is a $\Cb$-proper open convex set. If $(\Omega, d_{\Omega})$ is Gromov hyperbolic then $\partial \Omega$ does not contain any non-trivial holomorphic disks.
\end{theorem}

We will construct ``fat'' quasi-geodesic triangles when $\partial \Omega$ contains a non-trivial holomorphic disk. The first step is to construct quasi-geodesics.

\begin{lemma}\label{lem:quasi_geodesic}
Suppose $\Omega \subset \Cb^d$ is an open convex set, $p \in \Omega$, and $x \in \partial \Omega$ such that 
\begin{align*}
\delta_{\Omega}(p; \vec{px}) \geq \epsilon\norm{x-p}
\end{align*}
for some $\epsilon >0$. If 
\begin{align*}
x_t = x+e^{-2t}(p-x)
\end{align*}
then 
\begin{align*}
\abs{t_1-t_2} \leq d_{\Omega}(x_{t_1}, x_{t_2}) \leq 2\epsilon^{-1}\abs{t_1-t_2}
\end{align*}
for all $t_1,t_2 \geq 0$. In particular, the line segment $[p,x)$ can be parametrized to be an $\left(2\epsilon^{-1},0\right)$-quasi-geodesic in $(\Omega, d_{\Omega})$.
\end{lemma}

\begin{proof}[Proof of Lemma~\ref{lem:quasi_geodesic}]

Using a complex affine transformation, we can assume 
\begin{enumerate}
\item $x=0$,
\item $p=\left(e^{i\theta},0,\dots,0\right)$ for some $\theta \in \Rb$,
\item $H = \left\{\vec{z} \in \Cb^d : \Imaginary(z_1)=0\right\}$ is a supporting hyperplane of $\Omega$ at $0$,
\item $\Omega \subset \left\{ \vec{z} \in \Cb^d : \Imaginary(z_1) >0\right\}$.
\end{enumerate}

With respect to this choice of coordinates our parametrization of $[p,x)$ is given by
\begin{align*}
x_t = (e^{-2t}e^{i\theta},0 \dots, 0)=e^{-2t}p.
\end{align*} 
By Lemma~\ref{lem:convex_lower_bd_2} we have 
\begin{align*}
d_{\Omega}(x_{t_1}, x_{t_2}) \geq \frac{1}{2} \abs{ \log \frac{\norm{x_{t_1}-0}}{\norm{x_{t_2}-0}}} = \abs{t_1-t_2}.
\end{align*}
To see the upper bound, first let $L$ be the complex line $L = \{ (z,0,\dots, 0) : z \in \Cb\}$. Now in these coordinates 
\begin{align*}
\delta_{\Omega \cap L}(p) = \delta_{\Omega}(p; \vec{px}) \geq \epsilon \norm{x-p} = \epsilon
\end{align*}
so $B_{\epsilon}(p) \subset \Omega \cap L$. Since $\Omega$ is convex and $0 \in \partial \Omega$, $\Omega \cap L$ contains the interior of the convex hull of $B_{\epsilon}(p)$ and $0$. So for $\lambda \in (0,1)$ the set $B_{\lambda \epsilon}(\lambda p)$ is in $\Omega \cap L$. Thus 
\begin{align*}
\delta_{\Omega \cap L}(x_t) \geq \epsilon e^{-2t}.
\end{align*}
Then, by Proposition~\ref{prop:basic_kob},
\begin{align*}
K_{\Omega}(x_t; \dot{x}_t) \leq \frac{\norm{\dot{x}_t}}{\delta_{\Omega}(x_t ; \dot{x}_t)} = 2\frac{e^{-2t}}{\delta_{\Omega \cap L}(x_t)} \leq \frac{2}{\epsilon}.
\end{align*}
So for $t_1 < t_2$ we have
\begin{equation*}
d_{\Omega}(x_{t_1},x_{t_2}) \leq \int_{t_1}^{t_2} K_{\Omega}(x_t ; \dot{x}_t) dt \leq  \frac{2}{\epsilon}\abs{t_1-t_2}. \qedhere
\end{equation*}

\end{proof}  

It will be helpful to know that the boundary contains a holomorphic disk if and only if it contains a complex affine disk. 

\begin{lemma}\label{lem:affine_disk}\cite{FS1998}
Suppose $\Omega$ is a convex open set and $\varphi: \Delta \rightarrow \Cb^{d}$ is a non-trivial holomorphic map with $\varphi(\Delta) \subset \partial \Omega$. Then there exists a complex line $L$ such that $L \cap \partial \Omega$ is open in $L$. 
\end{lemma}

\begin{remark} Fu and Straube~\cite{FS1998} actually proved that the boundary of a convex set contains a holomorphic variety of dimension $q$ if and only if it contains an affine ball of dimension $q$. \end{remark}

We can now prove the theorem:

\begin{proof}[Proof of Theorem~\ref{thm:no_affine_disk_body}]
Suppose that $\partial \Omega$ contains a non-trivial holomorphic disk, we will show that $(\Omega, d_{\Omega})$ is not Gromov hyperbolic. Using Lemma~\ref{lem:affine_disk} there exists a complex line $L$ such that $L \cap \partial \Omega$ contains an open set in $L$. Let $\Oc \subset L \cap \partial \Omega$ be the interior of $L \cap \partial \Omega$ in $L$. Fix a point $x \in \Oc$. Since $\Omega$ is $\Cb$-proper $\Oc \neq L$ and so there exists some $y \in \partial \Oc$. Finally fix a point $o \in \Omega$. 

By Lemma~\ref{lem:quasi_geodesic} there exists $A>1$ and parametrizations $x_t$ of $[o,x)$ and $y_t$ of $[o,y)$ making them $(A,0)$-quasi-geodesics. Recall that 
\begin{align*}
x_t = x+e^{-2t}(o-x) \text{ and } y_t = y+e^{-2t}(o-y).
\end{align*}

\noindent \textbf{Claim 1:} After possibly increasing $A$, there exists $T_0 \geq 0$ such that for all $T>T_0$ the line segments $[x_T, y_T]$ can be parametrized to be a $(A,0)$-quasi-geodesic in $(\Omega, d_{\Omega})$.

\begin{proof}[Proof of Claim 1] Since $o \in \Omega$ there exists $\delta_1 >0$ such that $B_{\delta_1}(o) \subset \Omega$ and since $x \in \Oc$ there exists $\delta_2 >0$ such that $B_{\delta_2}(x) \cap L \subset \partial \Omega$. Then, because $y_T-x_T = (1-e^{-2T})(y-x) \in L$ and $\Omega$ is convex, we have that 
\begin{align}
\label{eq:qg1}
\delta_{\Omega}(x_T; y_T-x_T) \geq \min\{\delta_1, \delta_2\}
\end{align}
for all $T>0$. 

Now let $\{ a_{T}, b_{T} \} = \partial \Omega \cap \overline{x_T y_T}$ with the ordering $a_{T}, x_T, y_T, b_{T}$ along the line $\overline{x_T y_T}$. Since $y \in \partial \Oc$, we must have that $b_{T}\rightarrow y$ as $T \rightarrow \infty$. In particular, there exists $T_0 > 0$ such that 
\begin{align}
\label{eq:qg2}
\sup_{T \geq T_0} \norm{x_T-b_T} < \infty.
\end{align}

Then using Equation~\ref{eq:qg1}, Equation~\ref{eq:qg2}, and Lemma~\ref{lem:quasi_geodesic} we may assume that $[x_T,y_T]$ is a $(A,0)$-quasi-geodesic for all $T\geq T_0$. \end{proof}

\noindent \textbf{Claim 2:} $\lim_{t \rightarrow \infty} d_{\Omega}( x_t, [o,y))= \infty$. 

\begin{proof}[Proof of Claim 2]Suppose not, then there exists $R >0$ and a sequence $t_n \rightarrow \infty$ and $s_n >0$ such that
\begin{align*}
d_{\Omega}(x_{t_n}, y_{s_n}) < R.
\end{align*}
Now
\begin{align*}
d_{\Omega}(x_{t_n}, y_{s_n}) \geq d_{\Omega}(x_{t_n}, o) - d_{\Omega}(o, y_{s_n}) \geq \frac{t_n}{A} - A s_n
\end{align*}
and so we must have that $s_n \rightarrow \infty$. 
Let $\{ a_n, b_n \} = \partial \Omega \cap \overline{x_{t_n}y_{s_n}}$ with the ordering $a_n, x_{t_n}, y_{s_n}, b_n$ along the line $\overline{x_{t_n}y_{s_n}}$. Since $y \in \partial \Oc$ we see that $b_n \rightarrow y$.  Now by Lemma~\ref{lem:convex_lower_bd_2}
\begin{align*}
d_{\Omega}(x_{t_n}, y_{s_n}) \geq \frac{1}{2} \log \frac{\norm{x_{t_n}-b_n}}{\norm{y_{s_n}-b_n}}.
\end{align*}
Since $b_n \rightarrow y$, $x_{t_n}\rightarrow x$, and $y_{s_n} \rightarrow y$ we then have
\begin{align*}
\lim_{n \rightarrow \infty} d_{\Omega}(x_{t_n}, y_{s_n}) = \infty
\end{align*}
which is a contradiction. 
\end{proof}

Now by Proposition~\ref{prop:thin_tri}, if $(\Omega,d_{\Omega})$ is $\delta$-hyperbolic then there exists $M>0$ such that any $(A,0)$-quasi-geodesic triangle is $M$-thin. Thus the proposition will follow from the next claim: \newline

\noindent  \textbf{Claim 3:} For any $M>0$ there exists $T>0$ such that $[o,x_T], [x_T,y_T], [y_T,o]$ is not $M$-thin. 

\begin{proof}[Proof of Claim 3] By Claim 2 there exists $t_0>0$ such that $d_{\Omega}(x_{t_0}, [o,y)) > M$. Next, Lemma~\ref{lem:convex_lower_bd_2} implies that 
\begin{align*}
\lim_{T \rightarrow \infty} d_{\Omega}(x_{t_0}, [x_T, y_T])=\infty.
\end{align*}
Thus there exists $T > t_0$ such that 
\begin{align*}
d_{\Omega}(x_{t_0}, [x_T ,y_T] \cup [y_T, o] ) > M.
\end{align*}
So the $(A,0)$-quasi-geodesic triangle $[o,x_T], [x_T,y_T], [y_T,o]$ is not $M$-thin.\end{proof}

This completes the proof.
\end{proof}

\begin{remark}
Lemma~\ref{lem:quasi_geodesic} can also be used to provide a different proof of Theorem~\ref{thm:no_affine_disk_body}. It is well known that for any $\delta \geq 0$, $A\geq 1$, and $B\geq 0$ there exists $M>0$ such that whenever $(X,d)$ is a $\delta$-hyperbolic metric space and $\sigma_1,\sigma_2 : [0,\infty) \rightarrow X$ are $(A,B)$-quasi-geodesics with $\sigma_1(0)=\sigma_2(0)$ then either 
\begin{align*}
\sup_{t \geq 0} d_{\Omega}(\sigma_1(t), \sigma_2) < M
\end{align*}
or
 \begin{align*}
\sup_{t \geq 0} d_{\Omega}(\sigma_1(t), \sigma_2) =\infty.
\end{align*}
One way to prove this assertion is to use the geodesic shadowing property~\cite[Chapter III.H, Theorem 1.7]{BH1999} to reduce to the case when $\sigma_1$ and $\sigma_2$ are geodesic rays and then use the exponential divergence of geodesic rays~\cite[Chapter III.H, Proposition 1.25]{BH1999}.

Then let $\Oc$, $x \in \Oc$, and $y \in \partial \Oc$ be as in the proof of Theorem~\ref{thm:no_affine_disk}. Next let $y_n \in \Oc$ be a  sequence such that $y_n \rightarrow y$. Finally let $\sigma:[0,\infty) \rightarrow \Omega$ be a quasi-geodesic parameterizing $[o,x)$ and for each $n>0$ let $\sigma_n:[0,\infty) \rightarrow \Omega$ be a quasi-geodesic parameterizing $[0,y_n)$. Then one can show that 
\begin{align*}
\sup_{t \geq 0} d_{\Omega}(\sigma_n(t), \sigma) < \infty
\end{align*}
for all $n >0$ but 
\begin{align*}
\lim_{n \rightarrow \infty} \sup_{t \geq 0} d_{\Omega}(\sigma_n(t), \sigma) = \infty.
\end{align*}
Thus $(\Omega, d_{\Omega})$ is not Gromov hyperbolic. This is similar to the argument that Ivanov~\cite{I2002} gave showing that Teichm\"uller space endowed with the Teichm\"uller metric is not Gromov hyperbolic. 
\end{remark}

\section{Local Hausdorff topology and the Kobayashi metric}\label{sec:local_haus}

Given a set $A \subset \Cb^d$, let $\Nc_{\epsilon}(A)$ denote the \emph{$\epsilon$-neighborhood of $A$} with respect to the Euclidean distance. The \emph{Hausdorff distance} between two compact sets $A,B$ is given by
\begin{align*}
d_{H}(A,B) = \inf \left\{ \epsilon >0 : A \subset \Nc_{\epsilon}(B) \text{ and } B \subset \Nc_{\epsilon}(A) \right\}.
\end{align*}
Equivalently, 
\begin{align*}
d_H(A,B) = \max\left\{\sup_{a \in A}\inf_{b \in B} \norm{a-b}, \sup_{b \in B} \inf_{a \in A} \norm{a-b} \right\}.
\end{align*}
The Hausdorff distance is a complete metric on the space of compact sets in $\Cb^d$.

The space of all closed convex sets in $\Cb^d$ can be given a topology from the local Hausdorff semi-norms. For $R >0$ and a set $A \subset \Cb^d$ let $A^{(R)} := A \cap B_R(0)$. Then define the \emph{local Hausdorff semi-norms} by
\begin{align*}
d_H^{(R)}(A,B) := d_H(A^{(R)}, B^{(R)}).
\end{align*}
Since an open convex set is completely determined by its closure, we say a sequence of open convex sets $A_n$ converges in the local Hausdorff topology to an open convex set $A$ if $d_H^{(R)}(\overline{A}_n,\overline{A}) \rightarrow 0$ for all $R>0$. 

We now show that the Kobayashi metric is continuous with respect to the local Hausdorff topology.

\begin{theorem}
\label{thm:dist_conv}
Suppose $\Omega_n$ is a sequence of $\Cb$-proper convex open sets converging to a $\Cb$-proper convex open set $\Omega$ in the local Hausdorff topology. Then 
\begin{align*}
d_{\Omega}(x,y) = \lim_{n \rightarrow \infty} d_{\Omega_n}(x,y)
\end{align*}
for all $x,y \in \Omega$ uniformly on compact sets of $\Omega \times \Omega$.
\end{theorem}

As an application of Theorem~\ref{thm:dist_conv} we will establish the following normal family result:

\begin{proposition}\label{prop:normal_family}
Suppose $\Omega_n$ is a sequence of $\Cb$-proper convex open sets converging to a $\Cb$-proper convex  open set $\Omega$ in the local Hausdorff topology. If $\varphi_n: \Delta \rightarrow \Omega_n$ is a sequence of holomorphic maps then either
\begin{enumerate}
\item $\varphi_n(x) \rightarrow \infty$ for all $x \in \Delta$ or
\item there exists a subsequence which converges uniformly on compact sets to a holomorphic map $\varphi: \Delta \rightarrow \overline{\Omega}$. Moreover, either $\varphi(\Delta) \subset \partial \Omega$ or $\varphi(\Delta) \subset \Omega$.
\end{enumerate}
\end{proposition}

The proof of Theorem~\ref{thm:dist_conv} will require a series of lemmas.

\begin{lemma}
\label{lem:1}
For any $\epsilon >0$ there exists a $\delta >0$ such that whenever $\Omega_1$ and $\Omega_2$ are bounded open convex sets in $\Cb^d$, $B_{\epsilon}(p) \subset \Omega_1$, and $d_H(\overline{\Omega}_1,\overline{\Omega}_2) < \delta$  then $p \in \Omega_2$.
\end{lemma}

\begin{proof}
Let $e_1,\dots, e_d$ be the standard complex basis of $\Cb^d$. For $1 \leq i \leq d$ let
\begin{align*}
v_{4i-3} &=p-\epsilon e_i, \quad v_{4i-2}=p+\epsilon e_i, \\
 v_{4i-1}&=p-\epsilon i e_i, \quad v_{4i}=p+\epsilon ie_i.
\end{align*}
 Then the convex hull of $v_1,\dots, v_{4d}$ contains $p$ in its interior. Moreover there exists $\delta >0$ such that if $v_1^\prime, \dots, v_{4d}^\prime$ are points with $\norm{v_i-v_i^\prime}<\delta$ for all $1 \leq i \leq 4d$ then the convex hull of $v_1^\prime, \dots, v_{4d}^\prime$ contains $p$ in its interior. If $d_H(\overline{\Omega}_1,\overline{\Omega}_2) < \delta$ then $\overline{\Omega}_1\subset \Nc_{\delta}(\overline{\Omega}_2)$ and hence $\overline{\Omega}_2$ contains such points. 
\end{proof}

\begin{lemma}\label{lem:lem2}
Suppose $\Omega_n$ is a sequence of $\Cb$-proper convex open sets converging to a $\Cb$-proper convex open set $\Omega$ in the local Hausdorff topology. If $K$ is a compact subset of $\Omega$ then there exists $N$ such that  $K \subset \Omega_n$ for all $n >N$.
\end{lemma}

\begin{proof}
There exists $R>0$ such that $K$ is a compact subset of $\Omega^{(R)} := \Omega \cap B_R(0)$. Since $\Omega_n \rightarrow \Omega$ we have that $d_H(\overline{\Omega}^{(R)}, \overline{\Omega}_n^{(R)}) \rightarrow 0$. Moreover since $K$ is a compact subset of $\Omega^{(R)}$ there exists $\epsilon>0$ such that $B_{\epsilon}(p) \subset \Omega^{(R)}$ for all $p \in K$. Then the lemma follows from Lemma~\ref{lem:1}.
\end{proof}

\begin{lemma}
\label{lem:upper}
Suppose $\Omega_n$ is a sequence of $\Cb$-proper convex open sets converging to a $\Cb$-proper convex open set $\Omega$ in the local Hausdorff topology. If $K \subset \Omega$ is compact and $\epsilon >0$ then there exists $N$ such that 
\begin{align*}
d_{\Omega_n}(p,q) \leq (1+\epsilon)d_{\Omega}(p,q)
\end{align*}
for all $n >N$ and all $p,q \in K$.
\end{lemma}

\begin{proof}
Since $K$ is compact there exists $R>0$ such that 
\begin{align*}
 d_{\Omega}(p,q) < R
\end{align*}
for all $p,q \in K$. Fix $\delta <1$ such that
\begin{align*}
d_{\Delta}(0,\zeta/\delta) \leq (1+\epsilon)d_{\Delta}(0,\zeta)
\end{align*}
for all $\zeta \in \Delta$ with $d_{\Delta}(0,\zeta) \leq R$.

Now let 
\begin{align*}
K^\prime = \{ p \in \Omega : d_{\Omega}(p,K) \leq d_{\Delta}(\delta, 0)\}.
\end{align*}
Then $K^\prime$ is a compact subset of $\Omega$ and hence there exists $N$ such that $K^\prime \subset \Omega_n$ for all $n>N$. 

Now fix $p,q \in K$ and let $\varphi: \Delta \rightarrow \Omega$ be a complex geodesic with $\varphi(0)=p$ and $\varphi(\zeta_0)=q$ for some $\zeta_0 \in \Delta$. Notice that  $\varphi(B_{\delta}(0)) \subset K^\prime$ since $\varphi(0) \in K$ and 
\begin{align*}
\sup_{\zeta \in B_{\delta}(0)} d_{\Omega}(\varphi(\zeta), \varphi(0))=\sup_{\zeta \in B_{\delta}(0)} d_{\Delta}(\zeta, 0) = d_{\Delta}(\delta, 0).
\end{align*}
In particular if $\varphi_\delta: \Delta \rightarrow \Cb^d$ is defined by $\varphi_{\delta}(z) = \varphi(\delta z)$ then 
\begin{align*}
\varphi_\delta(\Delta)=\varphi(B_{\delta}(0)) \subset K^\prime \subset \Omega_n
\end{align*}
for $n >N$. Then
\begin{align*}
 d_{\Omega_n}(p,q) 
 &= d_{\Omega_n}(\varphi_{\delta}(0), \varphi_{\delta}(\zeta_0/\delta)) \leq d_{\Delta}(0,\zeta_0/\delta)\\
 & \leq  (1+\epsilon)d_{\Delta}(0,\zeta_0) = (1+\epsilon)d_{\Omega}(p,q).
\end{align*}
Since $p,q$ were arbitrary points in $K$ the lemma follows.
\end{proof}

\begin{lemma}\label{lem:containment}
Suppose $\Omega_n$ is a sequence of $\Cb$-proper convex open sets converging to a $\Cb$-proper  convex open set $\Omega$ in the local Hausdorff topology. If $K \subset \Omega$ is compact and $\delta < 1$ then there exists $N >0$ so that if $\varphi : \Delta \rightarrow \Omega_n$ is a holomorphic map with $\varphi(0) \in K$ then $\varphi(B_{\delta}(0)) \subset \Omega$.
\end{lemma}

\begin{proof}
Suppose not, then by passing to a subsequence for each $n$ there exists a holomorphic map $\varphi_n:\Delta \rightarrow \Omega_n$ with $\varphi_n(0) \in K$ and $\varphi_n(B_{\delta}(0)) \not\subset \Omega$. Next let $\zeta_n^\prime \in B_{\delta}(0)$ be such that $\varphi_n(\zeta_n^\prime) \notin \Omega$. Now pick $t_n \in (0,1)$ so that $\varphi_n(t_n \zeta^\prime) \in \Omega$ but $\delta_{\Omega}(\varphi_n(t_n \zeta^\prime)) \rightarrow 0$. g let $\zeta_n = t_n \zeta_n^\prime$ and $z_n = \varphi_n(\zeta_n)$. 

By passing to another subsequence we can suppose that $K \subset \Omega_n$ for all $n$. Then by Lemma~\ref{lem:upper} if we fix a point $o \in K$ the quantity
\begin{align*}
R = \sup\{ d_{\Omega_n}(k,o) : k \in K, n \in \Nb\}
\end{align*}
is finite. 
Since $\zeta_n \in B_{\delta}(0)$ we see that
\begin{align*}
d_{\Omega_n}(z_n, o) \leq d_{\Omega_n}(z_n, \varphi_n(0))+R \leq d_{\Delta}(\zeta_n,0) + R \leq d_{\Delta}(0,\delta)+R.
\end{align*}

Now let $L_n$ be the complex line containing $o$ and $z_n$. By passing to a subsequence we can suppose that the sequence $L_n$ converges to a complex line $L$. Since $\Omega$ is $\Cb$-proper there exists $\xi \in L \setminus L \cap \Omega$. Since $\Omega_n$ converges to $\Omega$ in the local Hausdorff topology there exists $\xi_n \in L_n \setminus L_n \cap \Omega_n$ such that $\xi_n \rightarrow \xi$. By passing to a subsequence we can suppose that $\norm{\xi_n-\xi} < 1$ for all $n$.

Then by Lemma~\ref{lem:convex_lower_bd_2}
\begin{align*}
d_{\Omega_n}(o,z_n)
& \geq \frac{1}{2} \log \frac{ \norm{z_n-\xi_n}}{\norm{o-\xi_n}} \geq \frac{1}{2}\log \frac{ \norm{z_n-\xi}-\norm{\xi-\xi_n}}{\norm{o-\xi}+\norm{\xi-\xi_n}} \\
& \geq \frac{1}{2}\log \frac{ \norm{z_n-\xi}-1}{\norm{o-\xi}+1}.
\end{align*}
So $z_n$ must be a bounded sequence. Since $\Omega_n \rightarrow \Omega$ in the local Hausdorff topology and $\delta_{\Omega}(z_n) \rightarrow 0$ this implies that
\begin{align*}
\delta_{\Omega_n}(z_n) \rightarrow 0.
\end{align*}
Since $\Omega_n$ is convex this in turn implies that
\begin{align*}
\delta_{\Omega_n}(z_n; \vec{z_no}) \rightarrow 0.
\end{align*}
Since $o \in \Omega$ there exists $\epsilon >0$ such that $B_{\epsilon}(o) \subset \Omega_n$ for $n$ large enough. But then by Lemma~\ref{lem:convex_lower_bd_2}
\begin{align*}
d_{\Omega_n}(o,z_n) \geq \frac{1}{2} \log \frac{ \epsilon }{\delta_{\Omega_n}(z_n; \vec{z_no}) }
\end{align*}
which contradicts the fact that $d_{\Omega_n}(o,z_n)$ is bounded. 
\end{proof}

\begin{lemma}\label{lem:lower}
Suppose $\Omega_n$ is a sequence of $\Cb$-proper convex open sets converging to a $\Cb$-proper convex open set $\Omega$ in the local Hausdorff topology. If $K \subset \Omega$ is compact and $\epsilon >0$ then there exists $N$ such that 
\begin{align*}
d_{\Omega}(p,q) \leq (1+\epsilon)d_{\Omega_n}(p,q)
\end{align*}
for all $n >N$ and all $p,q \in K$.
\end{lemma}

\begin{proof} Since $K$ is compact there exists $R >0$ such that 
\begin{align*}
d_{\Omega}(p,q) < R
\end{align*}
for all $p,q \in K$. By Lemma~\ref{lem:upper} we can pick $N^\prime >0$ such that 
\begin{align*}
d_{\Omega_n}(p,q) \leq (1+\epsilon)d_{\Omega}(p,q) < (1+\epsilon)R
\end{align*}
for all $n > N^\prime$ and all $p,q \in K$. Fix $\delta <1$ such that
\begin{align*}
d_{\Delta}(0,\zeta/\delta) \leq (1+\epsilon)d_{\Delta}(0,\zeta)
\end{align*}
for all $\zeta \in \Delta$ with $d_{\Delta}(0,\zeta) \leq (1+\epsilon)R$.

By the last lemma there exists $N \geq N^\prime$ such that for all $n > N$ and every holomorphic map $\varphi : \Delta \rightarrow \Omega_n$ with $\varphi(0) \in K$ we have 
\begin{align*}
\varphi(B_{\delta}(0)) \subset \Omega.
\end{align*}
Now suppose $n > N$ and $p,q \in K$ then there exists a complex geodesic $\varphi:\Delta \rightarrow \Omega_n$ with $\varphi(p)=0$ and $\varphi(\zeta_0) =q$ for some $\zeta_0 \in \Delta$. Since $d_{\Omega_n}(p,q) \leq (1+\epsilon)R$ we see that $d_{\Delta}(0,\zeta_0) \leq (1+\epsilon)R$. By construction the map $\varphi_{\delta}: \Delta \rightarrow \Cb^d$ given by $\varphi_{\delta}(\zeta) = \varphi(\delta \zeta)$ has image in $\Omega$. Moreover,
\begin{align*}
d_{\Omega}(p,q) 
&= d_{\Omega}( \varphi_{\delta}(0), \varphi_{\delta}(\zeta_0/\delta)) \leq d_{\Delta}(0, \zeta_0/\delta) \\
& \leq (1+\epsilon)d_{\Delta}(0,\zeta_0) = (1+\epsilon)d_{\Omega_n}(p,q). 
\end{align*}
Since $p,q \in K$ and $n >N$ were arbitrary the lemma follows. 
\end{proof}

\begin{proof}[Proof of Theorem~\ref{thm:dist_conv}] 
This is just Lemma~\ref{lem:upper} and Lemma~\ref{lem:lower}.
\end{proof}

\begin{proof}[Proof of Proposition~\ref{prop:normal_family}]
Suppose that case one does not hold, that is by passing to a subsequence there exists $x \in \Delta$ such that $\varphi_n(x) \rightarrow y \in \Cb^d$. By reparametrizing $\Delta$ we may assume $x=0$. 

Now fix $\delta < 1$. Using an argument similar to the proof of Lemma~\ref{lem:containment} we see that there exists an $R>0$ such that $\varphi_n(B_{\delta}(0)) \subset B_R(0)$. Then using Proposition~\ref{prop:taut}, there exists a subsequence such that $\varphi_n$ converges locally uniformly on $B_{\delta}(0)$. Now since $\delta < 1$ was arbitrary, a diagonal argument implies that there exists a subsequence such that $\varphi_n$ converges locally uniformly to a holomorphic map $\varphi : \Delta \rightarrow \Cb^d$. Since $\varphi_n(\Delta) \subset \Omega_n$ for all $n$ and $\Omega_n \rightarrow \Omega$ we see that $\varphi(\Delta) \subset \overline{\Omega}$.

Now suppose that $\varphi(\Delta) \cap \Omega \neq \emptyset$. Let $\zeta_0 \in \Delta$ be such that $\varphi(\zeta_0) \in \Omega$. Now using the fact that 
\begin{align*}
d_{\Omega_n}(p,q) \rightarrow d_{\Omega}(p,q)
\end{align*}
uniformly on compact subsets of $\Omega$ we see that $\varphi(\Delta) \subset \Omega$.
\end{proof}

\section{Orbits of convex sets}\label{sec:nec}

Recall that $\Xb_d$ is the set of $\Cb$-proper convex open sets in $\Cb^d$.

\begin{theorem}\label{thm:nec}
Suppose $\Omega$ is $\Cb$-proper convex open set and $(\Omega, d_{\Omega})$ is Gromov hyperbolic. Then $(\wh{\Omega}, d_{\wh{\Omega}})$ is Gromov hyperbolic whenever $\wh{\Omega} \in \overline{\Aff(\Cb^d) \Omega} \cap \Xb_d$.
\end{theorem}

\begin{remark}
The proof will show that there exists a $\delta >0$ such that for any $\wh{\Omega} \in \overline{\Aff(\Cb^d) \Omega} \cap \Xb_d$ and any four points $p,x,y,z \in \wh{\Omega}$ we have 
\begin{align*}
\min \{ (x,y)_p, (y,z)_{p}\}-(x,z)_p \geq \delta.
\end{align*}
This implies that there exists a $\delta_1 >0$ such that any geodesic triangle in any $\wh{\Omega} \in \overline{\Aff(\Cb^d) \Omega} \cap \Xb_d$ is $\delta_1$-thin. 
\end{remark}

\begin{proof}[Proof of Theorem~\ref{thm:nec}]
Suppose $\Omega_n:=A_n \Omega \rightarrow \wh{\Omega}$ in the local Hausdorff topology. Let $(\cdot,\cdot)_{\cdot}^{(n)}$ denote the Gromov product on $(\Omega_n, d_{\Omega_n})$ and $(\cdot | \cdot)_{\cdot}$ denote the Gromov product on $(\wh{\Omega}, d_{\wh{\Omega}})$. Now the affine map $A_n$ induces an isometry between $(\Omega, d_{\Omega})$ and $(\Omega_n, d_{\Omega_n})$. In particular there exists an $\delta >0$ such that for any $n$ and any four points $p, x,y,z \in \Omega_n$ we have
\begin{align*}
\min \{ (x,y)_p^{(n)}, (y,z)_{p}^{(n)}\}-(x,z)_p^{(n)} \geq \delta.
\end{align*}
Now suppose that $p, x,y,z \in \wh{\Omega}$ then by Theorem~\ref{thm:dist_conv}
\begin{align*}
\min \{ (x,y)_p, (y,z)_{p}\}-(x,z)_p = \lim_{n \rightarrow \infty} \min \{ (x,y)_p^{(n)}, (y,z)_{p}^{(n)}\}-(x,z)_p^{(n)} \geq \delta
\end{align*}
Thus $(\wh{\Omega}, d_{\wh{\Omega}})$ is Gromov hyperbolic.
\end{proof}

As a corollary to Theorem~\ref{thm:nec} and Theorem~\ref{thm:no_affine_disk_body} we obtain:

\begin{corollary}\label{cor:no_disks_orbit}
Suppose $\Omega$ is a $\Cb$-proper convex open set and $(\Omega, d_{\Omega})$ is Gromov hyperbolic. If $\wh{\Omega} \in \overline{\Aff(\Cb^d) \Omega} \cap \Xb_d$ then the boundary of $\wh{\Omega}$ has no non-trivial holomorphic disks.
\end{corollary}

\section{Infinite type boundary points}

In this section we prove Proposition~\ref{prop:infinite_type} whose statement we recall:

\begin{proposition}\label{prop:infinite_type_2}
Suppose $\Omega \subset \Cb^d$ is a $\Cb$-proper convex open set with $0 \in \partial \Omega$ and 
\begin{align*}
\Omega \cap \Oc = \{\vec{z} \in \Oc :  \Imaginary(z_1) > f(\Real(z_1), z_2, \dots, z_d)\}
\end{align*}
where $\Oc$ is a neighborhood of the origin and $f:\Rb \times \Cb^{d-1} \rightarrow \Rb$ is a convex non-negative function. If 
\begin{align*}
\lim_{z \rightarrow 0} \frac{f(0,z,0,\dots,0)}{\abs{z}^n} = 0
\end{align*}
for all $n>0$ then $(\Omega, d_{\Omega})$ is not Gromov hyperbolic.
\end{proposition}

We begin by recalling a result of Frankel:

\begin{lemma}\cite[Theorem 9.3]{F1991}\label{lem:slices}
Suppose $\Omega \subset \Cb^d$ is a $\Cb$-proper convex open set. If $V \subset \Cb^d$ is a complex affine subspace intersecting $\Omega$ and $A_n \in \Aff(V)$ is a sequence of affine maps such that $A_n(\Omega \cap V)$ converges in the local Hausdorff topology to a  $\Cb$-proper convex open set $\wh{\Omega}_V \subset V$, then there exists affine maps $B_n \in \Aff(\Cb^d)$ such that $B_n\Omega$ converges in the local Hausdorff topology to a  $\Cb$-proper convex open set $\wh{\Omega}$ with $\wh{\Omega} \cap V = \wh{\Omega}_V$. 
\end{lemma}

\begin{proof}[Proof of Proposition~\ref{prop:infinite_type_2}]
We claim that there exists affine maps $A_n \in \Aff(\Cb^d)$ such that the sequence $A_n \Omega$ converges in the local Hausdorff topology to a $\Cb$-proper convex open set $\wh{\Omega}$ where $\partial \wh{\Omega}$ contains a non-trivial complex affine disk. By Corollary~\ref{cor:no_disks_orbit} this will imply that $(\Omega,d_{\Omega})$ is not Gromov hyperbolic. Using Lemma~\ref{lem:slices} we may assume that $d=2$.  

Let $V,W \subset \Rb$ and $U \subset \Cb$ be neighborhoods of $0$ such that $f: V \times U \rightarrow W$ and 
\begin{align*}
\Omega \cap \Oc = \{ (x+iy, z) : x \in V, z \in U, y > f(x,z)\}
\end{align*}
where $\Oc = (V+iW) \times U$. By rescaling we may assume that $B_1(0) \subset U$.

We may assume that $\partial \Omega$ does not contain any non-trivial complex affine disks (otherwise there is nothing to prove). In particular, for any neighborhood $U^\prime \subset U$ of $0$ in $\Cb$ there exists an $z \in U^\prime$ such that $f(0,z) \neq 0$.

Since 
\begin{align*}
\lim_{z \rightarrow 0} \frac{ f(0,z)}{\abs{z}^n} = 0,
\end{align*}
we can find $a_n \rightarrow 0$ and $z_n \in B_1(0)$ such that $f(0,z_n) = a_n \abs{z_n}^n$ and for all $w \in \Cb$ with $\abs{w} \leq \abs{z_n}$ we have 
\begin{align*}
f(0,w) \leq a_n \abs{w}^n.
\end{align*}
Since $\partial \Omega$ has no non-trivial complex affine disks, we see that $z_n \rightarrow 0$ and hence $f(0,z_n) \rightarrow 0$. By passing to a subsequence we may assume that $\abs{f(0,z_n)} < 1$.

Consider the linear transformations
\begin{align*}
A_n= \begin{pmatrix} \frac{1}{f(0,z_n)}  & 0 \\ 0 & z_n^{-1} \end{pmatrix} \in \GL(\Cb^2)
\end{align*}
and let $\Omega_n = A_n \Omega$. Now there exists $\epsilon_1,\epsilon_2 >0$ such that $B_{\epsilon_1}( (\epsilon_2 i,0)) \subset \Omega$. Since $\vec{0} \in \partial \Omega$ and $\abs{z_n},f(0,z_n) < 1$ this implies that $B_{\epsilon_1}( (\epsilon_2 i,0)) \subset \Omega_n$ for all $n$. Moreover for any $R>0$ the set 
\begin{align*}
\{ \Omega^\prime \text{ is open and convex } : B_{\epsilon_1}( (\epsilon_2 i,0)) \subset \Omega^\prime \subset B_R(0) \} \subset \Xb_d
\end{align*}
is compact in the Hausdorff topology. Thus we can pass to a subsequence such that $\Omega_n$ converges in the local Hausdorff topology to a convex open set $\Omega_{\infty}$. 

We claim that $\partial \Omega_{\infty}$ contains a non-trivial complex affine disk and that $\Omega_{\infty}$ is $\Cb$-proper. By the remarks at the start of the proof this will imply the proposition. 

We first show that $\partial \Omega_{\infty}$ contains a non-trivial complex affine disk. If $\Oc_n = A_n \Oc$ we have
\begin{align*}
 \Omega_n \cap \Oc_n = \{ (x+iy,z ) : x \in V_n, z \in U_n,  y > f_n(x,z)\}
\end{align*}
where $V_n = f(z_n,0)^{-1}V$, $U_n = z_n^{-1}U$, and
\begin{align*}
f_n(x,z) =  \frac{1}{f(0,z_n)} f\left( f(0,z_n)x, z_n z\right).
\end{align*}
For $\abs{w} < 1$ we then have
\begin{align}
f_n(0,w) =\frac{f\left( 0, z_n w\right)}{f(0, z_n)} \leq \frac{ a_n \abs{z_n}^n \abs{w}^n }{f(0,z_n)} = \abs{w}^n
\end{align}
which implies that
\begin{align*}
\{( 0, z) : \abs{z} \leq 1\} \subset \partial \Omega_{\infty}.
\end{align*}
Thus $\partial \Omega_{\infty}$ contains a non-trivial complex affine disk.

Establishing that $\Omega_{\infty}$ is $\Cb$-proper is slightly more involved. As a preliminary step we will show that
\begin{align*}
\Omega_{\infty} \cap (\Cb \times \{1 \})= \emptyset.
\end{align*}
For $0 \leq \alpha < \beta \leq \pi$ let $\Cc(\alpha,\beta)$ be the open convex cone in $\Cb$ defined by 
\begin{align*}
\Cc(\alpha,\beta) = \{ z \in \Cb :z \neq 0 \text{ and } \alpha <  \Arg(z) < \beta \}.
\end{align*}
Let $0 \leq \alpha_0 < \beta_0 \leq \pi$ be such that $C(\alpha_0, \beta_0)\times\{0\}$ is the tangent cone of $\Omega \cap (\Cb \times \{0\})$ at $0$, that is 
\begin{align*}
C(\alpha_0, \beta_0) \times \{0\}= \cup_{t > 0} t\Big(\Omega \cap (\Cb \times \{0\})\Big). 
\end{align*}
Notice that $\alpha_0 < \pi/2 < \beta_0$ since $(\epsilon_2 i,0) \in \Omega$. Since 
\begin{align*}
\Omega_n \cap (\Cb \times \{0\}) =  \frac{1}{f(0,z_n)} \Big(\Omega \cap (\Cb \times \{0\}) \Big)
\end{align*}
and $f_n(0,z_n) \rightarrow 0$ there exists $\alpha_n \rightarrow \alpha_0$, $\beta_n \rightarrow \beta_0$, and $R_n \rightarrow \infty$ such that 
\begin{align}
\Omega_n \supset \Big(C(\alpha_n,\beta_n) \times \{0\}\Big) \cap B_{R_n}(0).
\end{align}
Then
\begin{align*}
\Cc(\alpha_0,\beta_0) \times \{0\} \subset \Omega_{\infty}.
\end{align*}
Since $\Omega_{\infty}$ is open and convex this implies that 
\begin{align*}
\vec{z} + \Cc(\alpha_0,\beta_0) \times \{0\} \subset \Omega_{\infty}
\end{align*}
for any $\vec{z} \in \Omega_{\infty}$. Since $(0,1) \in \partial \Omega_{\infty}$ we then have that 
\begin{align*}
\Cc(\alpha_0,\beta_0) \times \{1\} \subset \overline{\Omega}_{\infty}.
\end{align*}
Since $\Omega_{\infty}$ is open and convex either 
\begin{align*}
\Cc(\alpha_0,\beta_0) \times \{1\} \subset \Omega_{\infty}\text{ or }  \Cc(\alpha_0,\beta_0) \times \{1\} \subset \partial \Omega_{\infty}.
\end{align*}
We claim that the latter situation holds. Since $f_n(0,1)=1$ we see that $(i,1) \in \partial \Omega_n$ for all $n$ and so $(i,1) \in \partial \Omega_\infty$. Since $(i,1) \in \Cc(\alpha_0, \beta_0) \times \{1\}$ this implies that
\begin{align*}
\Cc(\alpha_0,\beta_0) \times \{1\} \subset \partial \Omega_{\infty}.
\end{align*}
Which in turn implies that
\begin{align*}
\Omega_{\infty} \cap (\Cb \times \{1 \} )= \emptyset.
\end{align*}
We can now show that $\Omega_{\infty}$ is $\Cb$-proper. Suppose that an affine map $z \rightarrow (a_1,a_2)z+(b_1,b_2)$ has image in $\Omega_{\infty}$. Since 
\begin{align*}
\Omega_{\infty} \subset \{ (z_1, z_2) \in \Cb^2 : \Imaginary(z_1) > 0\}
\end{align*}
 we see that $a_1=0$. And since $\Omega_{\infty} \cap (\Cb \times \{1 \} )= \emptyset$ we also see that $a_2=0$. So $\Omega_{\infty}$ does not contain any non-trivial complex affine lines and hence is $\Cb$-proper.
\end{proof}

\part{A sufficient condition}

\section{M-convexity}\label{sec:m_convex}

In this section we study limits of geodesics $\sigma_n : \Rb \rightarrow \Omega_n$ when $\Omega_n$ converges in the local Hausdorff topology to a set $\Omega$. The first step is to understand limits of complex geodesics $\varphi_n : \Delta \rightarrow \Omega_n$. Using Proposition~\ref{prop:normal_family}, we can pass to a subsequence such that $\varphi_n$ converges locally uniformly to a holomorphic function $\varphi:\Delta \rightarrow \overline{\Omega}$ and either $\varphi(\Delta) \subset \Omega$ or $\varphi(\Delta) \subset \partial \Omega$. When the sequence $( \Omega_n)_{n \in \Nb}$ has uniform convexity properties, we will give conditions on the $\varphi_n$ so that $\varphi(\Delta) \subset \Omega$. This in turn will give us information about limits of geodesics. 

The next definition and Theorem~\ref{thm:m_convex} are motivated by Mercer's work on m-convex sets~\cite{M1993}. 

\begin{definition} \
\begin{enumerate}
\item A $\Cb$-proper convex open set $\Omega$ is called \emph{locally m-convex} if for every $R>0$ there exists $C >0$ such that 
\begin{align*}
\delta_{\Omega}(p; v) \leq C \delta_{\Omega}(p)^{1/m}
\end{align*}
for all $p \in \Omega \cap B_R(0)$ and nonzero $v \in \Cb^d$.

\item Suppose $(\Omega_n)_{n \in \Nb}$ is a sequence of $\Cb$-proper convex open sets, then $\Omega_n$ is called a \emph{locally m-convex sequence} if for every $R>0$ there exists $N,C >0$ such that 
\begin{align*}
\delta_{\Omega_n}(p; v) \leq C \delta_{\Omega_n}(p)^{1/m}
\end{align*}
for all $n > N$, $p \in \Omega_n \cap B_R(0)$, and nonzero $v \in \Cb^d$.
\end{enumerate}
\end{definition}

\begin{observation}\label{obs:no_bd_disks} \
\begin{enumerate}
\item If $\Omega$ is locally $m$-convex then $\partial \Omega$ contains no non-trivial holomorphic disks. To see this, recall that $\partial \Omega$ contains a non-trivial holomorphic disk if and only if $\partial \Omega$ contains a non-trivial complex affine disk (see  Lemma~\ref{lem:affine_disk}), but the latter is clearly impossible for a locally $m$-convex set.
\item Suppose $\Omega_n$ is a locally m-convex sequence of $\Cb$-proper convex open sets converging to a $\Cb$-proper convex open set $\Omega$ is the local Hausdorff topology. Then $\Omega$ is locally $m$-convex. 
\end{enumerate}
\end{observation}

\begin{example}\label{ex:L_convex}
Suppose that $\Omega = \{ r(z) < 0\}$ is $\Cb$-proper convex open set of finite line type $L$. Then by Proposition~\ref{prop:finite_type}, $\Omega$ is locally $L$-convex.
\end{example}

We first give a sufficient condition for a sequence of complex geodesics $\varphi_n : \Delta \rightarrow \Omega_n$ to converge to a complex geodesic $\varphi:\Delta \rightarrow \Omega$.

\begin{theorem}\label{thm:m_convex}
Suppose $\Omega_n$ is a locally m-convex sequence of $\Cb$-proper convex open sets converging to a $\Cb$-proper convex open set $\Omega$ in the local Hausdorff topology. If  $\varphi_n : \Delta \rightarrow \Omega_n$ is a sequence of complex geodesics and there exists numbers $R >0$ and $a_n,b_n \in (-1,1)$ such that for all $n$
\begin{enumerate}
\item $0 \in [a_n,b_n]$,
\item $\varphi_n([a_n,b_n]) \subset B_R(0)$
\item $\lim_{n \rightarrow \infty} \norm{\varphi_n(a_n)-\varphi_n(b_n)}>0$, and
\item $\delta_{\Omega_n}(\varphi_n(0)) = \max \{ \delta_{\Omega_n}(\varphi_n(t)) : t \in [a_n,b_n]\}$,
\end{enumerate}
then there exists $n_k \rightarrow \infty$ such that $\varphi_{n_k}$ converges locally uniformly to a complex geodesic $\varphi:\Delta \rightarrow \Omega$.
\end{theorem}

The proof of Theorem~\ref{thm:m_convex} closely follows arguments found in Mercer~\cite{M1993} and Chang, Hu, and Lee~\cite{CHL1988}. 

\begin{lemma}
Suppose $\Omega_n$ is a sequence of $\Cb$-proper convex open sets converging to a $\Cb$-proper convex open set $\Omega$ in the local Hausdorff topology. Fix $R>0$ and a point $o \in \Omega$ then there exists $C_1>0$, $N>0$,  and $\alpha >1$ such that 
\begin{align*}
d_{\Omega_n}(p, o) \leq C_1-\frac{\alpha}{2}\log  \delta_{\Omega_n}(p)
\end{align*}
for all for $n>N$ and all $p \in \Omega_n \cap B_R(o)$.
\end{lemma}

\begin{proof}
By Lemma~\ref{lem:lem2}, there exists $\delta >0$ such that for all $n$ sufficiently large $B_{\delta}(o) \subset \Omega_n$. Next define
\begin{align*}
T:= \sup\{ \delta_{\Omega_n}(p) : p \in \Omega_n \cap B_R(o), n \in \Nb\},
\end{align*}
\begin{align*}
\alpha := 2(R+1)/\delta,
\end{align*}
and
\begin{align*}
C_1 := \max\left\{ \frac{(R+1)R}{\delta} + \frac{\alpha}{2}\log(T), \frac{\alpha}{2}\log(R+1)\right\}.
\end{align*}
 
Let $p \in B_R(o) \cap \Omega_n$. We will consider two cases. \newline

\noindent \textbf{Case 1:} Suppose that the line segment $\{ o+s (p-o) : s\geq 0\}$ does not intersect $\partial \Omega_n$ at a point in $B_{R+1}(o)$. Then since $\Omega_n$ contains the interior of the convex hull of $B_{\delta}(o)$ and $\{ o + s(p-o) : s\geq0\} \cap B_{R+1}(o)$ we see that 
\begin{align*}
\delta_{\Omega_n}( o+s (p-o)) \geq \frac{\delta}{R+1}
\end{align*}
for $s \in [0,1]$. Which implies that 
\begin{align*}
d_{\Omega_n}(x,p)
& \leq \int_0^1 K_{\Omega_n}(o+s(p-o); p-o) ds \leq \int_0^1 \frac{ \norm{p-o}}{\delta_{\Omega_n}( o+s (p-o))} ds \\
& \leq \frac{R+1}{\delta}\norm{p-o} \leq \frac{(R+1)R}{\delta}.
\end{align*}
So
\begin{align*}
d_{\Omega_n}(x,p) 
&\leq  \frac{(R+1)R}{\delta} +\frac{\alpha}{2} \log\left(T\right)-\frac{\alpha}{2}\log \delta_{\Omega_n}(p) \\
&\leq C_1-\frac{\alpha}{2}\log  \delta_{\Omega_n}(p).
\end{align*}

\noindent \textbf{Case 2:} Suppose that the line segment $\{ o+s (p-o) : s\geq 0\}$ intersects $\partial \Omega_n$ at a point in $x \in B_{R+1}(o)$. Then
\begin{align*}
\delta_{\Omega_n}(o; \vec{ox}) \geq \delta_{\Omega_n}(o) \geq \delta \geq \frac{\delta}{R+1} \norm{x-o}.
\end{align*}
So by Lemma~\ref{lem:quasi_geodesic}
\begin{align*}
d_{\Omega_n}(o, x+e^{-2t}(o-x)) \leq \alpha t
\end{align*}
for $t \geq 0$. Since  $p = x+e^{-2t_p}(o-x)$ where
\begin{align*}
t_p = -\frac{1}{2}\log\left( \frac{\norm{p-x}}{\norm{x-o}} \right) \leq \frac{1}{2}\log(R+1) -\frac{1}{2}\log\delta_{\Omega_n}(p)
\end{align*}
we then have 
\begin{align*}
d_{\Omega_n}(o,p) 
&= d_{\Omega_n}(o, x+e^{-2t_p}(o-x)) \leq \alpha t_p \leq  \frac{\alpha}{2} \log(R+1)  - \frac{\alpha}{2}\log\delta_{\Omega_n}(p) \\
& \leq C_1-\frac{\alpha}{2}\log  \delta_{\Omega_n}(p).
\end{align*}
\end{proof}

\begin{lemma}
Suppose $\Omega_n, \Omega, \varphi_n, R, a_n, b_n$ are as in the statement of Theorem~\ref{thm:m_convex}. Then there exists $N>0$, $C_2>0$, and $\alpha >1$ such that 
\begin{align*}
\delta_{\Omega_n}(\varphi_n(t)) \leq C_2 (1- \abs{t})^{1/(2\alpha)}.
\end{align*}
for all $n > N$ and $t \in [a_n,b_n]$.
\end{lemma}

\begin{proof}
By the previous Lemma there exists $C_1>0$ and $\alpha >1$ such that
\begin{align*}
d_{\Omega}(\varphi_n(0), \varphi_n(t))
& \leq d_{\Omega}(\varphi_n(0), o) + d_{\Omega}(\varphi_n(t),o)  \\
&\leq 2C_1 - \frac{1}{2}\log \Big( \delta_{\Omega_n}(\varphi_n(0))^{\alpha} \delta_{\Omega_n}(\varphi_n(t))^{\alpha} \Big).
\end{align*}
for all $n$ sufficiently large and $t \in [a_n,b_n]$. 

Since $\varphi_n$ is a complex geodesic 
\begin{align*}
d_{\Omega}(\varphi_n(0), \varphi_n(t)) = d_{\Delta}(0,t) =\frac{1}{2} \log \frac{1+\abs{t}}{1-\abs{t}} \geq \frac{1}{2} \log \frac{1}{1-\abs{t}}.
\end{align*}
Hence
\begin{align*}
 \delta_{\Omega_n}(\varphi_n(0))^{\alpha} \delta_{\Omega_n}(\varphi_n(t))^{\alpha} \leq e^{4C_1} (1-\abs{t}).
 \end{align*}
Since  $\delta_{\Omega_n}(\varphi_n(t)) \leq \delta_{\Omega_n}(\varphi_n(0))$ we have that 
\begin{equation*}
 \delta_{\Omega_n}(\varphi_n(t)) \leq e^{2C_1/\alpha} (1-\abs{t})^{1/(2\alpha)}. \qedhere
 \end{equation*}
\end{proof}

\begin{lemma}\label{lem:lipschitz}
Suppose $\Omega_n, \Omega, \varphi_n, R, a_n, b_n$ are as in the statement of Theorem~\ref{thm:m_convex}. Then there exists $N>0$, $C_3>0$, and $\alpha >1$ such that 
\begin{align*}
\norm{ \varphi_n(t_1) - \varphi_n(t_2)} \leq C_3\abs{t_1-t_2}^{1/(2\alpha m)}
\end{align*}
for all $n >N$ and $t_1,t_2 \in [a_n, b_n]$.
\end{lemma}

\begin{proof}
By Lemma~\ref{lem:convex_inf} we have
\begin{align*}
\frac{\norm{ \varphi_n^\prime(t)}}{2\delta_{\Omega_n}(\varphi_n(t); \varphi_n^\prime(t))} \leq K_{\Omega_n}(\varphi_n(t); \varphi_n^\prime(t)) 
\end{align*}
and since $\varphi_n$ is a complex geodesic 
\begin{align*}
K_{\Omega_n}(\varphi_n(t); \varphi_n^\prime(t)) = K_{\Delta}(t; 1) = \frac{1}{1-\abs{t}^2} \leq \frac{1}{1-\abs{t}}.
\end{align*}
Since $(\Omega_n)_{n \in \Nb}$ is a locally convex sequence there exists $C>0$ such that 
\begin{align*}
\delta_{\Omega_n}(\varphi_n(t); \varphi_n^\prime(t)) \leq C \delta_{\Omega_n}(\varphi_n(t))^{1/m}
\end{align*}
for $t \in [a_n,b_n]$ and $n$ sufficiently large. Then by the previous lemma there exists $M >0$ such that 
\begin{align*}
\norm{ \varphi_n^\prime(t)} 
&\leq \frac{2 \delta_{\Omega_n}(\varphi_n(t); \varphi_n^\prime(t))}{1-\abs{t}} \leq  \frac{2 C \delta_{\Omega_n}(\varphi_n(t))^{1/m}}{1-\abs{t}} \\
& \leq \frac{M}{(1-\abs{t})^{1-1/(2\alpha m)}}.
\end{align*}

The rest of the lemma follows the proof of the Hardy-Littlewood theorem (see for instance~\cite[Theorem 2.6.26]{A1989}), but for convenience we will provide the argument. Let $\beta = 1/(2\alpha m)$. If we fix $\epsilon \in (0,1)$ then it is enough to show that there exists $C_3 >0$ such that
\begin{align*}
\norm{\varphi_n(t_1) - \varphi_n(t_2)} \leq C_3\abs{t_1-t_2}^{\beta}
\end{align*}
for all $t_1, t_2 \in [a_n,b_n]$ with $\abs{t_1-t_2} < \epsilon$.  

First suppose that $0 \leq t_1 \leq t_2 \leq b_n$. Then
\begin{align*}
\norm{\varphi_n(t_1)-\varphi_n(t_2)} \leq \int_{t_1}^{t_2} \norm{\varphi_n^\prime(t)} dt \leq M \int_{t_1}^{t_2} \frac{dt}{(1-t)^{1-\beta}}.
\end{align*}
If $t_2-t_1 \leq 1-t_2$ then 
\begin{align*}
\norm{\varphi_n(t_1)-\varphi_n(t_2)} \leq M\frac{ t_2-t_1}{(1-t_2)^{1-\beta}} \leq M \abs{t_2-t_1}^{\beta}.
\end{align*}
If $t_2 - t_1 \geq 1-t_2$ then 
\begin{align*}
\norm{\varphi_n(t_1)-\varphi_n(t_2)} \leq \left. M\frac{(1-t)^{\beta}}{\beta} \right|_{t_1}^{t_2} \leq \frac{M}{\beta} \abs{t_2-t_1}^{\beta}.
\end{align*}
As similar argument establishes the necessary bounds when $a_n < t_2 \leq t_1 \leq 0$. 

Finally suppose that $a_n < t_1 < 0 < t_2 < b_n$ and $\abs{t_1-t_2} < \epsilon$. Then 
\begin{align*}
\norm{\varphi_n(t_1)-\varphi_n(t_2)} 
&\leq M \int_{t_1}^{t_2} \frac{dt}{(1-\abs{t})^{1-\beta}} \leq M  \int_{t_1}^{t_2} \frac{dt}{(1-\epsilon)^{1-\beta}}\\
& = \frac{M}{(1-\epsilon)^{1-\beta}}\abs{t_1-t_2} \leq \frac{M}{(1-\epsilon)^{1-\beta}}\abs{t_1-t_2}^{\beta}.
\end{align*}
\end{proof}

\begin{proof}[Proof of Theorem~\ref{thm:m_convex}]
By Proposition~\ref{prop:normal_family}, we can pass to a subsequence such that $\varphi_n$ converges locally uniformly to a holomorphic map $\varphi: \Delta \rightarrow \overline{\Omega}$ and either $\varphi(\Delta) \subset \Omega$ or $\varphi(\Delta) \subset \partial \Omega$. Since $\Omega$ is locally $m$-convex, if $\varphi$ is non-constant then by Observation~\ref{obs:no_bd_disks} we must have that $\varphi(\Delta) \subset \Omega$. In this case we immediately see that $\varphi:\Delta \rightarrow \Omega$ is a complex geodesic by Theorem~\ref{thm:dist_conv}. So it is enough to show that $\varphi$ is non-constant.

By the above lemma there exists $\beta >1$ and $C>0$ such that 
\begin{align*}
\norm{\varphi_n(t_1) - \varphi_n(t_2)} \leq C\abs{t_1-t_2}^{\beta}
\end{align*}
for all $n$ sufficiently large and $t_1,t_2 \in [a_n,b_n]$. 

By passing to a subsequence we can assume that $a_n \rightarrow a_{\infty}$ and $b_n \rightarrow b_{\infty}$. Since 
\begin{align*}
0 < \lim_{n \rightarrow \infty} \norm{\varphi_n(a_n) - \varphi_n(b_n))} \leq \lim_{n \rightarrow \infty} C\abs{a_n-b_n}^{\beta}=C\abs{a_\infty-b_\infty}^{\beta}
\end{align*}
we must have that $a_{\infty} \neq b_{\infty}$. 

Now fix $\epsilon >0$ sufficiently small. By passing to a subsequence we may suppose that $[a_{\infty}+\epsilon, b_{\infty}-\epsilon] \subset [a_n,b_n]$ for all $n$. Then 
\begin{align*}
\norm{\varphi(a_{\infty}+\epsilon)-\varphi(b_{\infty}-\epsilon)} 
&= \lim_{n \rightarrow \infty} \norm{\varphi_n(a_{\infty}+\epsilon)-\varphi_n(b_{\infty}-\epsilon)} \\
& \geq \lim_{n \rightarrow \infty} \norm{\varphi_n(a_{n})-\varphi_n(b_{n})} -\norm{\varphi_n(a_{n})-\varphi_n(a_{\infty}+\epsilon)}\\
& \quad \quad -\norm{\varphi_n(b_{n})-\varphi_n(b_{\infty}-\epsilon)}  \\
& \geq \lim_{n \rightarrow \infty} \norm{\varphi_n(a_{n})-\varphi_n(b_{n})} - 2C \abs{\epsilon}^{\beta}.
\end{align*}
So for $\epsilon$ small enough, we see that 
\begin{align*}
\norm{\varphi(a_{\infty}+\epsilon)-\varphi(b_{\infty}-\epsilon)} >0
\end{align*}
and thus $\varphi$ is not constant. Hence by the remarks at the start of the proof $\varphi: \Delta \rightarrow \Omega$ is a complex geodesic. 
\end{proof}

We now turn our attention to general geodesics. 

\begin{proposition}\label{prop:m_convex}
Suppose $\Omega_n$ is a locally m-convex sequence of $\Cb$-proper convex open sets converging to a $\Cb$-proper convex open set $\Omega$ in the local Hausdorff topology. Assume $\sigma_n : \Rb \rightarrow \Omega_n$ is a sequence of geodesics such that there exists $a_n \leq b_n$ and $R >0$ satisfying 
\begin{enumerate}
\item $\sigma_n([a_n,b_n]) \subset B_R(0)$,
\item $\lim_{n \rightarrow \infty} \norm{\sigma_n(a_n)-\sigma_n(b_n)} > 0$, 
\end{enumerate}
then there exists $T_n \in [a_n,b_n]$ such that a subsequence of $t\rightarrow \sigma_n(t+T_n)$ converges locally uniformly to a geodesic $\sigma:\Rb \rightarrow \Omega$.  
\end{proposition}

\begin{proof}
Suppose 
\begin{align*}
\lim_{n \rightarrow \infty} \norm{\sigma_n(a_n)-\sigma_n(b_n)} = \epsilon
\end{align*}
and pick $m_n \in [a_n,b_n]$ such that 
\begin{align*}
\lim_{n \rightarrow \infty} \norm{\sigma_n(a_n)-\sigma_n(m_n)} =\lim_{n \rightarrow \infty} \norm{\sigma_n(m_n)-\sigma_n(b_n)}\geq \epsilon/2.
\end{align*}
By passing to a subsequence we may suppose that $\sigma_n(m_n)$ converges to some point $y \in \overline{\Omega} \cap B_R(0)$. If $y \in \Omega$ then the Arzel{\`a}-Ascoli theorem implies that a subsequence of $\sigma_n(t+m_n)$ converges locally uniformly to a geodesic $\sigma:\Rb \rightarrow \Omega$. So assume for a contradiction that $y \in \partial \Omega$. 

Now let $\varphi_n:\Delta \rightarrow \Omega_n$ be a complex geodesic with $\varphi_n(u_n) = \sigma_n(a_n)$ and $\varphi_n(v_n) = \sigma_n(m_n)$ for some $u_n < v_n \in (-1,1)$. We would like to apply Theorem~\ref{thm:m_convex} to the complex geodesics $\varphi_n$ but it is unclear if the sequence $\varphi_n|_{[u_n,v_n]}$ is uniformly bounded. With this in mind define
\begin{align*}
u_n^\prime = \inf \{ t\geq u_n : \varphi_n([t,v_n]) \subset B_{R+1}(0) \}.
\end{align*}
Notice that 
\begin{align*}
\lim_{n \rightarrow \infty} \norm{\varphi_n(u_n^\prime) - \varphi_n(v_n)} > 0.
\end{align*}
By reparametrizing $\varphi_n$, we may assume that  $0 \in [u_n^\prime, v_n]$ and 
\begin{align*}
\delta_{\Omega_n}(\varphi_n(0)) = \max \{ \delta_{\Omega_n}(\varphi_n(t)) : t \in [u_n^\prime, v_n] \}.
\end{align*}
In a similar fashion let $\phi_n:\Delta \rightarrow \Omega_n$ be a complex geodesic with $\phi_n(s_n) = \sigma_n(m_n)$ and $\phi_n(t_n) = \sigma_n(b_n)$ for some $s_n < t_n \in (-1,1)$. Now let 
\begin{align*}
t_n^\prime = \sup \{ t\leq t_n : \varphi_n([s_n,t]) \subset B_{R+1}(0) \}.
\end{align*}
We may assume that $0 \in [s_n, t_n^\prime]$ and 
\begin{align*}
\delta_{\Omega_n}(\phi_n(0)) = \max \{ \delta_{\Omega_n}(\phi_n(t)) : t \in [s_n, t_n^\prime] \}.
\end{align*}
Since 
\begin{align*}
d_{\Omega_n}(\sigma_n(a_n), \sigma_n(b_n)) = d_{\Omega_n}(\sigma_n(a_n), \sigma_n(m_n)) +d_{\Omega_n}(\sigma_n(m_n), \sigma_n(b_n)) 
\end{align*}
the concatenation of $\varphi_n|_{[u_n,v_n]}$ and $\phi_n|_{[s_n,t_n]}$ is a geodesic in $\Omega_n$. In particular, 
\begin{align*}
d_{\Omega_n}(\varphi_n(0), \phi_n(0)) = d_{\Omega_n}(\varphi_n(0), \sigma_n(m_n)) + d_{\Omega_n}(\sigma_n(m_n), \phi_n(0)).
\end{align*}

Now by Theorem~\ref{thm:m_convex} we may pass to a subsequence such that $\varphi_n$ converges locally uniformly to a complex geodesic $\varphi :\Delta \rightarrow \Omega$ and $\phi_n$ converges locally uniformly to a complex geodesic $\phi:\Delta \rightarrow \Omega$. But then
\begin{align*}
d_{\Omega}(\varphi(0), \phi(0)) 
&= \lim_{n \rightarrow \infty} d_{\Omega_n}(\varphi_n(0), \phi_n(0)) \\
& = \lim_{n \rightarrow \infty} d_{\Omega_n}(\varphi_n(0), \sigma_n(m_n)) + d_{\Omega_n}(\sigma_n(m_n), \phi_n(0)) \\
&= \infty
\end{align*}
because $y \in \partial \Omega$. This contradicts the fact that $\varphi(0), \phi(0) \in \Omega$.
\end{proof}

\begin{proposition}\label{prop:limits_exist}
Suppose $\Omega_n$ is a locally m-convex sequence of $\Cb$-proper convex open sets converging to a $\Cb$-proper convex open set $\Omega$ in the local Hausdorff topology. Assume $\sigma_n : \Rb \rightarrow \Omega_n$ is a sequence of geodesics converging locally uniformly to a geodesic $\sigma:\Rb \rightarrow \Omega$. If $t_n \rightarrow \infty$ is a sequence such that $\lim_{n \rightarrow \infty} \sigma_n(t_n) = x_{\infty} \in \overline{\Cb^d}$  then
\begin{align*}
\lim_{t \rightarrow \infty} \sigma(t) = x_{\infty}.
\end{align*}
\end{proposition}

\begin{proof}
Suppose for a contradiction that $\lim_{t \rightarrow \infty} \sigma(t) \neq x_{\infty}$. Then there exists $s_n^\prime \rightarrow \infty$ such that $\sigma(s_n^\prime) \rightarrow y_{\infty} \in \overline{\Cb^d}$ and $x_{\infty} \neq y_{\infty}$. Now since $\sigma_n$ converges locally uniformly to $\sigma$ there exists $s_n \rightarrow \infty$ such that $\sigma_n(s_n) \rightarrow y_{\infty}$.

Now since $x_{\infty}$ and $y_{\infty}$ are distinct at least one is finite and hence there exists $[u_n,v_n] \subset [\min\{s_n,t_n\}, \infty)$ and $R,\epsilon >0$ such that 
\begin{enumerate}
\item $\sigma_n([u_n,v_n]) \subset B_R(0)$, 
\item $\norm{\sigma_n(u_n) - \sigma_n(v_n)} > \epsilon$.
\end{enumerate}
Then by Proposition~\ref{prop:m_convex}, there exists $T_n \in [u_n,v_n]$ such that $\sigma_n(t+T_n)$ converges locally uniformly to a geodesic $\wh{\sigma}:\Rb \rightarrow \Omega$. But this is a contradiction since 
\begin{equation*}
d_{\Omega}(\sigma(0),\wh{\sigma}(0)) = \lim_{n \rightarrow \infty} d_{\Omega_n}(\sigma_n(0), \sigma_n(T_n)) \geq \lim_{n \rightarrow \infty} \min\{s_n, t_n\} = \infty. \qedhere
\end{equation*}
\end{proof} 

\begin{corollary}\label{cor:geodesic_limits}
Suppose $\Omega$ is a locally m-convex open set. If $\sigma: \Rb \rightarrow \Omega$ is a geodesic then 
\begin{align*}
\lim_{t \rightarrow -\infty} \sigma(t) \text{ and } \lim_{t \rightarrow +\infty} \sigma(t) 
\end{align*}
both exist in $\overline{\Cb^d}$.
\end{corollary}

\begin{proof}
Notice that the constant sequence $\Omega_n : = \Omega$ converges to $\Omega$ in the local Hausdorff topology and $\sigma_n:=\sigma$ converges to $\sigma$ locally uniformly. So we may apply the previous proposition. 
\end{proof}

Combining Proposition~\ref{prop:limits_exist} and Corollary~\ref{cor:geodesic_limits} we have:

\begin{corollary}
Suppose $\Omega_n$ is a locally m-convex sequence of $\Cb$-proper convex open sets converging to a $\Cb$-proper convex open set $\Omega$ in the local Hausdorff topology. Assume $\sigma_n : \Rb \rightarrow \Omega_n$ is a sequence of geodesics converging locally uniformly to a geodesic $\sigma:\Rb \rightarrow \Omega$. If each $\Omega_n$ is locally $m$-convex  then
\begin{align*}
\lim_{n \rightarrow \infty}\lim_{ t\rightarrow \infty} \sigma_n(t) = \lim_{t \rightarrow \infty} \sigma(t).
\end{align*}
\end{corollary}

\section{A sufficient condition for Gromov hyperbolicity}\label{sec:suff}

Before stating our sufficient condition for Gromov hyperbolicity we need one more definition.

\begin{definition}
Suppose $\Omega$ is a $\Cb$-proper convex open set. A geodesic $\sigma : \Rb\rightarrow \Omega$ is \emph{well behaved} if both limits 
\begin{align*}
\lim_{t \rightarrow -\infty} \sigma(t) \text{ and } \lim_{t \rightarrow +\infty} \sigma(t)
\end{align*}
exist in $\overline{\Cb^d}$ and are distinct. 
\end{definition}

\begin{remark}
We have already seen in Corollary~\ref{cor:geodesic_limits} that these limits exists when $\Omega$ is locally m-convex. In Section~\ref{sec:gromov_prod} we will show that the limits (when they exist) are distinct if $\Omega$ is bounded and $\partial\Omega$ is $C^2$. In Section~\ref{sec:multi_infty} we will show that certain unbounded domains also have well behaved geodesics.
\end{remark}

\begin{theorem}\label{thm:suff}
Suppose $\Omega$ is a $\Cb$-proper convex open set. If for every sequence $u_n \in \Omega$ there exists a subsequence $n_k \rightarrow \infty$, affine maps $A_k \in \Aff(\Cb^d)$, and a $\Cb$-proper convex open set $\wh{\Omega}$ such that 
\begin{enumerate}
\item $A_k \Omega \rightarrow \wh{\Omega}$ in the local Hausdorff topology,
\item $A_k u_{n_k} \rightarrow u_{\infty} \in \wh{\Omega}$,
\item $(A_k \Omega)_{k \in \Nb}$ is a locally m-convex sequence,
\item all geodesics in $\wh{\Omega}$ are well behaved,
\end{enumerate}
then $(\Omega, d_{\Omega})$ is Gromov hyperbolic.
\end{theorem}

\begin{proof}

Suppose $(\Omega, d_{\Omega})$ is not Gromov hyperbolic. Then there exists points $x_n,y_n, z_n \in \Omega$, geodesic segments $\sigma_{x_ny_n}, \sigma_{y_nz_n}, \sigma_{z_nx_n}$ joining them, and a point $u_n$ in the image of $\sigma_{x_ny_n}$ such that
\begin{align*}
d_{\Omega}(u_n, \sigma_{y_nz_n} \cup \sigma_{z_nx_n}) > n.
\end{align*}
By passing to a subsequence there exists affine maps $A_n \in \Aff(\Cb^d)$ and a $\Cb$-proper convex open set $\wh{\Omega}$ such that 

\begin{enumerate}
\item $A_n \Omega \rightarrow \wh{\Omega}$ in the local Hausdorff topology,
\item $A_n u_n \rightarrow u_{\infty} \in \wh{\Omega}$,
\item $(A_n \Omega)_{n \in \Nb}$ is a locally m-convex sequence,
\item geodesics in $\wh{\Omega}$ are well behaved.
\end{enumerate}
By passing to another subsequence we can suppose that $A_n x_n \rightarrow x_{\infty}$, $A_n y_n \rightarrow y_{\infty}$, and $A_n z_n \rightarrow z_{\infty}$ for some $x_{\infty}, y_{\infty}, z_{\infty} \in \overline{\Cb^d}$. 

Parametrize $\sigma_{x_ny_n}$ such that $\sigma_{x_ny_n}(0)=u_n$ then using the Arzel{\`a}-Ascoli theorem we can pass to a subsequence such that $A_n\sigma_{x_ny_n}$ converges locally uniformly to a geodesic $\sigma: \Rb \rightarrow \wh{\Omega}$. Moreover, by Proposition~\ref{prop:limits_exist}
\begin{align*}
\lim_{t \rightarrow -\infty} \sigma(t) = x_{\infty} \text{ and } \lim_{t \rightarrow +\infty} \sigma(t) = y_{\infty}.
\end{align*}
Since geodesics in $\wh{\Omega}$ are well behaved we must have that $x_{\infty} \neq y_{\infty}$. 

So $z_{\infty}$ does not equal at least one of $x_{\infty}$ or $y_{\infty}$. By relabeling we can suppose that $x_{\infty} \neq z_{\infty}$. Since $x_{\infty} \neq z_{\infty}$ at least one is finite and hence by Proposition~\ref{prop:m_convex} we may pass to a subsequence and parametrize $\sigma_{x_nz_n}$ so that it converges locally uniformly to a geodesic $\wh{\sigma}:\Rb \rightarrow \wh{\Omega}$. But then
\begin{align*}
d_{\wh{\Omega}}(u_{\infty},\wh{\sigma}(0)) 
&= \lim_{n \rightarrow \infty} d_{\Omega_n}(A_nu_n, A_n\sigma_{x_n z_n} (0))\\
&= \lim_{n \rightarrow \infty} d_{\Omega}(u_n, \sigma_{x_n z_n} (0))\\
& \geq  \lim_{n \rightarrow \infty} d_{\Omega}(u_n, \sigma_{x_nz_n})=\infty
\end{align*}
which is a contradiction. Thus $(\Omega, d_{\Omega})$ is Gromov hyperbolic. \end{proof}

\part{Convex domains of finite type}

In this part we apply Theorem~\ref{thm:suff} (Theorem~\ref{thm:suff_i} in the introduction) to convex domains of finite type and show that they have Gromov hyperbolic Kobayashi metric.

\section{Finite type and m-convexity}\label{sec:finite_type_m_conex} 

Finite line type is related to m-convexity by the following (well known) Proposition: 

\begin{proposition}\label{prop:finite_type}
Suppose $\Omega$ is a $\Cb$-proper open convex set and $\partial \Omega$ is $C^L$ and has finite line type $L$ near some $\xi \in \partial \Omega$.  Then there exists a neighborhood $U$ of $\xi$ and a $C>0$ such that 
\begin{align*}
\delta_{\Omega}(p;v) \leq C \delta_{\Omega}(p)^{1/L}
\end{align*}
for all $p \in U \cap \Omega$ and $v \in \Cb^d$ nonzero.
\end{proposition}

There are several ways to establish this Proposition, for instance it follows from the proof of Lemma 1.3 in~\cite{G1997}. In order to apply Theorem~\ref{thm:suff} we will need to show that certain sequences of convex sets are locally m-convex sequences and the rest of this section is devoted to giving an estimate on the constant $C>0$ in Proposition~\ref{prop:finite_type}. 

We begin with an observation:

\begin{lemma}
For any integer $L>0$ there exists $A_L>0$ such that if $P: \Cb \rightarrow \Cb$ is a polynomial of the form 
\begin{align*}
P(z) = \sum_{1 \leq a+b \leq L} \alpha_{a, b} z^a \overline{z}^b
\end{align*}
then 
\begin{align*}
A_L\epsilon^L\sup_{1 \leq a+b \leq L}\Big\{ \abs{\alpha_{a,b}}  \Big\} \leq \sup_{\abs{z} \leq \epsilon} \abs{P(z)}
\end{align*}
for any $\epsilon \in (0,1)$. 
\end{lemma}

\begin{proof} 
Consider the vector space $\Vc$ consisting of polynomials of the form
\begin{align*}
P(z) = \sum_{1 \leq a+b \leq L} \alpha_{a, b} z^a \overline{z}^b. 
\end{align*}
By the equivalence of finite dimensional norms we see that there exists $A_L>0$ such that 
\begin{align*}
A_L\sup_{1 \leq a+b \leq L} \Big\{ \abs{\alpha_{a,b}} \Big\} \leq  \sup_{ \abs{z} \leq 1} \abs{ P(z)}
\end{align*}
for any $P \in \Vc$. Then for $P \in \Vc$ 
\begin{align*}
\sup_{\abs{z} \leq \epsilon} \abs{P(z)} = \sup_{\abs{z} \leq 1} \abs{P(\epsilon z)} \geq  A_L \sup\Big\{ \abs{\alpha_{a,b}}\epsilon^{a+b} \Big\} \geq A_L \epsilon^L  \sup\Big\{ \abs{\alpha_{a,b}} \Big\} .
\end{align*}
\end{proof}

Now suppose $\Omega$ is a $\Cb$-proper open convex set and $\partial \Omega$ is a $C^L$ hypersurface. Assume that every point  $x \in \partial \Omega \cap \overline{B_{R+1}(0)}$ has finite line type at most $L$. Let $r: \Cb^d \rightarrow \Rb$ be a defining function of $\Omega$, that is $\Omega = \{ r(z) < 0\}$, $r$ is $C^L$, and $\nabla r \neq 0$ in a neighborhood of $\partial \Omega$. 

Let $\Lc(\Omega, R)$ be the set of affine lines $\ell: \Cb \rightarrow \Cb^d$ with $\ell(0) \in \partial \Omega \cap B_{R+1}(0)$, $\ell^\prime(0) \in T_{\ell(0)} \partial \Omega$, and $\abs{\ell^\prime(0)}=1$. For any $\ell \in \Lc(\Omega, R)$ let 
\begin{align*}
P_{\ell}(z) = \sum_{1 \leq a+b \leq L} \alpha_{a,b}(\ell) z^a \overline{z}^b
\end{align*}
be the $L^{th}$ order Taylor polynomial of $(r \circ \ell)(z)$ at $z=0$. Next define 
\begin{align*}
\alpha(\Omega, R):= \inf_{\ell \in \Lc(\Omega, R)} \sup\Big\{ \abs{\alpha_{a,b}(\ell)} : 1 \leq a+b \leq L \Big\}
\end{align*} 
for all $\ell \in \Lc(\Omega, R)$. By the finite type hypothesis $\alpha(\Omega, R)$ is positive. 

Next let $\epsilon(\Omega, R)$ be the largest number such
\begin{align*}
\abs{P_\ell(z) - (r\circ \ell)(z)} \leq \frac{A_L \alpha(\Omega, R)}{4} \abs{z}^L
\end{align*}
for every $\ell \in \Lc(\Omega, R)$ and $\abs{z} \leq \epsilon(\Omega, R)$. Since $\partial \Omega$ is a $C^L$ hypersurface $\epsilon(\Omega, R)$ is positive. 

Finally define
\begin{align*}
 \kappa(\Omega, R): = \sup\{ \norm{\nabla r(z)} : z \in B_{R+4}(0)\}
 \end{align*}
and
\begin{align*}
D(\Omega, R) := \sup\{ \delta_{\Omega}(p;v) : p \in \Omega \cap B_R(0), v \in \Cb^d \text{ nonzero} \}.
\end{align*}
Since $\Omega$ is $\Cb$-proper, $D(\Omega, R)$ is finite. 

Although the parameters $\kappa$, $\alpha$, and $\epsilon$ also depend on the defining function $r$ we suppress this dependency. With all this notation we have the following:

\begin{proposition}\label{prop:finite_type_2}
Suppose $\kappa_0, \alpha_0, \epsilon_0, D_0, R>0$. There exists $C>0$ such that if $\Omega=\{ r(z) < 0\}$ is a $\Cb$-proper open convex set, $\partial \Omega$ is $C^L$ and has finite line type $L$ in a neighborhood of $B_R(0)$, $\kappa(\Omega, R) \leq \kappa_0$, $\alpha(\Omega, R) \geq \alpha_0$, $\epsilon(\Omega, R) \geq \epsilon_0$, and $D(\Omega, R) \leq D_0$ then 
\begin{align*}
\delta_{\Omega}(p;v) \leq C \delta_{\Omega}(p)^{1/L}
\end{align*}
for all $p \in B_R(0) \cap \Omega$ and $v \in \Cb^d$ nonzero.
\end{proposition}

\begin{proof}[Proof of Proposition~\ref{prop:finite_type_2}]
We can assume that $\epsilon_0 \leq 1$. Let 
\begin{align*}
C_0: = \frac{4\kappa(\Omega,R)}{A_L\alpha(\Omega,R)} \leq \frac{4 \kappa_0}{A_L \alpha_0}.
\end{align*}
We will show that 
\begin{align*}
\delta_{\Omega}(p;v) \leq 2 C_0^{1/L} \delta_{\Omega}(p)^{1/L}
\end{align*}
for $v \in \Cb^d$ nonzero and $p \in \Omega \cap B_R(0)$ satisfying 
\begin{align*}
\delta_{\Omega}(p)  \leq \delta_0:=\min\left\{1, \frac{\epsilon_0^L}{C_0}, \left(\frac{C_0}{3^{L/2}}\right)^{1/(L-1)}\right\}.
\end{align*} 
Since $D(\Omega, R) \leq D_0$, this will imply that 
\begin{align*}
\delta_{\Omega}(p;v) \leq \max\left\{ 2D_0/\delta_0, 2C_0^{1/L} \right\}\delta_{\Omega}(p)^{1/L} 
\end{align*}
for all $p \in B_R(0) \cap \Omega$ and $v \in \Cb^d$ nonzero.

So suppose $p \in \Omega \cap B_R(0)$ satisfies $\delta_{\Omega}(p) \leq \delta_0$ and $v \in \Cb^d$ is a unit vector. Let $x$ be a closest point to $p$ in $\partial \Omega$ and
\begin{align*}
U = \{ u \in \Cb^d : \Real \ip{u,x-p} = 0\}
\end{align*}
then $T_{x} \partial \Omega = x + U$. Notice that $x \in B_{R+1}(0) \cap \partial \Omega$ since $\delta_{\Omega}(p) \leq \delta_0 \leq 1$.

First suppose that $\Real \ip{x-p,v}=0$.  Let $\ell : \Cb \rightarrow \Cb^d$ be the affine line with $\ell(0)=x$ and $\ell^\prime(0)=v$. If $\ell^\prime(0) \in T_{\ell(0)} \partial \Omega$ then
\begin{align*}
\sup_{\abs{z} \leq \left(C_0  \delta_{\Omega}(p) \right)^{1/L} }\abs{P_\ell(z)} \geq 4 \kappa(\Omega,R)  \delta_{\Omega}(p)
\end{align*}
since $\ell \in \Lc(\Omega, R)$.  Then since $\delta_{\Omega}(p) \leq \delta_0$ we have
\begin{align*}
\left(C_0 \delta_{\Omega}(p) \right)^{1/L} \leq \epsilon_0 \leq \epsilon(\Omega, R)
\end{align*}
and so 
\begin{align*}
\sup_{\abs{z} \leq \left(C_0 \delta_{\Omega}(p) \right)^{1/L} } \abs{P_\ell(z)-(r\circ\ell)(z)} \leq \kappa(\Omega,R) \delta_{\Omega}(p). 
\end{align*}
Then, exploiting the fact that $(r\circ \ell)(z) \geq 0$, we see that 
\begin{align*}
\sup_{\abs{z} \leq  \left(C_0 \delta_{\Omega}(p) \right)^{1/L} }  (r\circ \ell)(z) \geq 3\kappa(\Omega, R)\delta_{\Omega}(p).
\end{align*}
Thus, using the fact that $\norm{\nabla r} \leq \kappa(\Omega, R)$ and $vz + p = \ell(z)-x+p$, we have
\begin{align*}
\sup_{\abs{z} \leq (C_0\delta_{\Omega}(p))^{1/L}} r(vz+p) \geq 3\kappa(\Omega,R)\delta_{\Omega}(p)-\kappa(\Omega,R) \norm{p-x} = 2\kappa(\Omega,R) \delta_\Omega(p) >0.
\end{align*}
Notice that we need the assumption that 
\begin{align*}
\delta_{\Omega}(p) \leq \delta_0 \leq \min\{ 1, \epsilon_0^{L}/C_0\} \leq \min\{ 1, 1/C_0\}
\end{align*}
in order to apply the gradient estimate. Since $\Omega = \{ r(z) < 0\}$ we then have
\begin{equation*}
\delta_{\Omega}(p;v) \leq    C_0^{1/L} \delta_{\Omega}(p)^{1/L}.
\end{equation*}

Now we prove the general case. Decompose 
\begin{align*}
v = \beta \frac{p-x}{\delta_{\Omega}(p)} + \sqrt{1-\beta^2} u 
\end{align*}
where $\beta = \ip{p-x, v}/\delta_{\Omega}(p)$ and $u \in U$ is a unit vector. Since 
\begin{align*}
(-\delta_{\Omega}(p)/\beta)v+p \in T_{x} \partial \Omega \subset \Cb^d \setminus \Omega
\end{align*}
 we see that 
\begin{align}
\label{eq:ft_mc_1}
\delta_{\Omega}(p;v) \leq \delta_{\Omega}(p)/\abs{\beta}.
\end{align}
Now if 
\begin{align*}
\abs{\beta} \geq \frac{3\delta_{\Omega}(p)^{1-1/L}}{2C_0^{1/L}} 
\end{align*}
then Equation~\ref{eq:ft_mc_1} implies that 
\begin{align*}
\delta_{\Omega}(p;v) \leq \frac{2}{3}C_0^{1/L} \delta_{\Omega}(p)^{1/L}.
\end{align*}
Otherwise, since $\delta_{\Omega}(p) \leq \delta_0 \leq (C_0/3^{L/2})^{1/(L-1)}$ we see that $\abs{\beta} \leq \sqrt{3}/2$. Then since $u \in T_{x} \partial \Omega$, the special case above implies that
\begin{align*}
\sup_{ \abs{z} \leq \frac{1}{\sqrt{1-\beta^2}}  (C_0  \delta_{\Omega}(p))^{1/L}} & r\Big( u(\sqrt{1-\beta^2}z) + p\Big) \\
& = \sup_{ \abs{z} \leq  (C_0 \delta_{\Omega}(p) )^{1/L}} r\Big( uz + p\Big)
\geq 3\kappa(\Omega, R)  \delta_{\Omega}(p).
\end{align*}
And so, using the fact that $\norm{\nabla r} \leq \kappa(\Omega, R)$, we have
\begin{align*}
\sup_{ \abs{z} \leq \frac{1}{\sqrt{1-\beta^2}}  (C_0  \delta_{\Omega}(p))^{1/L}} & r(vz+p) \geq 3\kappa(\Omega, R)  \delta_{\Omega}(p) -  \frac{\kappa(\Omega,R)\abs{\beta}}{\sqrt{1-\beta^2}}  \left(C_0  \delta_{\Omega}(p) \right)^{1/L} \\
&  \geq 3\kappa(\Omega,R)  \delta_{\Omega}(p)-  2\kappa(\Omega,R) \abs{\beta}\left(C_0  \delta_{\Omega}(p)\right)^{1/L}.
\end{align*}
Notice that we need the assumption that 
\begin{align*}
\delta_{\Omega}(p) \leq \delta_0 \leq \min\{ 1, \epsilon_0^{L}/C_0\} \leq \min\{ 1, 1/C_0\}
\end{align*}
and 
\begin{align*}
\frac{1}{\sqrt{1-\beta^2}} \leq 2
\end{align*}
in order to apply the gradient estimate. Thus
\begin{align*}
\sup_{ \abs{z} \leq \frac{1}{\sqrt{1-\beta^2}}  (C_0 \delta_{\Omega}(p) )^{1/L}} r(vz+p) \geq 0
\end{align*}
and so
\begin{equation*}
\delta_{\Omega}(p;v) \leq \frac{1}{\sqrt{1-\beta^2}}C_0^{1/L}\delta_{\Omega}(p)^{1/L} \leq 2C_0^{1/L}\delta_{\Omega}(p)^{1/L}. \qedhere
\end{equation*}
\end{proof}

We end this section with an an example:

\begin{example}\label{ex:L_convex_seq}
Suppose that $\Omega = \{ r(z) < 0\}$ is $\Cb$-proper convex open set, $\partial \Omega$ is $C^L$ and every $x \in \partial \Omega$ has finite line type at most $L$. Then by Proposition~\ref{prop:finite_type}, $\Omega$ is locally $L$-convex. If $\Omega_n = \{ r_n(z) < 0\}$ is a sequence of $\Cb$-proper convex open sets such that $r_n$ converges to $r$ locally uniformly in the $C^0$ topology then 
\begin{align*}
\lim_{n \rightarrow \infty} D(\Omega_n,R) = D(\Omega, R).
\end{align*}
If $r_n$ converges to $r$ locally uniformly in the $C^1$ topology then 
\begin{align*}
\lim_{n \rightarrow \infty} \kappa(\Omega_n,R) = \kappa(\Omega, R).
\end{align*}
If $r_n$ converges to $r$ locally uniformly in the $C^L$ topology then 
\begin{align*}
\lim_{n \rightarrow \infty} \alpha(\Omega_n,R) = \alpha(\Omega, R).
\end{align*}
Finally, if $r_n$ converges to $r$ locally uniformly in the $C^{L+1}$ topology then 
\begin{align*}
\lim_{n \rightarrow \infty} \epsilon(\Omega_n,R) = \epsilon(\Omega, R).
\end{align*}
In particular, if  $r_n$ converges to $r$ locally uniformly in the $C^{L+1}$ topology then Proposition~\ref{prop:finite_type_2} implies that $(\Omega_n)_{n \in \Nb}$ is a locally $L$-convex sequence. 
\end{example}

\section{Rescaling convex domains of finite type}

In this section we recall an argument of Gaussier~\cite{G1997} which implies the following:

\begin{theorem}\label{thm:gaussier}
Suppose $\Omega \subset \Cb^{d+1}$ is a convex open set such that $\partial \Omega$ is $C^L$ and has finite line type $L$ near some $\xi \in \partial \Omega$. If $q_n \in \Omega$ is a sequence converging to $\xi$ then there exists $n_k \rightarrow \infty$ and affine maps $A_k \in \Aff(\Cb^d)$ such that
\begin{enumerate}
\item $A_k \Omega$ converges in the local Hausdorff topology to a $\Cb$-proper convex open set $\wh{\Omega}$ of the form:
\begin{align*}
\wh{\Omega} = \{ (z_0,z_1 \dots, z_d) \in \Cb^{d} : \Real(z_0)  > P(z_1,z_2, \dots, z_d) \}
\end{align*}
where $P$ is a non-negative non-degenerate convex polynomial with $P(0)=0$,
\item $A_k q_{n_k} \rightarrow q_{\infty} \in \wh{\Omega}$, and
\item $(A_k \Omega)_{k \in \Nb}$ is a locally L-convex sequence.
\end{enumerate}
\end{theorem}

\begin{remark} \

\begin{enumerate}
\item A polynomial is called \emph{non-degenerate} if the set $\{ P=0\}$ contains no complex affine lines. In the context of the above theorem this is equivalent to the set $\wh{\Omega}$ being $\Cb$-proper.
\item Following the notation in~\cite{G1997}, in this section we consider convex sets in $\Cb^{d+1}$. 
\item There is a small issue in Gaussier's proof that needs to be corrected: Proposition 1.3 part (iii) in~\cite{M1994} was used which was later pointed out to be incorrect in~\cite{NPT2013}. This problem is easily fixed by using the ``minimal'' bases introduced by Hefer~\cite{H2002} instead of the ``maximal'' bases used in the original proof. 
\item Gaussier also considers the case in which $\partial \Omega$ is $C^\infty$ near $\xi$. Weakening the regularity requires some additional estimates and is our main motivation for presenting the complete argument. 
\end{enumerate} 
\end{remark}

Now suppose $\Omega$ satisfies the hypothesis of Theorem~\ref{thm:gaussier}. Then there exists a bounded neighborhood $V$ of $\xi$ in $\Cb^{d+1}$ such that $\Omega \cap V$ is defined by a convex function of the form 
\begin{align*}
r(z_0,z_1, \dots, z_d) = \Real(z_0) + \varphi(\Imaginary(z_0), z_1, \dots, z_d)
\end{align*}
where $\varphi$ is $C^L$, $\norm{\nabla r}$ is bounded from above and below on $V$, and $\Omega \cap V$ has finite line type $L$. 

Given a point $q \in \Omega$ we will associate points $x_0(p), \dots, x_d(p) \in \partial \Omega$. First let $x_0(p)$ be a point in $\partial \Omega$ closest to $p$. Then assuming $x_0(q), \dots, x_k(q)$ have already been selected, let $P$ be the maximal complex plane through $q$ orthogonal to the lines $\{ \overline{q x_i(q)}$ : $0 \leq i \leq k\}$. Then let $x_{k+1}(q)$ be the point in $P \cap \partial \Omega$ closest to $q$. 

\begin{observation} There exists a neighborhood $U$ of $\xi$ so that for all $q \in U \cap \Omega$ the points $x_0(q), \dots, x_d(q)$ are in  $V$. 
\end{observation}

Once $x_0(q), \dots, x_d(q)$ have been selected let $\tau_i(q) = \norm{q-x_i(q)}$ for $0 \leq i \leq d$. Next let $T_q \in \Aff(\Cb^{d+1})$ be the translation $T_q(z) = z-q$, let $\Lambda_q \in \GL_{d+1}(\Cb)$ be the linear map 
\begin{align*}
\begin{pmatrix}
\tau_0(q)^{-1} & & \\
& \ddots & \\
& & \tau_{d+1}(q)^{-1}
\end{pmatrix},
\end{align*}
and let $U_q$ be the unitary map so that 
\begin{align*}
\Lambda_q U_q T_q(x_i(q)) = e_i.
\end{align*}
Next define the polydisk  
\begin{align*}
P(q): = \{ z \in \Cb^{d+1} : \abs{z_0} < \tau_0(q), \dots, \abs{z_d} < \tau_d(q) \}
\end{align*}
and a new function
\begin{align*}
r_q: = r \circ (U_q \circ T_q)^{-1}.
\end{align*}
Notice that $r_q$ is a defining function for the convex set $(U_q \circ T_q)(\Omega \cap V)$. Also for $0 \leq i \leq d$ let 
\begin{align*}
y_i(q) = (U_q \circ T_q)x_i(q).
\end{align*}

Then

\begin{lemma}\cite[Lemma 1.1, Lemma 1.2]{G1997}\label{lem:G1}
There exists $C>0$ such that for all $q \in \Omega \cap U$ 
\begin{align*}
CP(q) \subset \{ z \in \Cb^{d+1} : 2r(q) < r_q(z) < 0\}.
\end{align*}
\end{lemma}

\begin{proof} 
We can pick $C >0$ sufficiently small so that $CP(q) \subset (U_q \circ T_q)(\Omega \cap V)$. Then $r_q(z) < 0$ for all $z \in CP(q)$. 

Now if $z \in CP(q)$ we see that $-z \in CP(q)$. So suppose that $z \in CP(q)$ and $r_q(z) < 2r(q)$. Then the function $f(t) = r_q(-tz)$ is convex and by the mean value theorem there exists $t_0 \in [-1,0]$ so that 
\begin{align*}
f^\prime(t_0)(0+1) = r(q) - r(z) > -r(q).
\end{align*}
But then 
\begin{align*}
f(1) = f(0) + \int_{0}^{1} f^\prime(t) dt > r(q) - r(q) = 0
\end{align*}
by convexity. Which contradicts the fact that $-z \in CP(q) \subset \Omega \cap V$.
\end{proof}

\begin{lemma}\cite[Lemma 1.3]{G1997}\label{lem:G2}
There exists $c>0$ such that for all $q \in \Omega \cap U$ 
\begin{enumerate}
\item $\tau_0(q) \leq c(-r(q))$, 
\item for every $j \geq 1$, $\tau_j(q) \leq c(-r(q))^{1/L}$.
\end{enumerate}
\end{lemma}

\begin{proof} Part (1) follows from the fact that $\nabla r$ is bounded from below on $\Omega \cap V$. By Proposition~\ref{prop:finite_type_2} there exists $C>0$ so that 
\begin{align*}
\delta_{\Omega}(q;v) \leq C \delta_{\Omega}(q)^{1/L}
\end{align*}
for $q \in U$ and $v \in \Cb^{d+1}$ nonzero. Part (2) follows from this and the fact that $\nabla r$ is bounded from below.
\end{proof}

\begin{lemma}\cite[Lemma 1.4]{G1997}\label{lem:G3} For all $q \in \Omega \cap U$
\ \begin{enumerate}
\item For every $0 \leq i \leq d$, $\frac{\partial r_q}{\partial z_i}(y_i(q))$ is real,
\item there exists $c^\prime >0$ such that for all $0 \leq j \leq d$ 
\begin{align*}
\abs{\frac{\partial r_q}{\partial z_j}(y_j(q))} \geq c^\prime \frac{\tau_0(q)}{\tau_j(q)},
\end{align*}
\item if $0 \leq i < j \leq d$ then 
\begin{align*}
\frac{\partial r_q}{\partial z_j}(y_i(q)) = 0.
\end{align*}
\end{enumerate}
\end{lemma}

\begin{proof} Part (1) and part (3) follow from the fact that $x_i(q)$ is the closest point to $q$ in the complex plane spanned by $x_i(p), x_{i+1}(p), \dots, x_d(p)$. By construction $r_q(0)=r(q)$ and $r_q(y_i(q)) = 0$. Thus, by the mean value theorem, there exists some $t_0 \in [0,1]$ so that 
\begin{align*}
\left. \frac{d}{dt}\right|_{t=t_0} r_q(t y_i(q)) = -r(q)
\end{align*}
Then by convexity 
\begin{align*}
\left. \frac{d}{dt}\right|_{t=1} r_q(t y_i(q)) \geq -r(q)
\end{align*}
which implies
\begin{align*}
\frac{\partial r_q}{\partial z_i}(y_i(q)) \geq -\frac{r(q)}{\tau_i(q)}.
\end{align*}
Then part (2) follows from Lemma~\ref{lem:G2}.
\end{proof}

Now suppose that $q_n \rightarrow \xi$. By passing to a tail of the sequence we can suppose that $q_n \in U$ for all $n \in \Nb$. Let $\epsilon_n := -r(q_n)$, $U_n := U_{q_n}$, $T_n:=T_{q_n}$, $\Lambda_n:=\Lambda_{q_n}$, and $r_n := r_{q_n}$. Then 
\begin{align*}
r_n(z) = -\epsilon_n + \Real\left(\sum_{j=0}^d a_j^n z_j\right) + \sum_{2 \leq \abs{\alpha}+\abs{\beta} \leq L} C_{\alpha,\beta}^n z^{\alpha}\overline{z}^{\beta}+E_n(z)
\end{align*}
where $E_n$ is the error term in Taylor's formula. Since $q_n\in \Omega \cap V$, a bounded set, for any multi-indices $\alpha, \beta$ with $\abs{\alpha}+\abs{\beta} =k \leq L$ 
\begin{align}
\label{eq:5}
\lim_{z \rightarrow 0} \left( \frac{ 1}{\norm{z}^{L-k}}\frac{\partial^kE_n}{\partial z^{\alpha}\partial \overline{z}^\beta} (z) \right) =0.
\end{align}
Since each $U_n$ is unitary we may pass to a subsequence so that $U_n \rightarrow U$. Moreover, $T_n \rightarrow T$ where $T(z) = z-\xi$. Then $r_n$ converges in the $C^L$ topology to $r \circ (U \circ T)^{-1}$ and in particular
\begin{align*}
 \lim_{n \rightarrow \infty} a_j^n \text{ and }  \lim_{n \rightarrow \infty} C_{\alpha,\beta}^n
\end{align*}
exist for $0 \leq j \leq d$ and $2 \leq \abs{\alpha}+\abs{\beta} \leq L$.

Now
\begin{align*}
\wt{r}_n := \frac{1}{\epsilon_n} r_n \circ \Lambda_n^{-1} = \frac{1}{\epsilon_n} r_n \circ(\Lambda_n \circ U_n \circ T_n)^{-1}
\end{align*}
is a defining function for the domain $(\Lambda_n \circ U_n \circ T_n)(\Omega \cap V)$ and
\begin{align*}
\wt{r}_n(z) = 1+\frac{1}{\epsilon_n} \Real\left(\sum_{j=0}^d a_j^n \tau_{j}(q_n) z_j\right) + \frac{1}{\epsilon_n} \sum_{2 \leq \abs{\alpha}+\abs{\beta} \leq L} C_{\alpha,\beta}^n \tau(q_n)^{\alpha+\beta} z^{\alpha}\overline{z}^{\beta}+\frac{1}{\epsilon_n}E_n(\Lambda_n^{-1} z)
\end{align*}
where $\tau(q_n)^{\alpha+\beta} = \prod \tau_{i}(q_n)^{\alpha_i+\beta_i}$. 

\begin{proposition}\cite[Proposition 2.1, Lemma 3.1]{G1997}
The functions $\wt{r}_n$ are $C^L$ and convex. Moreover there exists a subsequence of $(\wt{r}_n)_{n \in \Nb}$ that converges locally uniformly in the $C^L$ topology to a smooth convex function $\wt{r}$ of the form 
\begin{align*}
\wt{r}(z) = -1 + \Real\left(\sum_{j=0}^d b_j z_j\right) + P(z_1, \dots, z_d)
\end{align*}
where $b_0 \neq 0$ and $P$ is a non-degenerate convex polynomial with $P(0)=0$.
\end{proposition}

\begin{proof}
Each $\wt{r}_n$ is clearly $C^L$ and convex. 

For multi-indices $\alpha, \beta$ with $\abs{\alpha}+\abs{\beta}=k \leq L$ we have
\begin{align*}
\abs{ \frac{1}{\epsilon_n}\frac{\partial^k }{\partial z^\alpha\partial\overline{z}^\beta}  \Big(E_n(\Lambda_n^{-1} z)\Big) } = \abs{\frac{\tau(q_n)^{\alpha+\beta}}{\epsilon_n}\frac{\partial^k E_n}{\partial z^\alpha\partial\overline{z}^\beta}(\Lambda_n^{-1} z)}
  \leq\frac{c^k}{\epsilon_n^{1-k/L}}\abs{\frac{\partial^k E_n}{\partial z^\alpha\partial\overline{z}^\beta}(\Lambda_n^{-1} z)}
\end{align*} 
by Lemma~\ref{lem:G2}. Using Lemma~\ref{lem:G2} again we have that 
\begin{align*}
\norm{\Lambda_n^{-1} z}^{L-k} \leq \left( c\epsilon_n^{1/L} \norm{z} \right)^{L-k} = c^{L-k}\epsilon_n^{1-k/L} \norm{z}^{L-k}
\end{align*}
and hence
\begin{align*}
\abs{ \frac{1}{\epsilon_n}\frac{\partial^k }{\partial z^\alpha\partial\overline{z}^\beta}\Big(E_n(\Lambda_n^{-1} z)\Big)} \leq c^{L} \norm{z}^{L-k} \abs{\frac{1}{\norm{\Lambda_n^{-1} z}^{L-k}}\frac{\partial^k E_n}{\partial z^\alpha\partial\overline{z}^\beta}(\Lambda_n^{-1} z)}.
\end{align*}
Then for any $R>0$ 
\begin{align*}
\lim_{n \rightarrow \infty} \sup_{\abs{z} \leq R} \abs{ \frac{1}{\epsilon_n}\frac{\partial^k }{\partial z^\alpha\partial\overline{z}^\beta}\Big(E_n(\Lambda_n z)\Big)} = 0
\end{align*}
by Equation~\ref{eq:5}. This implies that $\wt{r}_n$ converges locally uniformly in the $C^L$ topology if and only if the polynomial 
\begin{align*}
 1+\frac{1}{\epsilon_n} \Real\left(\sum_{j=0}^d a_j^n \tau_j(q_n) z_j\right) + \frac{1}{\epsilon_n} \sum_{2 \leq \abs{\alpha}+\abs{\beta} \leq L} C_{\alpha,\beta}^n \tau(q_n)^{\alpha+\beta} z^{\alpha}\overline{z}^{\beta}
 \end{align*}
 converges locally uniformly in the $C^L$ topology (which is equivalent to the coefficients converging).
 
 Now since every norm on a finite dimensional vector space is equivalent we see that there exists $d_1 >0$ such that 
 \begin{align*}
 \sup_{j, \alpha, \beta}  & \left\{ \abs{a_j^n}\tau_j(q_n),  \abs{C_{\alpha,\beta}^n}\tau(q_n)^{\alpha+\beta} \right\} \\
&  \leq d_1 \sup_{\abs{w} \leq C} \abs{  \Real\left(\sum_{j=0}^d a_j^n \tau_j(q_n) w_j\right) + \sum_{2 \leq \abs{\alpha}+\abs{\beta} \leq L} C_{\alpha,\beta}^n \tau(q_n)^{\alpha+\beta} w^{\alpha}\overline{w}^{\beta}} \\
& \leq  d_1\sup_{z \in CP(q_n)} \abs{  \Real\left(\sum_{j=0}^d a_j^n z_j\right) + \sum_{2 \leq \abs{\alpha}+\abs{\beta} \leq L} C_{\alpha,\beta}^n z^{\alpha}\overline{z}^{\beta}}
 \end{align*}
 Now by Lemma~\ref{lem:G1}
 \begin{align*}
 \sup_{z \in CP(q_n)} \abs{r_n(z)} \leq 2 \epsilon_n.
 \end{align*}
 Using Equation~\ref{eq:5} and Lemma~\ref{lem:G2} we can pick an $N>0$ such that 
  \begin{align*}
\abs{E_n(z)} \leq 2 \norm{z}^L 
 \end{align*}
 for all $n>N$ and $z \in CP(q_n)$. Which implies by Lemma~\ref{lem:G2} that
  \begin{align*}
 \sup_{z \in CP(q_n)} \abs{E_n(z)} \leq d_2 \epsilon_n
 \end{align*} 
 for some $d_2>0$. Then 
 \begin{align*}
 \sup_{z \in CP(q_n)} & \abs{  \Real\left(\sum_{j=0}^d a_j^n z_j\right) + \sum_{2 \leq \abs{\alpha}+\abs{\beta} \leq L} C_{\alpha,\beta}^n z^{\alpha}\overline{z}^{\beta}}\\
 & = \sup_{z \in CP(q_n)} \abs{ r_n(z) - E_n(z)} \\
 &  \leq  \sup_{z \in CP(q_n)} \abs{r_n(z)} +  \sup_{z \in CP(q_n)} \abs{E_n(z)}  \leq (2+d_2)\epsilon_n.
 \end{align*}
So 
\begin{align*}
 \sup_{j, \alpha, \beta}  \left\{ \abs{a_j^n}\tau_{j}(q_n), \abs{C_{\alpha,\beta}^n}\tau(q_n)^{\alpha+\beta} \right\} \leq d_1(2+d_2)\epsilon_n.
 \end{align*}
 Thus we can pass to a subsequence such that $(\wt{r}_n)$ converges locally uniformly in the $C^L$ topology to a function of the form 
 \begin{align*}
\wt{r}(z) = -1 + \Real\left(\sum_{j=0}^d b_j z_j\right) + P(z_0, z_1, \dots, z_d).
\end{align*}
Since $\wt{r}$ is the limit of convex functions, it is convex. 

We next claim that $P$ does not depend on $z_0$. By Lemma~\ref{lem:G2}
\begin{align*}
\tau(q_n)^{\alpha+\beta} \leq c^{\abs{\alpha}+\abs{\beta}} \epsilon_n^{\alpha_0+\beta_0} \epsilon_n^{\frac{1}{L}(\abs{\alpha}+\abs{\beta}-\alpha_0-\beta_0)}.
\end{align*}
So if $\abs{\alpha}+\abs{\beta} \geq 2$ and $(\alpha_0,\beta_0) \neq (0,0)$ then
\begin{align*}
\lim_{n \rightarrow \infty} \frac{1}{\epsilon_n} C_{\alpha,\beta}^n \tau(q_n)^{\alpha+\beta} = 0.
\end{align*}
Thus $P$ does not depend on $z_0$. 

It remains to show that $b_0 \neq 0$ and $P$ is non-degenerate. This is equivalent to showing the convex  set
\begin{align*}
\Omega_{\infty}:=\{ z \in \Cb^{d+1} : \wt{r}(z) < 0\}
\end{align*}
is $\Cb$-proper. Let $e_0, \dots, e_{d+1}$ be the standard basis in $\Cb^{d+1}$. Since $\wt{r}_n(e_0)=0$ for all $n$ we see that $\wt{r}(e_0)=0$. Thus $b_0 \neq 0$. 

Now $\partial \Omega_{\infty}$ is a $C^\infty$ hypersurface since $b_0 \neq 0$. The tangent plane at a point $x \in \partial\Omega_{\infty}$ is given by 
\begin{align*}
T_x \partial \Omega_{\infty} = \left\{ z \in \Cb^{d+1} : \Real\left( \sum_{i=0}^d \frac{\partial \wt{r}}{\partial z_i}(x)\overline{z_i} \right) = 0\right\}.
\end{align*}
Since $\Omega_\infty$ is convex, $T_x \partial \Omega_\infty \cap \Omega_\infty = \emptyset$ for all $x \in \partial \Omega$. 

Notice that $(\Lambda_n U_n T_n)(x_j(q_n)) = e_j$ and so $e_j \in \partial \Omega$ for $0 \leq j \leq d$.   Then by Lemma~\ref{lem:G3}
\begin{align*}
\frac{\partial \wt{r}}{\partial z_k}(e_j) = \lim_{n \rightarrow \infty} \frac{\partial \wt{r}_n}{\partial z_k}(e_j)=0
\end{align*}
for $k >j$ and
\begin{align*}
\frac{\partial \wt{r}}{\partial z_k}(e_k) = \lim_{n \rightarrow \infty} \frac{\partial \wt{r}_n}{\partial z_k}(e_k)
\end{align*}
is nonzero. In particular, if 
\begin{align*}
v_k := \left( \frac{\partial \wt{r}}{\partial z_0}(e_k), \dots, \frac{\partial \wt{r}}{\partial z_d}(e_k) \right)^t
\end{align*}
then $v_0, \dots, v_d$ is a basis of $\Cb^{d+1}$ and 
\begin{align*}
T_{e_k} \partial \Omega_{\infty} = \{ z \in \Cb^{d+1} : \Real( \overline{z}^tv_k ) =0\}.
\end{align*}
Since $\Omega_{\infty}$ is convex, there exists $\delta_k \in \{-1,1\}$ such that 
 \begin{align*}
 \Omega_{\infty} \subset \{ z \in \Cb^{d+1} : \delta_k  \Real( \overline{z}^t v_k ) >0\}.
 \end{align*}
 Since $v_0, \dots, v_d$ is a basis of $\Cb^{d+1}$ this implies that $\Omega_{\infty}$ is $\Cb$-proper.
 \end{proof}
 
We can now complete the proof of Theorem~\ref{thm:gaussier}. Let $A_n:= (\Lambda_nU_nT_n)$ and $\Omega_n := A_n\Omega$. By passing to a subsequence we can assume that  $\wt{r}_n$ converges to $\wt{r}$ locally uniformly in the $C^L$ topology.

We first claim that $\Omega_n$ converges to $\Omega_{\infty}:=\{ \wt{r}(z) < 0\}$ in the local Hausdorff topology. This follows from the fact that for any $R>0$ there exists $N$ such that 
\begin{align*}
B_R(0) \cap S_n\Omega =B_R(0) \cap A_n(\Omega \cap V)
\end{align*}
for all $n >N$. 

We next claim that $(\Omega_n)_{n \in \Nb}$ is a locally $L$-convex sequence. Using the notation in Section~\ref{sec:finite_type_m_conex}, since $\wt{r}_n$ converges to $\wt{r}$ locally uniformly in the $C^L$ topology we see that 
\begin{align*}
\lim_{n \rightarrow \infty} \kappa(\Omega_n,R) = \kappa(\Omega_\infty, R)< \infty,
\end{align*}
\begin{align*}
\lim_{n \rightarrow \infty} D(\Omega_n,R) = D(\Omega_\infty, R)< \infty,
\end{align*}
and
\begin{align*}
\lim_{n \rightarrow \infty} \alpha(\Omega_n,R) = \alpha(\Omega_\infty, R)>0. 
\end{align*}
So by passing to a subsequence we can suppose that $\kappa(\Omega_n,R) \leq \kappa_0$, $D(\Omega_n,R) \leq D_0$, and $\alpha(\Omega_n, R) \geq \alpha_0$ for some $\kappa_0, D_0, \alpha_0>0$. It remains to show that $\epsilon(\Omega_n, R)$ is uniformly bounded from below for large $n$. 

Let $\Lc$ be set of affine lines $\ell: \Cb \rightarrow \Cb^d$ such that $\ell(0) \in V \cap \partial \Omega$, $\ell^\prime(0) \in T_{\ell(0)} \partial \Omega$, and $\norm{\ell^\prime(0)} =1$. For $\ell \in \Lc$ let $P_{\ell}$ be the $L^{th}$ order Taylor polynomial of $(r \circ \ell)(z)$. Let $c >0$ be as in Lemma~\ref{lem:G2}. Since $\partial \Omega$ is $C^L$, there exists $\epsilon_0 > 0$ such that 
\begin{align*}
\abs{ P_{\ell}(z) - (r \circ \ell)(z) } \leq \frac{ A_L \alpha_0}{4 c^L} \abs{z}^{L}
\end{align*}
for all $\ell \in \Lc$ and $\abs{z} \leq \epsilon_0$. Fix $R>0$ and let $\Lc(\Omega_n, R)$ be the the set of affine lines $\ell: \Cb \rightarrow \Cb^d$ such that $\ell(0) \in B_{R+1}(0) \cap \partial \Omega_n$, $\ell^\prime(0) \in T_{\ell(0)} \partial \Omega_n$, and $\norm{\ell^\prime(0)} =1$. For $\ell \in \Lc(\Omega_n,R)$ let $P_{n,\ell}$ be the $L^{th}$ order Taylor polynomial of $(\wt{r}_n \circ \ell)(z)$. 

Now given $\ell \in \Lc(\Omega_n, R)$ let $\wh{\ell}$ be the affine line 
\begin{align*}
\wh{\ell}(z) = A_n^{-1} \ell\left( \frac{1}{\norm{d(A_n)^{-1} \ell^\prime(0)}}z\right).
\end{align*}
Then 
\begin{align*}
\abs{ P_{n,\ell}(z) - (\wt{r}_n \circ \ell)(z) } = \frac{1}{\epsilon_n} \abs{ P_{\wh{\ell}}\left(\norm{d(A_n)^{-1}\ell^\prime(0)}z\right) - (r \circ \wh{\ell})\left(\norm{d(A_n)^{-1}\ell^\prime(0)}z\right)  }.
\end{align*}
Notice that $\norm{d(A_n)^{-1}} \leq c \epsilon_n^{1/L}$ by Lemma~\ref{lem:G2}. For $n$ large, $\wh{\ell} \in \Lc$ and $c \epsilon_n^{1/L} \leq \epsilon_0$ so for $\abs{z} < 1$ we have 
\begin{align*}
\abs{ P_{n,\ell}(z) - (\wt{r}_n \circ \ell)(z) } \leq \frac{ A_L \alpha_0}{4 c^L  \epsilon_n} \abs{\norm{d(A_n)^{-1} \ell^\prime(0)}z }^{L} 
\leq \frac{ A_L \alpha(\Omega_n, R)}{4} \abs{z }^{L}
\end{align*}
Thus for $n$ large $\epsilon_n(\Omega, R) \geq 1$. This implies, by Proposition~\ref{prop:finite_type_2}, that $(\Omega_n)_{n \in \Nb}$ is a locally $L$-convex sequence. 

Finally since $b_0 \neq 0$ we can make a linear change of coordinates such that 
\begin{align*}
\Omega_{\infty} = \{ (z_0, \dots, z_d) \in \Cb^{d+1} : \Real(z_0) > P(z_1, \dots, z_d)\}
\end{align*}
and $P$ is a non-negative non-degenerate convex polynomial with $P(0)=0$.

\section{Geodesics and the Gromov product in convex domains}\label{sec:gromov_prod}

The primary goal of the next two sections is to show that the polynomial domains produced by Gaussier's theorem have well behaved geodesics. In this section we investigate the asymptotic properties of geodesics and the Gromov product on general convex sets. In the next section we will specialize to polynomial domains. 

In some of the arguments that follow we will need to know that certain lines are not just quasi-geodesics (as guaranteed by Lemma~\ref{lem:quasi_geodesic}), but have Lipschitz factor one. Suppose $\Omega$ is an open set with $C^1$ boundary. If $x \in \partial \Omega$ let $n_x$ be the inward pointing normal vector at $x$. 

\begin{proposition}\label{prop:rough_geodesics}
Suppose $\Omega$ is a $\Cb$-proper convex set with $C^2$ boundary in a neighborhood $\Oc$ of some $y \in \partial \Omega$. If $x \in \Oc \cap \partial \Omega$ then there exists $\epsilon=\epsilon(x) >0$ such that the curve $\sigma: \Rb_{\geq 0} \rightarrow \Omega$ given by
\begin{align*}
\sigma(t)=x+e^{-2t}\epsilon n_x
\end{align*}
is an $(1,\log\sqrt{2})$-quasi-geodesic in $(\Omega, d_{\Omega})$. Moreover, we can choose $\epsilon$ to depend continuously on $x \in \Oc \cap \partial \Omega$.
\end{proposition}

\begin{remark} \
 If $A \in \Aff(\Cb^d)$ then $A$ induces an isometry between $(\Omega, d_{\Omega})$ and $(A\Omega, d_{A\Omega})$. In particular, Proposition~\ref{prop:rough_geodesics} actually implies that many additional real lines can be parametrized as $(1,K)$-quasi-geodesics. 
\end{remark}

\begin{proof}
By translating and rotating, we may assume that $x = 0$, 
\begin{align*}
T_0 \partial \Omega = \{ (z_1, \dots, z_d) \in \Cb^d : \Imaginary(z_1)=0\},
\end{align*}
and $\Omega \subset \{ (z_1, \dots, z_d) \in \Cb^d : \Imaginary(z_1)> 0\}$. With this normalization $n_x = (i,0,\dots, 0)$. 

From Lemma~\ref{lem:convex_lower_bd_2} we obtain:
\begin{align*}
d_{\Omega}(\sigma(t_1),\sigma(t_2)) \geq \abs{t_2-t_1}.
\end{align*}

Using the formulas in Section~\ref{sec:prelim}
\begin{align*}
d_{\Delta}(0,1-e^{-2t}) = \frac{1}{2} \log \frac{ 2-e^{-2t}}{e^{-2t}} = t + \frac{1}{2} \log (2-e^{-2t})
\end{align*}
for $t \geq 0$. So
\begin{align*}
t \leq d_{\Delta}(0,1-e^{-2t}) \leq t + \log\sqrt{2}
\end{align*}
for $t \geq 0$. Since $(-1,1)$ can be parametrized as a geodesic in $\Delta$, for $t_2 \geq t_1 \geq 0$ we have 
\begin{align*}
d_{\Delta}(1-e^{-2t_2},1-e^{-2t_1}) 
&= d_{\Delta}(1-e^{-2t_2},0)- d_{\Delta}(1-e^{-2t_1},0) \\
& \leq t_2-t_1 + \log\sqrt{2} \\
& = \abs{t_2-t_1}+\log\sqrt{2}.
\end{align*}
For any $\epsilon >0$ the affine map $z \in \Cb \rightarrow \epsilon i (1-z) \in \Cb$ induces an isometry between $(\Delta, d_{\Delta})$ and $(B_{\epsilon}(\epsilon i), d_{B_{\epsilon}(\epsilon i)})$ and so 
\begin{align*}
d_{B_{\epsilon}(\epsilon i)}(i\epsilon e^{-2t_2},i\epsilon e^{-2t_1}) =d_{\Delta}(1-e^{-2t_2},1-e^{-2t_1}) \leq \abs{t_2-t_1} + \log\sqrt{2}.
\end{align*}
Finally since $\partial \Omega$ is $C^2$ near $y$ there exists $\epsilon >0$ such that $B_{\epsilon}(\epsilon i) \subset \Omega$. Then
\begin{align*}
d_{\Omega}(\sigma(t_1),\sigma(t_2)) \leq d_{B_{\epsilon}(\epsilon i)}(i\epsilon e^{-2t_1},i\epsilon e^{-2t_2}) \leq \abs{t_1-t_2} + \log\sqrt{2}.
\end{align*}
Notice that the maximal such $\epsilon >0$ depends continuously on $x$.
\end{proof}

Recall that the Gromov product is defined to be
\begin{align*}
(p|q)_o = \frac{1}{2} \left(d_{\Omega}(p,o)+d_{\Omega}(o,q)-d_{\Omega}(p,q) \right).
\end{align*}

\begin{proposition}\label{prop:gromov_prod_finite}
Suppose $\Omega$ is a $\Cb$-proper convex open set. Assume $p_n, q_n \in \Omega$ are sequences with $\lim_{n \rightarrow \infty} p_n = \xi^+ \in \partial \Omega$, $\lim_{n \rightarrow \infty} q_n = \xi^- \in \partial \Omega \cup \{ \infty\}$, and 
\begin{align*}
\liminf_{n,m \rightarrow \infty} (p_n | q_m)_o < \infty.
\end{align*}
If $\partial \Omega$ is $C^2$ near $\xi^+$ then $\xi^+ \neq \xi^-$.
\end{proposition} 

\begin{proof}
By passing to subsequences we may assume that $\lim_{n \rightarrow \infty} (p_n | q_n)_o$ is finite. Assume for a contradiction that $\xi^+=\xi^- \in \partial \Omega$. 

Let $\xi_n^+ \in \partial \Omega$ be a point in the boundary closest to $p_n$ and $\xi_n^- \in \partial \Omega$ be a point in the boundary closest to $q_n$. Then $\lim_{n \rightarrow \infty} \xi_n^{\pm} = \xi^+$. By Proposition~\ref{prop:rough_geodesics} there exists $\epsilon>0$ such that for any $y$ near $\xi^+$ the curve $\sigma_y: \Rb_{\geq 0} \rightarrow \Omega$ given by
\begin{align*}
\sigma_y(t)=y+e^{-2t}\epsilon n_y
\end{align*}
is an $(1,\log\sqrt{2})$-quasi-geodesic in $(\Omega, d_{\Omega})$. Moreover, for $n$ large $\norm{p_n- \xi_n^+} < \epsilon$ and $\norm{q_n- \xi_n^-} < \epsilon$. So for large $n$, $p_n \in \sigma_{\xi_n^+}$ and $q_n \in \sigma_{\xi_n^-}$. There also exists $R>0$ such that 
\begin{align*}
d_{\Omega}(\sigma_y(0), o) \leq R
\end{align*}
for all $y \in \partial \Omega$ near $\xi^+$.

Now fix $T > 0$ then
\begin{align*}
d_{\Omega}(o,p_n) 
&\geq d_{\Omega}(\sigma_{\xi_n^+}(0),  p_n) -d_{\Omega}(o,\sigma_{\xi_n^+}(0)) \\
&  \geq d_{\Omega}(\sigma_{\xi_n^+}(0),  p_n) -R\\
& \geq d_{\Omega}(\sigma_{\xi_n^+}(0),  \sigma_{\xi_n^+}(T))+d_{\Omega}(\sigma_{\xi_n^+}(T),  p_n) - R-\log\sqrt{2}\\
& \geq d_{\Omega}(\sigma_{\xi_n^+}(T),  p_n)+T - R-2\log\sqrt{2}.
\end{align*}
The same argument shows that 
\begin{align*}
d_{\Omega}(o,q_n) \geq d_{\Omega}(\sigma_{\xi_n^-}(T),  q_n)+T - R-2\log\sqrt{2}.
\end{align*}
So 
\begin{align*}
2(p_n|q_n)_o
& = d_{\Omega}(o,p_n)+d_{\Omega}(o, q_n) - d_{\Omega}(p_n,q_n) \\
& \geq 2T - 2R-2\log(2) \\
& \quad +d_{\Omega}(\sigma_{\xi_n^+}(T),  p_n)+ d_{\Omega}(\sigma_{\xi_n^-}(T),q_n)- d_{\Omega}(p_n,q_n) \\
& \geq 2T-2R-2\log(2) - d_{\Omega}(\sigma_{\xi_n^-}(T),\sigma_{\xi_n^+}(T)).
\end{align*}
Since  $\sigma_y(T)$ depends continuously on $y$ and $\xi_n^{\pm} \rightarrow \xi^+$ we have that
\begin{align*}
d_{\Omega}(\sigma_{\xi_n^-}(T),\sigma_{\xi_n^+}(T))\rightarrow 0.
\end{align*}
Which implies that
\begin{align*}
\lim_{n \rightarrow \infty} (p_n|q_n)_o \geq T-R-\log(2).
\end{align*}
Since $T>0$ was arbitrary we have an contradiction. 
\end{proof}

\begin{corollary}\label{cor:geod_limits_distinct}
Suppose $\Omega$ is a $\Cb$-proper convex open set with $C^2$ boundary and $\sigma: \Rb \rightarrow \Omega$ is a geodesic. If $t_n, s_n \rightarrow \infty$ are two sequences such that $\lim_{n \rightarrow \infty} \sigma(t_n)=\xi^+ \in \partial \Omega$ and $\lim_{n \rightarrow \infty} \sigma(-s_n)=\xi^- \in \partial \Omega \cup \{\infty\}$ then $\xi^+ \neq \xi^-$.
\end{corollary}

\begin{proof}Note that $(\sigma(t_n)|\sigma(-s_n))_{\sigma(0)} = 0$ when $s_n$ and $t_n$ are positive. 
\end{proof}

\begin{proposition}\label{prop:gromov_prod_infinite}
Suppose $\Omega$ is a locally $m$-convex open set and $p_n, q_n \subset \Omega$ are sequences of points such that $\lim_{n \rightarrow \infty} p_n = \xi_1 \in \partial \Omega \cup \{\infty\}$ and $\lim_{ n \rightarrow \infty} q_n = \xi_2 \in \partial \Omega \cup \{\infty\}$. If
\begin{align*}
\lim_{n \rightarrow \infty} (p_n|q_n)_o = \infty
\end{align*}
for some $o \in \Omega$, then $\xi_1 = \xi_2$.
\end{proposition}

\begin{proof}
Suppose for a contradiction that $\xi_1 \neq \xi_2$. Then at least one of $\xi_1$ and $\xi_2$ is finite. Now let $\sigma_n : [0,T_n] \rightarrow \Omega$ be a geodesic such that $\sigma_n(0) = p_n$ and $\sigma_n(T_n) = q_n$.

Using Proposition~\ref{prop:m_convex} there exists $\alpha_n \in [0,T_n]$ such that $\sigma_n(t+\alpha_n)$  converges locally uniformly to a geodesic $\sigma: \Rb \rightarrow \Omega$. Now since $\sigma_n$ is a geodesic 
\begin{align*}
(p_n|q_n)_o 
&= \frac{1}{2} \left( d_{\Omega}(p_n,o)+ d_{\Omega}(o,q_n) - d_{\Omega}(p_n,q_n) \right) \\
&=  \frac{1}{2} \left( d_{\Omega}(p_n,o)+ d_{\Omega}(o,q_n) - d_{\Omega}(p_n,\sigma_n(\alpha_n)) - d_{\Omega}(\sigma_n(\alpha_n), q_n) \right)\\
& \leq d_{\Omega}(o, \sigma_n(\alpha_n)).
\end{align*}
But then
\begin{align*}
\infty =\lim_{n \rightarrow \infty} (p_n|q_n)_o  \leq  \lim_{n\rightarrow \infty}  d_{\Omega}(o, \sigma_n(\alpha_n))=  d_{\Omega}(o, \sigma(0))
\end{align*}
which is a contradiction. 
\end{proof}

\section{Multi-type at infinity}\label{sec:multi_infty}

In this section we will show that the polynomial domains appearing in Theorem~\ref{thm:gaussier} have well behaved geodesics. From Corollary~\ref{cor:geodesic_limits} we know that the forward and backward limits of a geodesic exist. And from Corollary~\ref{cor:geod_limits_distinct} we know that the backward and forward limits of a geodesic are distinct if at least one is not $\infty$. So it remains to show that a geodesic cannot have $\infty$ as a backwards and forwards limit. 

To understand the geometry at infinity we will associated to infinity a multi-type in the spirit of~\cite{C1984, Y1992}.

\begin{proposition}\label{prop:multi_type}
Suppose $P: \Cb^d \rightarrow \Rb$ is a non-negative non-degenerate convex polynomial with $P(0)=0$. Then there exists a linear change of coordinates and integers $0 < m_1 \leq \dots \leq  m_d$ such that as $t \rightarrow 0$
\begin{align*}
t P(t^{-1/m_1}z_1, \dots, t^{-1/m_d}z_d)
\end{align*}
converges in the $C^\infty$ topology to a non-negative convex non-degenerate polynomial $P_1$. Moreover, 
\begin{align*}
 t P_1(t^{-1/m_1}z_1, \dots, t^{-1/m_d}z_d) = P_1(z_1, \dots, z_d)
 \end{align*}
 for all $t \in \Rb$. 
\end{proposition}

Delaying the proof of  Proposition~\ref{prop:multi_type} we show that geodesics are well behaved in polynomial domains. 

\begin{proposition}\label{prop:geodesics_well_behaved}
Suppose $\Omega$ is a domain of the form 
\begin{align*}
\Omega = \{ (z_0, \dots, z_d) \in \Cb^{d+1} : \Real(z_0)  > P(z_1, \dots, z_d) \}
\end{align*}
where $P$ is a non-negative non-degenerate convex polynomial with $P(0)=0$. If $\sigma: \Rb \rightarrow \Omega$ is a geodesic then $\lim_{t \rightarrow -\infty} \sigma(t)$ and $\lim_{t \rightarrow \infty} \sigma(t)$ both exist in $\overline{\Cb^d}$ and are distinct.
\end{proposition}

\begin{proof}
Since $\Omega$ has $C^2$ boundary and is locally $m$-convex using Corollary~\ref{cor:geodesic_limits} and Corollary~\ref{cor:geod_limits_distinct} we know that both limits exist and they are distinct if at least one is finite. So suppose for a contradiction that 
\begin{align*}
\lim_{t \rightarrow -\infty} \sigma(t)=\lim_{t \rightarrow+ \infty} \sigma(t)=\infty.
\end{align*}

Now by Proposition~\ref{prop:multi_type} we can make an linear change of coordinates such that $ t P(t^{-1/m_1}z_1, \dots, t^{-1/m_d}z_d)$ converges to $P_1(z_1,\dots, z_d)$ locally uniformly in the $C^\infty$ topology as $t \rightarrow 0$. If we let 
\begin{align*}
\Lambda_n = \begin{pmatrix}
 n^{-1} & & & \\
 & n^{-1/m_1} && \\
 && \ddots & \\
 &&& n^{-1/m_d} 
 \end{pmatrix}
 \end{align*}
 then $\Lambda_n \Omega$ converges in the local Hausdorff topology to the domain 
 \begin{align*}
\wh{\Omega} = \{ (z_0, \dots, z_d) \in \Cb^{d+1} : \Real(z_0)  > P_1(z_1, \dots, z_d) \}
\end{align*}
 as $n \rightarrow \infty$. Since $n^{-1} P(n^{1/m_1}z_1, \dots, n^{1/m_d}z_d)$ converges to $P_1(z_1,\dots, z_d)$ locally uniformly in the $C^\infty$ topology we see that the family $\Lambda_n \Omega$ is a locally $L$-convex sequence (see Example~\ref{ex:L_convex_seq}).
 
Next consider the geodesic $\sigma_n= \Lambda_n \circ \sigma : \Rb \rightarrow \Lambda_n \Omega$. Since 
\begin{align*}
\lim_{t \rightarrow -\infty} \sigma_n(t)=\lim_{t \rightarrow+ \infty} \sigma_n(t)=\infty
\end{align*}
and $\norm{\sigma_n(0)} < \norm{\sigma(0)}$ we can find $\alpha_n \in (-\infty,0]$ and $\beta_n \in [0,\infty)$ such that the geodesics $t \rightarrow \sigma_n(t +\alpha_t)$ and $t \rightarrow \sigma_n(t +\beta_t)$ restricted to some subinterval satisfy the hypothesis of Proposition~\ref{prop:m_convex}. So there exists $n_k \rightarrow \infty$ such that $\sigma_{n_k}(t+\alpha_{n_k})$ and $\sigma_{n_k}(t +\beta_{n_k})$ converge locally uniformly to geodesics $\wh{\sigma}_1:\Rb \rightarrow \wh{\Omega}$ and $\wh{\sigma}_2:\Rb \rightarrow \wh{\Omega}$. Since $ \sigma_n(0) \rightarrow 0$ and $0 \in \partial \wh{\Omega}$ we see that $\alpha_{n_k} \rightarrow -\infty$ and $\beta_{n_k} \rightarrow \infty$. Then
\begin{align*}
d_{\wh{\Omega}}(\wh{\sigma}_1(0),\wh{\sigma}_2(0)) 
&= \lim_{k \rightarrow \infty} d_{\Lambda_{n_k} \Omega}(\Lambda_{n_k} \sigma(\alpha_{n_k}), \Lambda_{n_k} \sigma(\beta_{n_k}))\\
& = \lim_{k \rightarrow \infty} d_{ \Omega}( \sigma(\alpha_{n_k}),  \sigma(\beta_{n_k}))=  \lim_{n \rightarrow \infty} \abs{\alpha_{n_k}-\beta_{n_k}}=\infty.
\end{align*}
Which is a contradiction. 
\end{proof}

\begin{proof}[Proof of Proposition~\ref{prop:multi_type}]
Given a vector $\vec{v} \in \Cb^d$ let $\degree(\vec{v})$ denote the degree of the polynomial $z \in \Cb \rightarrow P(z\vec{v})$. Since $P$ is non-degenerate and $P(0)=0$,  $\deg(\vec{v}) > 0$ for all nonzero vectors $\vec{v}\in \Cb^d$. Moreover, $\deg(\vec{v}) \leq M$ if and only if there exists $C \geq 0$ so that $\abs{P(z\vec{v})} \leq C+C \abs{z}^M$ for all $z \in \Cb$. 

If $\vec{v}, \vec{w} \in \Cb^d$ then 
\begin{align*}
\abs{ P(z\vec{v}+z\vec{w}) } 
& =P(z\vec{v}+z\vec{w}) = P\left(\frac{1}{2}(2z\vec{w}) + \frac{1}{2}(2z\vec{v}) \right)
\leq \frac{1}{2} P(2z\vec{w}) + \frac{1}{2} P(2z\vec{v})\\
& \leq C_0+C_1 \abs{z}^{\deg(\vec{w})} + C_2\abs{z}^{\deg(\vec{v})}
 \end{align*}
by the convexity of $P$. Hence $\deg(\vec{v}+\vec{w}) \leq \max\{ \deg(\vec{v}), \deg(\vec{w})\}$. This implies that there exists subspaces
\begin{align*}
\{0\} = V_0 \subset V_1 \subset \dots \subset V_k = \Cb^d
\end{align*}
and integers $0=D_0 < D_1 < \dots < D_k$ such that for all $\vec{v}\in V_\ell \setminus V_{\ell-1}$ we have $\deg(\vec{v}) = D_\ell$.

Now fix linear coordinates such that for all $1 \leq \ell \leq k$
\begin{align*}
V_\ell = \{ (z_1, \dots, z_{d_{\ell}}, 0, \dots, 0): z_1, \dots, z_{d_\ell}\in \Cb\}
\end{align*}
where 
\begin{align*}
d_{\ell} =  \dimension(V_\ell).
\end{align*}
Let $e_1, \dots, e_d$ be the standard basis in these coordinates and for $1 \leq i \leq d$ let $m_i= \deg(e_i)$. Now in these linear coordinates $P$ can be written as
\begin{align*}
P(z) = \sum_{\alpha, \beta} a_{\alpha,\beta} z^{\alpha}\overline{z}^{\beta}
\end{align*}
using the usual multi-indices notation.  For $\alpha=(\alpha_1, \dots, \alpha_d)$ and $\beta=(\beta_1,\dots, \beta_d)$ let 
\begin{align*}
\omega(\alpha, \beta) = \sum_{i=1}^d \frac{ \alpha_i+\beta_i}{m_i}.
\end{align*}
We first claim that $\omega(\alpha,\beta) \leq 1$ whenever $a_{\alpha,\beta} \neq 0$. Suppose not, then let 
\begin{align*}
\delta = \max \{ \omega(\alpha, \beta) : a_{\alpha, \beta} \neq 0\}
\end{align*}
and 
\begin{align*}
P_{\delta}(z) = \sum_{\omega(\alpha,\beta)=\delta} a_{\alpha,\beta} z^{\alpha}\overline{z}^{\beta}.
\end{align*}
Since $\delta$ was picked maximally we see that 
\begin{align*}
 \lim_{t \rightarrow 0} t^{\delta} P(t^{-1/m_1}z_1, \dots, t^{-1/m_d}z_d) = P_{\delta}(z).
 \end{align*}
Since $P_\delta$ is the limit of convex non-negative polynomials, it is a convex and non-negative polynomial. Since $z \rightarrow P(ze_i)$ is a polynomial of degree $m_i$ we see that $P$ has no terms of the form $z_i^{\alpha_i} \overline{z_i}^{\beta_i}$ with 
\begin{align*}
\frac{ \alpha_i+\beta_i}{m_i} > 1.
\end{align*}
Thus for any $i$, the polynomial $z \in \Cb \rightarrow P_{\delta}(z e_i)$ is identically zero. So by convexity $P_{\delta}\equiv 0$. Which is a contradiction, thus $\omega(\alpha,\beta) \leq 1$ whenever $a_{\alpha,\beta} \neq 0$.

Now let 
\begin{align*}
P_{1}(z) = \sum_{\omega(\alpha,\beta)=1} a_{\alpha,\beta} z^{\alpha}\overline{z}^{\beta}.
\end{align*}
Since $\omega(\alpha,\beta) \leq 1$ whenever $a_{\alpha,\beta} \neq 0$ we see that 
\begin{align*}
\lim_{t \rightarrow 0} t P(t^{-1/m_1}z_1, \dots, t^{-1/m_d}z_d) = P_1(z_1, \dots, z_d).
\end{align*}
Moreover 
\begin{align*}
 t P_1(t^{-1/m_1}z_1, \dots, t^{-1/m_d}z_d) = P_1(z_1, \dots, z_d)
 \end{align*}
 for all $t \in \Rb$. Since $P_1$ is the limit of non-negative convex functions, it is non-negative and convex as well. Thus it only remains to verify that $P_1$ is non-degenerate. Since $P_1$ is convex it is enough to verify that for any nonzero $\vec{v} \in \Cb^d$ the polynomial $z \in \Cb \rightarrow P_{1}(z\vec{v})$ is not identically zero. So fix $\vec{v} \in \Cb^d$ nonzero. Then there exists $\ell$ such that $\vec{v} \in V_{\ell} \setminus V_{\ell-1}$. We claim that 
  \begin{align*}
   \text{terms of degree $D_\ell$ in } P(z\vec{v})  =    \text{terms of degree $D_\ell$ in } P_{1}(z\vec{v}).
   \end{align*}
   Now 
   \begin{align*}
     \text{terms of degree $D_\ell$ in } P(z\vec{v}) = \sum \left\{ a_{\alpha,\beta} (zv)^{\alpha}\overline{(zv)}^\beta : \sum_{i=1}^d \alpha_i+\beta_i = D_\ell\right\}
     \end{align*}
    and 
      \begin{align*}
     \text{terms of degree $D_\ell$ in } P_1(z\vec{v}) = \sum \left\{ a_{\alpha,\beta} (zv)^{\alpha}\overline{(zv)}^\beta : \omega(\alpha,\beta) =1 \text{ and } \sum_{i=1}^d \alpha_i+\beta_i = D_\ell\right\}.
     \end{align*}

 So suppose that $a_{\alpha,\beta} z^{\alpha}\overline{z}^{\beta}$ is a term in $P$ such that 
 \begin{align*}
 \sum_{i=1}^d \alpha_i+\beta_i = D_\ell.
 \end{align*}
 We will show that either $(zv)^{\alpha}\overline{(zv)}^\beta \equiv 0$ or $\omega(\alpha,\beta)=1$ which will imply the claim. Since $v_i = 0$ for $i>d_{\ell}$ either $(zv)^{\alpha}\overline{(zv)}^\beta \equiv 0$ or $\alpha_i=\beta_i = 0$ for $i > d_{\ell}$. In the latter case
 \begin{align*}
 1\geq \omega(\alpha,\beta) = \sum_{i=1}^d \frac{\alpha_i+\beta_i}{m_i} =  \sum_{i=1}^{d_{\ell}} \frac{\alpha_i+\beta_i}{m_i} \geq  \frac{1}{m_{d_{\ell}}} \sum_{i=1}^{d_{\ell}} \alpha_i+\beta_i = 1
 \end{align*}
 since $m_{d_\ell} = D_\ell$. So $\omega(\alpha,\beta) =1$.
 \end{proof}
 
\section{The proof of Theorem~\ref{thm:main} and Proposition~\ref{prop:main}}

\begin{proof}[Proof of Theorem~\ref{thm:main}] 
Suppose that $\Omega$ is a bounded convex open set with $C^\infty$ boundary.

If $\Omega$ has finite type then Theorem~\ref{thm:suff}, Theorem~\ref{thm:gaussier}, and Proposition~\ref{prop:geodesics_well_behaved} imply that $(\Omega, d_{\Omega})$ is Gromov hyperbolic. 

Conversely if $(\Omega, d_{\Omega})$ is Gromov hyperbolic, then Proposition~\ref{prop:infinite_type_2} implies that $\Omega$ has finite type. 
\end{proof}

\begin{proof}[Proof of Proposition~\ref{prop:main}] Given a geodesic $\sigma:[0,\infty) \rightarrow \Omega$ the limit 
\begin{align*}
\lim_{t \rightarrow \infty} \sigma(t) 
\end{align*}
exists by Corollary~\ref{cor:geodesic_limits} and is in $\partial \Omega$. We claim that this limit only depends on the choice of asymptotic class of geodesic. To see this suppose that $\sigma_1$ and $\sigma_2$ are two geodesic rays with 
\begin{align*}
\sup_{t \geq 0} d_{\Omega}(\sigma_1(t),\sigma_2(t)) < \infty.
\end{align*}
Then 
\begin{align*}
\lim_{t \rightarrow \infty} (\sigma_1(t) | \sigma_2(t))_{o} = \infty
\end{align*}
which implies by Proposition~\ref{prop:gromov_prod_infinite} that
\begin{align*}
\lim_{t \rightarrow \infty} \sigma_1(t) = \lim_{t \rightarrow \infty} \sigma_2(t).
\end{align*}
Thus the map  $\Phi: \Omega \cup \Omega(\infty) \rightarrow \Omega \cup \partial \Omega$ given by 
\begin{align*}
\Phi(\xi) = \left\{ \begin{array}{ll}
\xi & \text{ if } \xi \in \Omega \\
 \lim_{t \rightarrow \infty} \sigma(t) & \text{ if } \xi=[\sigma] \in \Omega(\infty) 
\end{array}
\right.
\end{align*}
is well defined.

We claim that $\Phi$ is continuous, injective, and surjective. Since $\Omega \cup \Omega(\infty)$ is compact this will imply that $\Phi$ is a homeomorphism. 

\textbf{Surjective:} It is enough to show that for all $x \in \partial \Omega$ there exists a geodesic ray $\sigma:[0,\infty) \rightarrow \Omega$ such that 
\begin{align*}
\lim_{t \rightarrow \infty} \sigma(t) = x.
\end{align*}
Fix a point $o \in \Omega$ and a sequence $x_n \in \Omega$ such that $x_n \rightarrow x$. Then let $\sigma_n:[0,T_n] \rightarrow \Omega$ be a geodesic such that $\sigma_n(0)=o$ and $\sigma_n(T_n) = x_n$. We can pass to a subsequence so that $\sigma_n$ converges locally uniformly to a geodesic ray $\sigma : [0,\infty) \rightarrow \Omega$. But then by Proposition~\ref{prop:limits_exist} 
\begin{align*}
\lim_{t \rightarrow \infty} \sigma(t) = x.
\end{align*}
Hence $\Phi$ is onto. 

\textbf{Continuous:} Suppose $\xi_n \rightarrow \xi$ in $\Omega \cup \Omega(\infty)$. If $\xi \in \Omega$ then clearly $\Phi(\xi_n) \rightarrow \Phi(\xi)$. So we can assume that $\xi_n \in \Omega(\infty)$. Since $\Omega \cup \partial\Omega$ is compact, it is enough to show that every convergent subsequence of $\Phi(\xi_n)$ converges to $\Phi(\xi)$. So we may also assume that $\Phi(\xi_n) \rightarrow x$ for some $x \in \partial \Omega$. 

Now fix $o \in \Omega$ and  let $\sigma_n:[0,T_n) \rightarrow \Omega$ be a geodesic with $\sigma_n(0)=o$ and 
\begin{align*}
\lim_{t \rightarrow T_n} \sigma_n(t)= \Phi(\xi_n).
\end{align*}
Notice that $T_n$ could be $\infty$, so pick $T_n^\prime \in (0, T_n)$ such that 
\begin{align*}
\lim_{n \rightarrow \infty} \sigma_n(T_n^\prime) = \lim_{n \rightarrow \infty} \Phi(x_n) = x.
\end{align*}
Now by the definition of topology on $\Omega \cup \Omega(\infty)$, if $\sigma$ is the limit of a convergent subsequence $\sigma_{n_k}$ of $\sigma_n$ then $\xi = [\sigma]$. Moreover, by Proposition~\ref{prop:limits_exist}
\begin{align*}
 \lim_{t \rightarrow \infty} \sigma(t) =\lim_{k \rightarrow \infty} \sigma_{n_k}(T_{n_k}^\prime) = x.
 \end{align*}
Thus $\Phi(\xi) = x$. So $\Phi$ is continuous. 
 
 \textbf{Injective:} Suppose for a contradiction that $\Phi(\xi_1)=\Phi(\xi_2)$ for some $\xi_1 \neq \xi_2$ in $\Omega \cup \Omega(\infty)$. Since $\Phi|_{\Omega} = id$ we must have that $\xi_1,\xi_2 \in \Omega(\infty)$. Now let $\sigma_1, \sigma_2$ be geodesic representatives of $\xi_1, \xi_2$. Then 
 \begin{align*}
 \lim_{t \rightarrow \infty} \sigma_1(t) = \Phi(\xi_1) = \Phi(\xi_2)= \lim_{t \rightarrow \infty} \sigma_2(t).
 \end{align*}
This implies, by Proposition~\ref{prop:gromov_prod_finite}, that 
\begin{align*}
\lim_{t \rightarrow \infty} (\sigma_1(t) | \sigma_2(t))_o = \infty.
\end{align*}
But by~\cite[Chapter III.H Lemma 3.13]{BH1999} this happens only if $\xi_1 = [\sigma_1]=[\sigma_2]=\xi_2$ which is a contradiction.

\end{proof}

\part{Locally convexifiable sets}

\section{Locally  convexifiable sets}\label{sec:local_convex}

If $\Omega$ is a bounded open set and $\xi \in \partial\Omega$ we call a pair $(V_{\xi}, \Phi_{\xi})$ a \emph{local convex chart of $\Omega$ at $\xi$} if $V_{\xi}$ is a neighborhood of $\xi$ in $\Cb^d$ and $\Phi_{\xi} : V_{\xi} \rightarrow \Cb^d$ is a bi-holomorphism onto its image such that $\Omega_{\xi}:=\Phi_{\xi}( V_{\xi} \cap \Omega)$ is convex. If for each $\xi \in \partial \Omega$ there exists a local convex chart at $\xi$ then we call $\Omega$ \emph{locally convexifiable}. Given a locally convex set $\Omega$, a \emph{locally convex atlas} is a choice $(V_{\xi}, \Phi_{\xi})$ for each $\xi \in \partial \Omega$. 

Suppose $\Omega$ is locally convexifiable and has finite type in the sense of D'Angelo, then there exists $L>0$ such that for any local convex chart $(V_{\xi}, \Phi_{\xi})$ of $\Omega$ the hypersurface $\Phi_{\xi}(\partial \Omega \cap V_{\xi})$ has  line type at most $L$ near $\Phi_{\xi}(\xi)$ (see~\cite{BS1992}). This motivates the next definition:

\begin{definition} 
Suppose $\Omega$ is a bounded open set and $\partial \Omega$ is a $C^L$ hypersurface. If $\Omega$ is locally convexifiable and there exists a locally convex atlas $\{ (V_{\xi}, \Phi_{\xi}) : \xi \in \partial \Omega \}$  such that for every $\xi$ the hypersurface $\Phi_{\xi}(\partial \Omega \cap V_{\xi})$ has  line type at most $L$ near $\Phi_{\xi}(\xi)$ then we say that $\Omega$ has \emph{finite local line type at most $L$.} If, in addition, there exists a $\xi$ where $\Phi_{\xi}(\partial \Omega \cap V_{\xi})$ has  line type  $L$ near $\Phi_{\xi}(\xi)$ then we say $\Omega$ has \emph{finite local line type $L$.}
\end{definition}

\begin{example} It is well known that a strongly pseudo-convex domain with $C^2$ boundary is locally convexifiable and has finite local line type 2. \end{example}

The rest of the paper is devoted to proving:

\begin{theorem}\label{thm:loc_main}
Suppose $\Omega$ is locally convexifiable and has finite local line type $L$. Then $(\Omega, d_{\Omega})$ is Gromov hyperbolic. Moreover the identity map $\Omega \rightarrow \Omega$ extends to a homeomorphism $\Omega \cup \partial \Omega \rightarrow \Omega \cup \Omega(\infty)$.
\end{theorem}

By the remarks above we have the following corollaries:

\begin{corollary}
Suppose $\Omega$ is locally convexifiable and has finite type in the sense of D'Angelo. Then $(\Omega, d_{\Omega})$ is Gromov hyperbolic. Moreover the identity map $\Omega \rightarrow \Omega$ extends to a homeomorphism $\Omega \cup \partial \Omega \rightarrow \Omega \cup \Omega(\infty)$.
\end{corollary}

\begin{corollary}\cite[Theorem 1.4]{BB2000}
Suppose $\Omega$ is a bounded strongly pseudo convex domain with $C^2$ boundary. Then $(\Omega, d_{\Omega})$ is Gromov hyperbolic.  Moreover, the identity map $\Omega \rightarrow \Omega$ extends to a homeomorphism $\Omega \cup \partial \Omega \rightarrow \Omega \cup \Omega(\infty)$.
\end{corollary}

\subsection{Finding a good locally convex atlas} 

We begin by finding a locally convex atlas with nice properties.

\begin{definition}
Suppose $\Omega$ is locally convexifiable and has finite local line type $L$. Then we say a locally convex atlas $\{ (V_{\xi}, \Phi_{\xi}) : \xi \in \partial \Omega\}$ is \emph{good} if there exists numbers $C, \tau>0$ such that for every $\xi \in \partial \Omega$ we have
\begin{enumerate}
\item $B_{\tau}(\xi) \subset V_{\xi}$,
\item for every $p \in  \Omega_{\xi} := \Phi_{\xi}(\Omega \cap V_{\xi})$ and $v \in \Cb^d$ nonzero
\begin{align*}
\delta_{\Omega_{\xi}}(p; v) \leq C \delta_{\Omega_{\xi}}(p)^{1/L}
\end{align*}
\item for every $p,q \in V_{\xi}$ 
\begin{align*}
\frac{1}{C}\norm{p-q} \leq \norm{\Phi_{\xi}(p) - \Phi_{\xi}(q)} \leq C\norm{p-q}.
\end{align*}
\end{enumerate}
\end{definition}

\begin{proposition}\label{prop:good_cover}
Suppose $\Omega$ is locally convexifiable and has finite local line type $L$. Then there exist a good locally convex atlas.\end{proposition}

\begin{proof}
Let $\{ (W_{\xi}, \Psi_{\xi}) : \xi \in \partial \Omega\}$ be a locally convex atlas that satisfies the definition of finite local line type. Let $\Cc_{\xi} = \Psi_{\xi}(\Omega \cap W_{\xi})$. Since the hypersurface $\Psi_{\xi}(\partial \Omega \cap W_{\xi})$ has line type at most $L$ near $\xi$ there exists $\tau(\xi) >0$ and $C(\xi)>1$ such that 
\begin{align*}
\delta_{\Cc_{\xi}}(p; v) \leq C(\xi) \delta_{\Cc_{\xi}}(p)^{1/L}
\end{align*}
for all $p \in B_{\tau(\xi)}(\Phi_{\xi}(\xi))$ and $v \in \Cb^d$ (see Proposition~\ref{prop:finite_type}).

For $r=r(\xi) < \tau(\xi)$ sufficiently small let 
\begin{align*}
W^\prime_{\xi} : = W_{\xi} \cap \Psi_{\xi}^{-1}(B_{r}(\Phi_{\xi}(\xi)))
\end{align*}
and 
\begin{align*}
\Cc_{\xi}^\prime : =  \Psi_{\xi}(\Omega \cap W_{\xi}^\prime)= \Cc_{\xi} \cap B_r(\Phi_\xi(\xi)).
\end{align*}
Now since $B_r(\xi)$ is strongly convex there exists $C_r>0$ such that 
\begin{align*}
\delta_{B_r(\xi)}(p; v) \leq C_r \delta_{B_r(\xi)}(p)^{1/2}
\end{align*}
for all $p \in B_r(\xi)$ and $v \in \Cb^d$ nonzero. By increasing $C(\xi)$ we may assume $C_r \leq C(\xi)$. Then 
\begin{align*}
\delta_{\Cc^\prime_{\xi}}(p; v) 
&= \min\{ \delta_{\Cc_{\xi}}(p; v), \delta_{B_r(\xi)}(p; v) \} \leq \min\{ C(\xi) \delta_{\Cc_{\xi}}(p)^{1/L}, C_r \delta_{B_r(\xi)}(p)^{1/2} \} \\
& \leq C(\xi) \delta_{\Cc^\prime_{\xi}}(p)^{1/L}
\end{align*}
for all $p \in \Cc_{\xi}^\prime$ and $v \in \Cb^d$ nonzero. By possibly decreasing $r$ and increasing $C(\xi)$ we can also assume that 
\begin{align*}
\frac{1}{C(\xi)}\norm{p-q} \leq \norm{\Psi_{\xi}(p) - \Psi_{\xi}(q)} \leq C(\xi)\norm{p-q}
\end{align*}
for all $p,q \in W_{\xi}^\prime$. 

Now $\cup W_{\xi}^\prime$ is an open cover of $\partial \Omega$ and $\partial \Omega$ is compact, so there exists $\xi_1, \dots, \xi_N \in \partial \Omega$ such that 
\begin{align*}
\partial \Omega \subset \cup_{i=1}^N W_{\xi_i}^\prime
\end{align*}
Let $V_i = W_{\xi_i}^\prime$ and $\Phi_i = \Psi_{\xi_i}^\prime$. Then there exists $\tau >0$ and a map $I:\partial\Omega \rightarrow \{1,\dots, N\}$ such that for all $\xi \in \partial \Omega$ 
\begin{align*}
B_{\tau}(\xi) \subset V_{I(\xi)}.
\end{align*}
So let $V_{\xi}=V_{I(\xi)}$, $\Phi_{\xi} = \Phi_{I(\xi)}$, and 
\begin{equation*}
C = \max\{ C(\xi_i) : 1 \leq i \leq N\}. \qedhere
\end{equation*}
\end{proof} 

\section{Almost geodesics}

In the proof of Theorem~\ref{thm:loc_main} we will need to understand the limits of ``almost'' geodesics $\sigma_n: \Rb \rightarrow \Omega_n$ when a sequence of $\Cb$-proper convex sets $\Omega_n$ converges in the local Hausdorff topology to a $\Cb$-proper convex set $\Omega$. 

\begin{definition}
Suppose $(X,d)$ is a metric space and $K \geq 1$. A map $\sigma: (a,b) \rightarrow \Omega$ is a called a \emph{$K$-almost-geodesic segment} if for all $s,t \in (a,b)$ 
\begin{align*}
\abs{s-t} - K \leq d(\sigma(s),\sigma(t)) \leq \abs{s-t} + K
\end{align*}
and 
\begin{align*}
\frac{1}{K}\abs{s-t}  \leq d(\sigma(s),\sigma(t)) \leq K\abs{s-t}.
\end{align*}
\end{definition}

\begin{remark}
The first condition on a $K$-almost-geodesic allows one to estimates the Gromov product: if $\sigma :\Rb \rightarrow X$ is a $K$-almost geodesic then 
\begin{align*}
\abs{ (\sigma(t)|\sigma(s))_{\sigma(0)} -\Big( \abs{t}+\abs{s}-\abs{t-s}\Big) } \leq 3K. 
\end{align*}
The second condition on a $K$-almost-geodesic allows one to take limits: if $\sigma_n : \Rb \rightarrow X$ is a sequence of $K$-almost-geodesics and $\{\sigma_n(0)\}_{n \in \Nb}$ is relatively compact then a subsequence of $\sigma_n$ converges locally uniformly to a $K$-almost-geodesic $\sigma: \Rb \rightarrow X$.
\end{remark}

The proofs of Proposition~\ref{prop:m_convex}, Proposition~\ref{prop:limits_exist}, and Corollary~\ref{cor:geodesic_limits} adapt essentially verbatim to the almost geodesic case:

\begin{proposition}\label{prop:rough_geodesic_limits_exist}
Suppose $\Omega_n$ is a locally m-convex sequence of $\Cb$-proper convex open sets converging to a $\Cb$-proper convex open set $\Omega$ in the local Hausdorff topology. Assume $\sigma_n : \Rb \rightarrow \Omega_n$ is a sequence of $K$-almost-geodesics such that there exists $a_n \leq b_n$ and $R >0$ satisfying 
\begin{enumerate}
\item $\sigma_n([a_n,b_n]) \subset B_R(0)$,
\item $\lim_{n \rightarrow \infty} \norm{\sigma_n(a_n)-\sigma_n(b_n)} > 0$, 
\end{enumerate}
then there exists $T_n \in [a_n,b_n]$ such that a subsequence of $\sigma_n(t+T_n)$ converges locally uniformly to a $K$-almost-geodesic $\sigma:\Rb \rightarrow \Omega$.  
\end{proposition}

\begin{proposition}\label{prop:rough_geodesic_limits_exist_2}
Suppose $\Omega_n$ is a locally m-convex sequence of $\Cb$-proper convex open sets converging to a $\Cb$-proper convex open set $\Omega$ in the local Hausdorff topology. Assume $\sigma_n : \Rb \rightarrow \Omega_n$ is a sequence of $K$-almost-geodesics converging locally uniformly to a $K$-almost-geodesic $\sigma:\Rb \rightarrow \Omega$. If $t_n \rightarrow \infty$ is a sequence such that $\lim_{n \rightarrow \infty} \sigma_n(t_n) = x_{\infty} \in \overline{\Cb^d}$  then
\begin{align*}
\lim_{t \rightarrow \infty} \sigma(t) = x_{\infty}.
\end{align*}
\end{proposition}

\begin{corollary}\label{cor:rough_geodesic_limits}
Suppose $\Omega$ is a locally m-convex open set. If $\sigma: \Rb \rightarrow \Omega$ is a $K$-almost-geodesic then 
\begin{align*}
\lim_{t \rightarrow -\infty} \sigma(t) \text{ and } \lim_{t \rightarrow +\infty} \sigma(t) 
\end{align*}
both exist in $\overline{\Cb^d}$.
\end{corollary}

We also need to know that almost geodesics are well behaved in polynomial domains:

\begin{proposition}\label{prop:rough_geodesics_well_behaved}
Suppose $\Omega$ is a domain of the form 
\begin{align*}
\Omega = \{ (z_0, \dots, z_d) \in \Cb^{d+1} : \Real(z_0)  > P(z_1, \dots, z_d) \}
\end{align*}
where $P$ is a non-negative non-degenerate convex polynomial with $P(0)=0$. If $\sigma: \Rb \rightarrow \Omega$ is a $K$-almost-geodesic then $\lim_{t \rightarrow -\infty} \sigma(t)$ and $\lim_{t \rightarrow \infty} \sigma(t)$ both exist in $\overline{\Cb^d}$ and are distinct.
\end{proposition}

\begin{proof}
By Proposition~\ref{prop:finite_type} $\Omega$ is a $L$-locally convex set so by Corollary~\ref{cor:rough_geodesic_limits} both limits exist. Since 
\begin{align*}
(\sigma(-s)|\sigma(t))_{\sigma(0)} \leq 3K
\end{align*}
when $s$ and $t$ are positive, Proposition~\ref{prop:gromov_prod_finite} implies that the limits are distinct if at least one is finite. To show that it is impossible for both limits to equal $\infty$ one can use the proof of Proposition~\ref{prop:geodesics_well_behaved} verbatim. 
\end{proof}

\section{Localization of the Kobayashi metric}

\subsection{Infinitesimal localization} 

The proof of the next theorem is based on an argument of Forstneri{\v{c}} and Rosay~\cite[Section 2]{FR1987} for strongly pseudo-convex domains. 

\begin{theorem}\label{thm:inf_loc}
Suppose $\Omega$ is locally convexifiable, has finite local line type $L$, and $\{ (V_{\xi}, \Phi_{\xi}) : \xi \in \partial \Omega\}$ is a good convex atlas. Then there numbers $c, \epsilon >0$ so that 
\begin{align*}
K_{\Omega}(p; v) \leq K_{\Omega \cap V_{\xi}}(p;v) \leq e^{c\delta_{\Omega}(p)^{1/L}}K_{\Omega}(p;v)
\end{align*}
for every $\xi \in \partial \Omega$, every $p \in B_{\epsilon}(\xi) \cap \Omega$, and every $v \in \Cb^d$.
\end{theorem}

The next lemma shows that locally convexifiable sets with finite type satisfy an analogue of Property ($*$) in~\cite[pp.  244]{FR1987}. 

\begin{lemma}\label{lem:inf_loc}  Suppose $\Omega$ is locally convexifiable, has finite local line type $L$, and $\{ (V_{\xi}, \Phi_{\xi}) : \xi \in \partial \Omega\}$ is a good convex atlas with parameters $C,\tau$.  For any $\eta >0$ there exists $\alpha>0$ such that if $\xi \in \partial \Omega$, $\varphi: \Delta \rightarrow B_{\tau}(\xi) \cap \Omega$, and 
\begin{align*}
\abs{\zeta} \leq 1 -\alpha\delta_{\Omega}(\varphi(0))^{1/L}
\end{align*}
then 
\begin{align*}
\norm{\varphi(\zeta)-\varphi(0)} \leq \eta.
\end{align*}
\end{lemma}

\begin{proof}
Let
\begin{align*} 
\alpha: = \frac{4C^{2+1/L}}{\eta}
\end{align*}
and consider some $\xi \in \partial \Omega$, $\varphi: \Delta \rightarrow B_{\tau}(\xi) \cap \Omega$, and $\zeta \in \Delta$ with 
\begin{align*}
\abs{\zeta} \leq 1 -\alpha\delta_{\Omega}(\varphi(0))^{1/L}.
\end{align*}
Notice that
\begin{align*}
\delta_{\Omega}(\varphi(0)) \leq \left(\frac{\eta}{2 C^{2+1/L}}\right)^{L}
\end{align*}
for otherwise 
\begin{align*}
1-\alpha\delta_{\Omega}(\varphi(0))^{1/L} < 0.
\end{align*}

Let $\wh{\varphi} = \Phi_{\xi} \circ \varphi:\Delta \rightarrow \Omega_{\xi}$. If we select $v \in \Cb^d$ such that $\wh{\varphi}(\zeta) \in \wh{\varphi}(0) + \Cb v$ we have 
\begin{align*}
\frac{1}{2} \log \abs{\frac{ \norm{\wh{\varphi}(\zeta)-\wh{\varphi}(0)} - \delta_{\Omega_\xi}(\wh{\varphi}(0); v)}{\delta_{\Omega_\xi}(\wh{\varphi}(0); v)}} \leq d_{\Omega_\xi}(\wh{\varphi}(0), \wh{\varphi}(\zeta))
\end{align*}
by Lemma~\ref{lem:convex_lower_bd_2}. Further, by the distance decreasing property of the Kobayashi metric, we have 
\begin{align*}
d_{\Omega_\xi}(\wh{\varphi}(0), \wh{\varphi}(\zeta)) \leq d_{\Delta}(0, \zeta) = \frac{1}{2} \log \frac{1 +\abs{\zeta}}{1-\abs{\zeta}} \leq \frac{1}{2} \log \frac{2}{1-\abs{\zeta}}.
\end{align*}
Combining the two inequalities we have 
\begin{align*}
\norm{\wh{\varphi}(\zeta)-\wh{\varphi}(0)} \leq \left(\frac{2}{1-\abs{\zeta}}+1\right) \delta_{\Omega_\xi}(\wh{\varphi}(0); v)
\end{align*}
Using  the properties of a good convex atlas we see that 
\begin{align*}
\norm{\varphi(\zeta)-\varphi(0)} 
&\leq C \norm{\wh{\varphi}(\zeta)-\wh{\varphi}(0)} \leq C\left(\frac{2}{1-\abs{\zeta}}+1\right) \delta_{\Omega_\xi}(\wh{\varphi}(0); v) \\
& \leq C^2\left(\frac{2}{1-\abs{\zeta}}+1\right) \delta_{\Omega_\xi}(\wh{\varphi}(0))^{1/L}  \\
& \leq  C^{2+1/L}\left(\frac{2}{1-\abs{\zeta}}+1\right) \delta_{\Omega}(\varphi(0))^{1/L}.
\end{align*}
Hence, by our choice of $\alpha$, we have that 
\begin{equation*}
\norm{\varphi(\zeta)-\varphi(0)} \leq \eta. \qedhere
\end{equation*}
\end{proof}

With the lemma established the rest of proof of Theorem~\ref{thm:inf_loc} follows~\cite[Theorem 2.1]{FR1987} essentially verbatim, but we will include the argument for the reader's convenience.

\begin{proof}[Proof of Theorem~\ref{thm:inf_loc}]
By shrinking $\tau$, we can assume that there exists $\eta >0$ such that whenever $p \in B_{\tau}(\xi)$ then $B_{\eta}(p) \subset V_{\xi}$. 

Now given $\epsilon>0$ sufficiently small let $\rho=\rho(\epsilon)$ be the largest number such that: if $\varphi: \Delta \rightarrow \Omega$ is holomorphic with $\varphi(0) \in B_{\epsilon}(\xi)$ and $\abs{\zeta} \leq \rho$ then $\abs{\varphi(\zeta)-\varphi(0)} \leq \eta$. To prove the theorem it is enough to show that $\rho \geq e^{-c \epsilon^{1/L}}$ for some $c>0$ (which does not depend on $\epsilon$).  

By scaling, we may assume that $\diam_{\operatorname{Euc}}(\Omega) \leq 1$. Then by the Schwarz lemma we have $\rho \geq \eta$. Now by Lemma~\ref{lem:inf_loc} there exists $\alpha>0$ such that if 
\begin{align*}
\delta_{\Omega}(\varphi(0)) \leq \epsilon \text{ and } \abs{\zeta} \leq \rho - \alpha\epsilon^{1/L}
\end{align*}
then $\abs{\varphi(\zeta)-\varphi(0)} \leq \eta/2$. 

If $\rho< 1$ then there exists a holomorphic map $\varphi: \Delta \rightarrow \Omega$ such that $\varphi(0) \in B_{\epsilon}(\xi)$ and
\begin{align*}
\eta = \sup_{\abs{\zeta}=\rho} \abs{\varphi(\zeta)-\varphi(0)}.
\end{align*}
Then by Hadamard's three circle lemma 
\begin{align*}
M(r):=\log \sup_{\abs{\zeta}=r} \abs{\varphi(\zeta)-\varphi(0)}
\end{align*}
is a convex function of $\log(r)$. Since $\rho-\alpha\epsilon^{1/L} < 1$ we have
\begin{align*}
\frac{\log(\rho-\alpha\epsilon^{1/L})}{\log(\eta/2)} \geq \frac{\log(\rho-\alpha\epsilon^{1/L})}{M(\rho-\alpha\epsilon^{1/L})} \geq \frac{ \log(\rho)}{M(\rho)} = \frac{\log(\rho)}{\log(\eta)}.
\end{align*}
Now we may assume that $\epsilon^{1/L} < \eta/(2\alpha)$ so $\rho-\alpha\epsilon^{1/L} \geq \eta/2$ and thus
\begin{align*}
\log(\rho-\alpha\epsilon^{1/L}) \geq \log(\rho) - \frac{2}{\eta}(\alpha\epsilon^{1/L})
\end{align*}
Since $\eta/2 < 1$, one can show that
\begin{align*}
\rho \geq e^{-c\epsilon^{1/L}}
\end{align*}
where 
\begin{equation*}
c = -\frac{2\alpha \log(\eta)}{\eta\log(2)}. \qedhere
\end{equation*}
\end{proof}

\subsection{Global localization}

\begin{theorem}\label{thm:global_loc}
Suppose $\Omega$ is locally convexifiable and has finite local line type $L$. Let $\{(V_{\xi}, \Phi_{\xi}) : \xi \in \partial \Omega\}$ be a good convex atlas. Then there exists $K,\delta >0$ such that for every $p,q \in B_{\delta}(\xi)$ we have 
\begin{align*}
d_{\Omega}(p,q) \leq d_{\Omega \cap V_{\xi}}(p,q) \leq d_{\Omega}(p,q)+K.
\end{align*}
\end{theorem}

For the remainder of this subsection fix a set $\Omega$ satisfying the hypothesis of Theorem~\ref{thm:global_loc}, a good convex atlas $\{(V_{\xi}, \Phi_{\xi}) : \xi \in \partial \Omega\}$, let $C,\tau >0$ be the parameters of the good convex atlas, and let $\epsilon, c >0$ be as in Theorem~\ref{thm:inf_loc}.

\begin{lemma}\label{lem:global_loc_1} There exists $K>0$ such that if $\xi \in \partial \Omega$ and $\sigma:[a,b] \rightarrow \Omega$ is a geodesic with $\sigma([a,b]) \subset B_{\epsilon}(\xi)$ then
\begin{align*}
\abs{t-s} \leq d_{\Omega \cap V_{\xi}}(\sigma(s),\sigma(t)) \leq \abs{t-s}+K
\end{align*}
for all $s,t \in [a,b]$.
\end{lemma}

\begin{proof} Since $\Omega \cap V_{\xi} \subset \Omega$ we immediately have that 
\begin{align*}
d_{\Omega \cap V_{\xi}}(\sigma(s),\sigma(t)) \geq d_{\Omega}(\sigma(t),\sigma(s)) = \abs{t-s}.
\end{align*}
To avoid rectifiable issues, we approximate $\sigma:[a,b] \rightarrow \Omega$ by a $C^1$ curve $\wh{\sigma}: [a,\wh{b}] \rightarrow \Omega$ which is parametrized by arc-length (that is $K_{\Omega}(\wh{\sigma}(t); \wh{\sigma}(t)) \equiv 1$) and which satisfies
\begin{align*}
d_{\Omega \cap V_{\xi}}(\sigma(t), \wh{\sigma}(t)) < 1
\end{align*}
for all $t \in [a,b]$. Notice that this implies that 
\begin{align*}
d_{\Omega}(\sigma(t), \wh{\sigma}(t)) \leq d_{\Omega \cap V_{\xi}}(\sigma(t), \wh{\sigma}(t)) < 1
\end{align*}
for all $t \in [a,b]$.

Now let $T \in [a,\wh{b}]$ be such that 
\begin{align*}
\delta_{\Omega}(\wh{\sigma}(T)) = \max \{ \delta_{\Omega}(\wh{\sigma}(t)) : t \in [a,\wh{b}]\}.
\end{align*}
Fix $o \in \Omega$ then 
\begin{align*}
\abs{T-t}
& = d_{\Omega}(\sigma(T), \sigma(t)) \leq d_{\Omega}(\wh{\sigma}(T), \wh{\sigma}(t))+2\\
&  \leq d_{\Omega}(\wh{\sigma}(T), o)+d_{\Omega}(o, \wh{\sigma}(t))+2
\end{align*}
and by~\cite[Theorem 2.3.51]{A1989} there exists $c_1>0$ such that
\begin{align*}
d_{\Omega}(\wh{\sigma}(T), o)+d_{\Omega}(o, \wh{\sigma}(t)) \leq 2c_1 -\frac{1}{2}\log \left( \frac{1}{\delta_{\Omega}(\wh{\sigma}(t))\delta_{\Omega}(\wh{\sigma}(T))} \right).
\end{align*}
So 
\begin{align*}
\delta_{\Omega}(\wh{\sigma}(t)) \leq \left( \delta_{\Omega}(\wh{\sigma}(t))\delta_{\Omega}(\wh{\sigma}(T)) \right)^{1/2} \leq e^{2c_1+2}e^{-\abs{T-t}}.
\end{align*}
Then using Theorem~\ref{thm:inf_loc} we have for  $s < t$
\begin{align*}
d_{\Omega \cap V_{\xi}}(\wh{\sigma}(s),\wh{\sigma}(t)) 
& \leq \int_{s}^t K_{\Omega \cap V_{\xi}}(\wh{\sigma}(r); \wh{\sigma}^\prime(r))dr \leq \int_{s}^t e^{Ce^{-\abs{T-r}/L}} K_{\Omega}(\wh{\sigma}(r); \wh{\sigma}^\prime(r))dr \\
& = \int_{s}^t e^{Ce^{-\abs{T-r}/L}}dr
\end{align*}
for some $C>0$. 

Now for $\lambda \in [0,1]$ 
\begin{align*}
e^{C\lambda} = 1 + \int_0^\lambda C e^{C s} ds \leq 1+ \int_0^{\lambda} C e^C ds \leq 1 + Ce^C\lambda.
\end{align*} 
Hence
\begin{align*}
e^{Ce^{-\abs{T-r}/L}} \leq 1 + Ce^C e^{-\abs{T-r}/L}
\end{align*}
and so
\begin{align*}
d_{\Omega \cap V_{\xi}}(\wh{\sigma}(s),\wh{\sigma}(t))
& \leq \int_s^{t} 1 + Ce^C e^{-\abs{T-r}/L} dr \leq \abs{t-s} +Ce^C  \int_{-\infty}^{\infty} e^{-\abs{T-r}/L} dr \\
& = \abs{t-s} + K^\prime
\end{align*}
where $K^\prime: = Ce^C  \int_{-\infty}^{\infty} e^{-\abs{r}/L} dr$. Finally, 
\begin{equation*}
d_{\Omega \cap V_{\xi}}(\sigma(s),\sigma(t)) \leq d_{\Omega \cap V_{\xi}}(\wh{\sigma}(s),\wh{\sigma}(t))+2 \leq \abs{t-s}+K^\prime+2. \qedhere
\end{equation*}
\end{proof}

\begin{lemma}\label{lem:global_loc_2}
For any $\eta>0$ there exists $\delta >0$ such that if $\xi \in \partial \Omega$, $p,q \in B_{\delta}(\xi)$, and $\sigma:[a,b] \rightarrow \Omega$ is a geodesic with $\sigma(a)=p$ and $\sigma(b)=q$ then $\sigma([a,b]) \subset B_{\eta}(\xi)$.
\end{lemma}

\begin{proof}
Suppose for a contradiction that the lemma does not hold for some $\eta>0$. We may assume that $\eta < \epsilon$. Then for each $n >0$ there exists a point $\xi_n \in \partial \Omega$ and a geodesic 
\begin{align*}
\sigma_n:[a_n,b_n] \rightarrow \Omega
\end{align*} 
with $\sigma_n(a_n), \sigma_n(b_n) \in B_{1/n}(\xi_n)$ and $\sigma_n(0) \in \Omega \setminus B_{\eta}(\xi)$. By passing to a subsequence we can assume that there exists a single $\xi \in \partial \Omega$ such that  $\sigma_n(a_n), \sigma_n(b_n) \in B_{1/n}(\xi)$ and $\sigma_n(0) \in \Omega \setminus B_{\eta/2}(\xi)$.

Now let 
\begin{align*}
a_n^\prime = \inf\{ t \in [a_n,b_n] : \sigma_n(t) \in \Omega \setminus B_{\eta/2}(\xi)\}.
\end{align*}
and 
\begin{align*}
b_n^\prime = \sup\{ t \in [a_n,b_n] : \sigma_n(t) \in \Omega \setminus B_{\eta/2}(\xi)\}.
\end{align*}
By Theorem~\ref{thm:inf_loc} and Lemma~\ref{lem:global_loc_1}, $\sigma_n|_{[a_n,a_n^\prime]}$ and $\sigma_n|_{[b_n^\prime, b_n]}$ are $K$-almost-geodesics in $\Omega \cap V_{\xi}$. Since $\Omega \cap V_{\xi}$ is bi-holomorphic to a locally m-convex set, by Proposition~\ref{prop:rough_geodesic_limits_exist} we can pass to a subsequence and find $\alpha_n \in [a_n,a_n^\prime]$ and $\beta_n \in [b_n^\prime,b_n]$ such that $t\rightarrow \sigma_n(t+\alpha_n)$ and $t\rightarrow \sigma_n(t+\beta_n)$ converge to $K$-almost-geodesics $\sigma$ and $\gamma$ in $\Omega \cap V_{\xi}$. 

Since 
\begin{align*}
d_{\Omega \cap V_{\xi}}(\sigma(0), \gamma(0)) = \lim_{n \rightarrow \infty} d_{\Omega \cap V_{\xi}}(\sigma(0), \sigma_n(\beta_n)) 
\end{align*}
there exists $R>0$ such that 
\begin{align*}
d_{\Omega \cap V_{\xi}}(\sigma(0), \sigma_n(\beta_n))  \leq R
\end{align*}
for all $n$. So
\begin{align*}
d_{\Omega \cap V_{\xi}}(\sigma(0),\sigma_n(b_n)) \leq R + \abs{b_n - \beta_n}+K.
\end{align*}
Also
\begin{align*}
d_{\Omega \cap V_{\xi}}(\sigma_n(a_n),\sigma_n(b_n)) \geq d_{\Omega}(\sigma_n(a_n),\sigma_n(b_n)) = \abs{b_n-a_n}.
\end{align*}
 
Now let $(\cdot |\cdot)^{V_{\xi}}$ be the Gromov product on $(\Omega \cap V_{\xi}, d_{\Omega \cap V_{\xi}})$. Using the estimates above 
\begin{align*}
2 \limsup_{n \rightarrow \infty} & (\sigma_n(a_n) | \sigma_n(b_n))_{\sigma(0)}^{V_{\xi}}  \\
& = \limsup_{n \rightarrow \infty} \left(d_{\Omega \cap V_\xi}(\sigma_n(a_n), \sigma(0)) + d_{\Omega \cap V_\xi}(\sigma(0), \sigma_n(b_n)) - d_{\Omega\cap V_{\xi}}(\sigma_n(a_n),\sigma_n(b_n)) \right)\\
& = \limsup_{n \rightarrow \infty} \left(d_{\Omega \cap V_\xi}(\sigma_n(a_n), \sigma_n(\alpha_n)) + d_{\Omega \cap V_\xi}(\sigma(0), \sigma_n(b_n)) - d_{\Omega \cap V_{\xi}}(\sigma_n(a_n),\sigma_n(b_n))\right) \\
& \leq \limsup_{n \rightarrow \infty} \left(\abs{a_n - \alpha_n} +K+R+ \abs{b_n - \beta_n} +K - \abs{b_n - a_n} \right) \\
& \leq R+2K. 
\end{align*}
On the other hand, $\Omega \cap V_{\xi}$ is convex and $\partial (\Omega \cap V_{\xi})$ is $C^2$ in a neighborhood of $\xi$. Moreover $\sigma_n(a_n) \rightarrow \xi$ and $\sigma_n(b_n)\rightarrow \xi$ hence by Proposition~\ref{prop:gromov_prod_finite} we have a contradiction. 
\end{proof}

\begin{proof}[Proof of Theorem~\ref{thm:global_loc}] Let $\delta >0$ be such that if $\xi \in \partial \Omega$, $p,q \in B_{\delta}(\xi)$, and $\sigma:[a,b] \rightarrow \Omega$ is a geodesic with $\sigma(a)=p$ and $\sigma(b)=q$ then $\sigma([a,b]) \subset B_{\epsilon}(\xi)$.

Now suppose $p,q \in B_{\delta}(\xi)$ and $\sigma:[a,b] \rightarrow \Omega$ is a geodesic joining them, then by the choice of $\delta$ we have $\sigma([a,b]) \subset B_{\epsilon}(\xi)$. But then by Lemma~\ref{lem:global_loc_1} we have 
\begin{equation*}
d_{\Omega \cap V_{\xi}}(p,q)=d_{\Omega \cap V_\xi}(\sigma(a),\sigma(b)) \leq d_\Omega(\sigma(a),\sigma(b))+K = d_\Omega(p,q)+K. \qedhere
\end{equation*}
\end{proof}

\section{Gromov hyperbolicity}

In this section we prove the first part of Theorem~\ref{thm:loc_main}.

\begin{theorem}
Suppose $\Omega$ is locally convexifiable and has finite local line type $L$. Then $(\Omega, d_{\Omega})$ is Gromov hyperbolic. 
\end{theorem}

\begin{proof}
First suppose for a contradiction that $\Omega$ is not Gromov hyperbolic. Then for all $n>0$ there exists points $x_n,y_n,z_n \in \Omega$,  geodesics $\sigma_{x_ny_n}, \sigma_{y_nz_n}, \sigma_{z_nx_n}$ joining them, and a point  $u_n \in \sigma_{x_ny_n}$ such that
\begin{align*}
d_{\Omega}(u_n, \sigma_{y_n z_n} \cup \sigma_{x_n z_n}) > n.
\end{align*}
By passing to a subsequence we may assume that $x_n,y_n,z_n,u_n \rightarrow x_{\infty},y_{\infty}, z_{\infty}, u_{\infty} \in \overline{\Omega}$. 

Fix a good convex atlas $\{(V_{\xi}, \Phi_{\xi}) : \xi \in \partial \Omega \}$ of $\Omega$. By Theorem~\ref{thm:inf_loc}, Lemma~\ref{lem:global_loc_2}, and Theorem~\ref{thm:global_loc} we can find $\epsilon, \delta,K>0$ such that 
\begin{enumerate}
\item if $\xi \in \partial \Omega$ then $B_{\epsilon}(\xi) \subset V_{\xi}$,
\item if $\xi \in \partial \Omega$, $p,q \in B_{\delta}(\xi)$, and $\sigma:[a,b] \rightarrow \Omega$ is a geodesic with $\sigma(a)=p$ and $\sigma(b)=q$ then $\sigma([a,b]) \subset B_{\epsilon}(\xi)$, and
\item if $\xi \in \partial \Omega$ and $\sigma:[a,b] \rightarrow \Omega$ is a geodesic with $\sigma([a,b]) \subset B_{\epsilon}(\xi)$ then $\sigma$ is a $K$-almost-geodesic in $(\Omega \cap V_{\xi}, d_{\Omega \cap V_{\xi}})$.
\end{enumerate}

\noindent \textbf{Special case 1:} Assume $u_{\infty} \in \Omega$. 

\begin{proof} Since 
\begin{align*}
d_{\Omega}(u_n, \{ x_n,y_n,z_n\}) > n
\end{align*}
we see that $x_{\infty}, y_{\infty}, z_{\infty} \in \partial \Omega$. By Lemma~\ref{lem:global_loc_2}, we must have $x_{\infty} \neq y_{\infty}$. Thus by relabeling we may assume that $z_{\infty} \neq x_{\infty}$. We can assume that $x_n \in B_{\epsilon}(x_{\infty})$ for all $n$. Parametrize $\sigma_{x_nz_n}:[0,T_n] \rightarrow \Omega$ such that $\sigma_{x_nz_n}(0) = x_n$ and let 
\begin{align*}
T_n^\prime = \sup\{ t \in [0,T_n] : \sigma_{x_nz_n}([0,t]) \subset B_{\epsilon}(x_{\infty})\}.
\end{align*}
Then $\sigma_{x_nz_n}|_{[0,T_n^\prime]}$ is a $K$-almost-geodesic in $\Omega \cap V_{x_{\infty}}$. Then Proposition~\ref{prop:rough_geodesic_limits_exist} implies that we may pass to a subsequence and find $\alpha_n \in [0,T_n^\prime]$ such that $t \rightarrow \sigma_{x_nz_n}( t+\alpha_n)$ converges locally uniformly to a $K$-almost-geodesic $\sigma:\Rb \rightarrow \Omega \cap V_{x_{\infty}}$. But then 
\begin{align*}
d_{\Omega}(u_\infty, \sigma(0)) = \lim_{n \rightarrow \infty} d_{\Omega}(u_n, \sigma_{x_nz_n}(\alpha_n)) \geq  \lim_{n \rightarrow \infty} d_{\Omega}(u_n, \sigma_{x_nz_n})=\infty
\end{align*}
which contradicts the fact that $u_\infty, \sigma(0) \in \Omega$.
 \end{proof}

We can now suppose $u_{\infty} = \xi \in \partial \Omega$. Then let $\Omega_{\xi} = \Phi_{\xi}( V_{\xi} \cap \Omega)$ and $\wh{u}_n = \Phi_{\xi}(u_n)$. By Theorem~\ref{thm:gaussier} we can pass to a subsequence so that there exists affine maps $A_n \in \Aff(\Cb^d)$ and a $\Cb$-proper open convex set $\wh{\Omega}$ of the form 
\begin{align*}
\wh{\Omega} = \{ (z_1, \dots, z_d) : \Imaginary(z_1) > P(z_2, \dots, z_d) \}
\end{align*}
where $P$ is a non-negative non-degenerate convex polynomial with $P(0)=0$ and
\begin{enumerate}
\item $A_n \Omega_{\xi} \rightarrow \wh{\Omega}$ in the local Hausdorff topology,
\item $A_n u_n \rightarrow \wh{u}_{\infty} \in \wh{\Omega}$,
\item $(A_n \Omega_{\xi})_{n \in \Nb}$ is locally L-convex sequence.
\end{enumerate}
\ \newline
\noindent \textbf{Special Case 2:} Assume $x_{\infty}=y_{\infty}=z_{\infty} =\xi$.

\begin{proof} 
By passing to a subsequence we may suppose that $x_n, y_n, z_n \in B_{\delta}(\xi)$ for all $n$, $A_n \Phi_{\xi}(x_n) \rightarrow \wh{x}_{\infty} \in \overline{\Cb^d}$, $A_n \Phi_{\xi}(y_n) \rightarrow \wh{y}_{\infty} \in \overline{\Cb^d}$, and $A_n \Phi_{\xi}(z_n) \rightarrow \wh{z}_{\infty} \in \overline{\Cb^d}$.

By our choice of $\delta>0$
\begin{align*}
\sigma_{x_ny_n}, \sigma_{y_nz_n}, \sigma_{z_nx_n} \subset B_{\epsilon}(\xi)
\end{align*}
hence 
\begin{align*}
\wh{\sigma}_{x_ny_n}:=\Phi_{\xi}(\sigma_{x_n y_n}), \quad \wh{\sigma}_{y_nz_n}:=\Phi_{\xi}(\sigma_{y_nz_n}), \quad \wh{\sigma}_{z_nx_n}:=\Phi_{\xi}(\sigma_{z_nx_n})
\end{align*}
are all $K$-almost-geodesics in $(\Omega_{\xi}, d_{\Omega_{\xi}})$. 

Now suppose $\wh{\sigma}_{x_n y_n}:[a_n,b_n] \rightarrow \Omega_{\xi}$ is parametrized so that $\wh{\sigma}_{x_n y_n}(0)=\wh{u}_n$. Then we can pass to a subsequence so that $A_n\wh{\sigma}_{x_ny_n}$ converges locally uniformly to a $K$-almost-geodesic $\wh{\sigma}:\Rb \rightarrow \wh{\Omega}$. By Proposition~\ref{prop:rough_geodesic_limits_exist_2}
\begin{align*}
\lim_{t \rightarrow -\infty} \wh{\sigma}(t) = \lim_{n \rightarrow \infty} A_n \Phi_{\xi}(x_n) = \wh{x}_{\infty}
\end{align*}
and 
\begin{align*}
\lim_{t \rightarrow +\infty} \wh{\sigma}(t) = \lim_{n \rightarrow \infty} A_n \Phi_{\xi}(y_n) = \wh{y}_{\infty}
\end{align*}
for some $\wh{x}_{\infty}, \wh{y}_{\infty} \in \overline{\Cb^d}$. By Proposition~\ref{prop:rough_geodesics_well_behaved} $\wh{x}_{\infty} \neq \wh{y}_{\infty}$. So by relabeling we can suppose that $\wh{x}_{\infty} \neq \wh{z}_{\infty}$. Then, by Proposition~\ref{prop:rough_geodesic_limits_exist}, there exists a parametrization of $A_n \wh{\sigma}_{x_n z_n}$ which converges locally uniformly to a $K$-almost-geodesic $\wh{\gamma}:\Rb \rightarrow \wh{\Omega}$. But then 
\begin{align*}
d_{\wh{\Omega}}(\wh{u}_{\infty}, \wh{\gamma}(0))
& = \lim_{n \rightarrow \infty} d_{A_n\Omega_{\xi}}(A_n\wh{u}_n, A_n \wh{\sigma}_{x_n z_n}(0)) = \lim_{n \rightarrow \infty} d_{\Omega_{\xi}}(\wh{u}_n, \wh{\sigma}_{x_n z_n}(0)) \\
&= \lim_{n \rightarrow \infty} d_{\Omega \cap V_{\xi}}(u_n, \sigma_{x_n z_n}(0)) \geq  \lim_{n \rightarrow \infty} d_{\Omega}(u_n, \sigma_{x_n z_n}) = \infty
\end{align*}
 which is a contradiction.
\end{proof}

We now prove the general case. Suppose $\sigma_{x_ny_n}:[a_n,b_n] \rightarrow \Omega$ is parametrized so that $\sigma_{x_ny_n}(0)=u_n$. Let 
\begin{align*}
a_n^\prime = \inf \{ t \in [a_n,b_n] : \sigma_{x_ny_n}([t,0]) \subset B_{\epsilon}(\xi)\}
\end{align*}
and 
\begin{align*}
b_n^\prime = \sup \{ t \in [a_n,b_n] : \sigma_{x_ny_n}([0,t]) \subset B_{\epsilon}(\xi)\}.
\end{align*}
Since $u_n \rightarrow \xi$, by Theorem~\ref{thm:global_loc} we have that $a_n^\prime \rightarrow -\infty$ and $b_n^\prime \rightarrow +\infty$. 

Also
\begin{align*}
\wh{\sigma}_n : =(A_n \circ \Phi_{\xi} \circ \sigma_{x_ny_n})|_{[a_n^\prime, b_n^\prime]}
\end{align*}
 is a $K$-almost-geodesic in $(A_n\Omega_{\xi}, d_{A_n\Omega_{\xi}})$. Hence we may pass to a subsequence such that $\wh{\sigma}_n$ converges locally uniformly to a $K$-almost-geodesic $\wh{\sigma} : \Rb \rightarrow \wh{\Omega}$. By passing to a subsequence we may assume that 
  \begin{align*}
 \lim_{n \rightarrow \infty} \wh{\sigma}_n(a_n^\prime) = \wh{x}_{\infty}
 \end{align*}
 and 
   \begin{align*}
 \lim_{n \rightarrow \infty} \wh{\sigma}_n(b_n^\prime) = \wh{y}_{\infty}
 \end{align*}
 for some $\wh{x}_{\infty}, \wh{y}_{\infty} \in \overline{\Cb^d}$. 
 
 The points $x_{\infty}, y_{\infty}$ and $\wh{x}_{\infty}, \wh{y}_{\infty}$ have the following relationship:
  
 \begin{observation}\label{obs:stupid} If $x_{\infty} \neq \xi$ then $\wh{x}_{\infty} = \infty$. Likewise, if $y_{\infty}\neq \xi$ then $\wh{y}_{\infty}  = \infty$.
 \end{observation} 
 
\begin{proof} 
If $x_{\infty} \neq \xi$ then
 \begin{align*}
 \lim_{n \rightarrow \infty} \norm{ \sigma_{x_ny_n}(a_n^\prime) - u_n} >0.
 \end{align*}
So $\wh{x}_{\infty} = \infty$ by the proof of Theorem~\ref{thm:gaussier}. The $y$ case is identical.
\end{proof}
 
Now by Proposition~\ref{prop:rough_geodesic_limits_exist_2}
 \begin{align*}
 \lim_{t \rightarrow -\infty} \wh{\sigma}(t)= \lim_{n \rightarrow \infty} \wh{\sigma}_n(a_n^\prime) = \wh{x}_{\infty}
 \end{align*}
 and 
  \begin{align*}
 \lim_{t \rightarrow +\infty} \wh{\sigma}(t) =  \lim_{n \rightarrow \infty} \wh{\sigma}_n(b_n^\prime) = \wh{y}_{\infty}.
 \end{align*}
 Hence $\wh{x}_{\infty} \neq \wh{y}_{\infty}$ by Proposition~\ref{prop:rough_geodesics_well_behaved}. So by relabeling we may assume that $\wh{x}_{\infty} \neq \infty$. This implies, by Observation~\ref{obs:stupid}, that $x_{\infty} = \xi$. So by passing to subsequence we can suppose that $x_n \in B_{\delta}(\xi)$ for all $n$. Then by our choice of $\delta$, $a_n^\prime = a_n$ for all $n$ and
 \begin{align*}
 \lim_{n \rightarrow\infty} A_n \Phi_\xi(x_n) = \lim_{n \rightarrow \infty}   \wh{\sigma}_n(a_n^\prime) = \wh{x}_{\infty}.
 \end{align*}
 
 Now suppose $\sigma_{x_n z_n} : [0,T_n] \rightarrow \Omega$ is parametrized so that $\sigma_{x_nz_n}(0) = x_n$. Let 
 \begin{align*}
 T_n^\prime = \sup \{ t \in [0,T_n] : \sigma_{x_nz_n}([0,t]) \subset B_{\epsilon}(\xi)\},
 \end{align*}
then
 \begin{align*}
\wh{\gamma}_n : =(A_n \circ \Phi_{\xi} \circ \sigma_{x_nz_n})|_{[0, T_n^\prime]}
\end{align*}
is an $K$-almost-geodesic in $(A_n\Omega_{\xi}, d_{A_n\Omega_{\xi}})$. By passing to a subsequence we may assume that 
  \begin{align*}
 \lim_{n \rightarrow \infty} \wh{\gamma}_n(T_n^\prime) = \wh{z}_{\infty}
 \end{align*}
for some $\wh{z}_{\infty} \in \overline{\Cb^d}$. 

If $\wh{z}_{\infty} \neq \wh{x}_{\infty}$ then, by Proposition~\ref{prop:rough_geodesic_limits_exist}, there exists some $\alpha_n \in [0,T_n^\prime]$ such that $t \rightarrow \wh{\gamma}_n(t+\alpha_n)$ converges to a  $K$-almost-geodesic $\wh{\gamma}:\Rb \rightarrow \wh{\Omega}$. But then 
  \begin{align*}
d_{\wh{\Omega}}(\wh{u}_{\infty}, \wh{\gamma}(0))
& = \lim_{n \rightarrow \infty} d_{A_n\Omega_{\xi}}(A_n\wh{u}_n, \wh{\gamma}_{n}(0))  = \lim_{n \rightarrow \infty} d_{\Omega \cap V_{\xi}}(u_n, \sigma_{x_n z_n}(\alpha_n)) \\
& \geq  \lim_{n \rightarrow \infty} d_{\Omega}(u_n, \sigma_{x_n z_n}) = \infty
\end{align*}
 which is a contradiction. 
 
It remains to consider the case where $\wh{z}_{\infty} = \wh{x}_{\infty}$. Then since $\wh{z}_{\infty} \neq \infty$ arguing as in Observation~\ref{obs:stupid} shows that $z_{\infty} = \xi$. So by passing to a subsequence we can suppose that $z_n \in B_{\delta}(\xi)$ for all $n$. Then by our choice of $\delta$, $T_n^\prime = T_n$.  So 
\begin{align*}
\lim_{n \rightarrow \infty} A_n \Phi_{\xi}(z_n) = \lim_{n \rightarrow \infty}\wh{\gamma}_n(T_n^\prime)= \wh{z}_{\infty}.
\end{align*} 
 
 Suppose $\sigma_{z_n y_n} : [0,S_n] \rightarrow \Omega$ is parametrized so that $\sigma_{z_ny_n}(0) = z_n$. Let 
 \begin{align*}
S_n^\prime = \sup \{ s \in [0,S_n] : \sigma_{z_ny_n}([0,s]) \subset B_{\epsilon}(\xi)\}.
 \end{align*}
Then 
 \begin{align*}
\wh{\eta}_n : =(A_n \circ \Phi_{\xi} \circ \sigma_{y_nz_n})|_{[0, S_n^\prime]}
\end{align*}
is an $K$-almost-geodesic in $(A_n\Omega_{\xi}, d_{A_n\Omega_{\xi}})$. By passing to a subsequence we may assume that 
  \begin{align*}
 \lim_{n \rightarrow \infty} \wh{\eta}_n(S_n^\prime) = \wh{w}_{\infty} \in \overline{\Cb^d}
 \end{align*}
 for some $\wh{w}_{\infty} \in \overline{\Cb^d}$. If $\wh{w}_{\infty} = \wh{z}_{\infty}$ then $\wh{w}_{\infty} \neq \infty$ and hence arguing as in Observation~\ref{obs:stupid} shows that $y_{\infty} = \xi$. But then we are in Special Case 2. 
 
If $\wh{z}_{\infty} \neq \wh{w}_{\infty}$ then, by Proposition~\ref{prop:rough_geodesic_limits_exist}, there exists some $\beta_n \in [0,S_n^\prime]$ such that $t \rightarrow \wh{\eta}_n(t+\beta_n)$ converges to a $K$-almost-geodesic $\wh{\eta}:\Rb \rightarrow \wh{\Omega}$. But then 
  \begin{align*}
d_{\wh{\Omega}}(\wh{u}_{\infty}, \wh{\eta}(0))
& = \lim_{n \rightarrow \infty} d_{A_n\Omega_{\xi}}(A_n\wh{u}_n, \wh{\eta}_{n}(0))  = \lim_{n \rightarrow \infty} d_{\Omega \cap V_{\xi}}(u_n, \sigma_{y_n z_n}(\beta_n)) \\
& \geq  \lim_{n \rightarrow \infty} d_{\Omega}(u_n, \sigma_{y_n z_n}) = \infty
\end{align*}
 which is a contradiction.  
 
 Thus $(\Omega, d_{\Omega})$ is Gromov hyperbolic. 
 \end{proof}
 
 \section{Geodesics and the Gromov product}
 
 Before proving the second part of Theorem~\ref{thm:loc_main} we need to establish some properties of geodesics in locally convexifiable domains. 
 
\begin{proposition}\label{prop:gromov_prod_infinite_loc}
Suppose $\Omega$ is locally convexifiable and has finite local line type $L$. Assume $p_n, q_n \subset \Omega$ are sequences of points such that $\lim_{n \rightarrow \infty} p_n = \xi_1 \in \partial \Omega$ and $\lim_{ n \rightarrow \infty} q_n = \xi_2 \in \partial \Omega$. If
\begin{align*}
\lim_{n \rightarrow \infty} (p_n|q_n)_o = \infty
\end{align*}
for some $o \in \Omega$, then $\xi_1 = \xi_2$.
\end{proposition}

\begin{proof}
Suppose for a contradiction that $\xi_1 \neq \xi_2$. Let $\sigma_n : [0,T_n] \rightarrow \Omega$ be a geodesic such that $\sigma_n(0) = p_n$ and $\sigma_n(T_n) = q_n$ for some $T_n> 0$.

Let $\{ (V_{\xi}, \Phi_{\xi}) : \xi \in \partial \Omega\}$ be a good convex atlas. Let $K,\delta >0$ be parameters as in the conclusion of Theorem~\ref{thm:global_loc}. 

By passing to a subsequence we can assume that $p_n \in B_{\delta}(\xi_1)$ for all $n$. Define
\begin{align*}
T_n^\prime = \sup \{ t : \sigma_n([0,t]) \subset B_{\delta}(\xi_1) \}.
\end{align*}
Since $\xi_1 \neq \xi_2$ we see that 
\begin{align*}
\lim_{n \rightarrow \infty} \norm{ \sigma_n(0)-\sigma_n(T_n^\prime)} > 0.
\end{align*}
Moreover $\sigma_n|_{[0,T_n^\prime]}$ is a $K$-almost-geodesic in $(\Omega \cap V_{\xi_1}, d_{\Omega \cap V_{\xi_1}})$. So by Proposition~\ref{prop:rough_geodesic_limits_exist} we can pass to a subsequence and find $\alpha_n \in  [0,T_n^\prime]$ such that the $K$-almost-geodesics $t \rightarrow \sigma(t+\alpha_n)$ converge locally uniformly to a $K$-almost-geodesic $\wh{\sigma}:\Rb \rightarrow V_{\xi} \cap \Omega$. 

Now since $\sigma_n$ is a geodesic 
\begin{align*}
(p_n|q_n)_o 
&= \frac{1}{2} \left( d_{\Omega}(p_n,o)+ d_{\Omega}(o,q_n) - d_{\Omega}(p_n,q_n) \right) \\
&=  \frac{1}{2} \left( d_{\Omega}(p_n,o)+ d_{\Omega}(o,q_n) - d_{\Omega}(p_n,\sigma_n(\alpha_n)) - d_{\Omega}(\sigma_n(\alpha_n), q_n) \right)\\
& \leq d_{\Omega}(o, \sigma_n(\alpha_n)).
\end{align*}
But then
\begin{align*}
\infty =\lim_{n \rightarrow \infty} (p_n|q_n)_o  \leq  \lim_{n\rightarrow \infty}  d_{\Omega}(o, \sigma_n(\alpha_n))=  d_{\Omega}(o, \wh{\sigma}(0))
\end{align*}
which is a contradiction. 
\end{proof}

Proposition~\ref{prop:gromov_prod_infinite_loc} has two corollaries about the behavior of geodesics:

  \begin{corollary}\label{cor:geod_loc_1}
 Suppose $\Omega$ is locally convexifiable and has finite local line type $L$. If $\sigma :[0,\infty) \rightarrow \Omega$ is a geodesic ray then 
 \begin{align*}
 \lim_{t \rightarrow \infty} \sigma(t)
 \end{align*}
 exists in $\partial \Omega$.
 \end{corollary}
 
 \begin{proof}
 Suppose for a contradiction that the limit does not exist. Then there exists sequences $s_n \rightarrow \infty$ and $t_n \rightarrow \infty$ such that $\sigma(s_n) \rightarrow \xi_1$, $\sigma_n(t_n) \rightarrow \xi_2$, and $\xi_1 \neq \xi_2$. But 
 \begin{align*}
 ( \sigma(s_n) | \sigma(t_n))_o = \min\{ t_n, s_n\}
 \end{align*}
 which contradicts Proposition~\ref{prop:gromov_prod_infinite_loc}.
 \end{proof}
 
   \begin{corollary}\label{cor:geod_loc_2}
 Suppose $\Omega$ is locally convexifiable and has finite local line type $L$. Assume $\sigma_n :[0,T_n] \rightarrow \Omega$ is a sequence of geodesics with $T_n \rightarrow \infty$ and $\sigma_n$ converges locally uniformly to a geodesic $\sigma:  [0,\infty) \rightarrow \Omega$. Then 
 \begin{align*}
 \lim_{t \rightarrow \infty} \sigma(t) = \lim_{n \rightarrow \infty} \sigma_n(T_n).
 \end{align*}
 \end{corollary}
 
 \begin{proof}
We can assume that 
 \begin{align*}
 \lim_{n \rightarrow \infty} \sigma_n(T_n) = \xi_1
 \end{align*}
 for some $\xi_1 \in \partial \Omega$. If 
  \begin{align*}
 \lim_{t \rightarrow \infty} \sigma(t) \neq \xi_1
 \end{align*}
 then we can find a sequence $s_m^\prime \rightarrow \infty$ such that $\sigma(s_m^\prime) \rightarrow \xi_2 \in \partial \Omega$ with $\xi_2 \neq \xi_1$. Then since $\sigma_n$ converges locally uniformly to $\sigma$ we can find $s_n \rightarrow \infty$ such that $\sigma_n(s_n) \rightarrow \xi_2$. Moreover there exists $R>0$ such that 
 \begin{align*}
 d_{\Omega}(\sigma_n(0),\sigma(0)) < R.
 \end{align*} 
Then
 \begin{align*}
 (\sigma_n(T_n)|\sigma_n(s_n))_{\sigma(0)} \geq \min\{s_n,T_n\} -2R
 \end{align*}
  which contradicts Proposition~\ref{prop:gromov_prod_infinite_loc}.
  \end{proof}
 
 \section{The Gromov boundary}

In this section we complete the proof of Theorem~\ref{thm:loc_main} (Theorem~\ref{thm:loc_main_i} in the introduction) by showing:

\begin{theorem}
Suppose $\Omega$ is locally convexifiable and has finite local line type $L$. Then the identity map $\Omega \rightarrow \Omega$ extends to a homeomorphism $\Omega \cup \Omega(\infty) \rightarrow \Omega \cup \partial \Omega$. 
\end{theorem}
 
\begin{proof}
Suppose $\sigma:[0,\infty) \rightarrow \Omega$ is a geodesic ray. Then 
\begin{align*}
\lim_{t \rightarrow \infty} \sigma(t)
\end{align*}
exists by Corollary~\ref{cor:geod_loc_1}. We next claim that this limit only depends on the asymptotic class of $\sigma$. So suppose that $\sigma_1,\sigma_2: [0,\infty) \rightarrow \Omega$ are two asymptotic geodesic rays. Since 
 \begin{align*}
 \sup_{t \geq 0} d_{\Omega}(\sigma_1(t), \sigma_2(t)) < \infty
 \end{align*}
 we see that 
 \begin{align*}
\lim_{t \rightarrow \infty} (\sigma_1(t)|\sigma_2(t))_o = \infty.
\end{align*}
Then by Proposition~\ref{prop:gromov_prod_infinite_loc}
\begin{align*}
\lim_{t \rightarrow \infty} \sigma_1(t)=\lim_{t \rightarrow \infty} \sigma_2(t).
\end{align*}

So the map  $\Phi: \Omega \cup \Omega(\infty) \rightarrow \Omega \cup \partial \Omega$ given by 
\begin{align*}
\Phi(\xi) = \left\{ \begin{array}{ll}
\xi & \text{ if } \xi \in \Omega \\
 \lim_{t \rightarrow \infty} \sigma(t) & \text{ if } \xi=[\sigma] \in \Omega(\infty) 
\end{array}
\right.
\end{align*}
is well defined.

We claim that $\Phi$ is continuous, injective, and surjective. Since $\Omega \cup \Omega(\infty)$ is compact this will imply that $\Phi$ is a homeomorphism. 

\textbf{Surjective:} It is enough to show that for all $x \in \partial \Omega$ there exists a geodesic ray $\sigma:[0,\infty) \rightarrow \Omega$ such that 
\begin{align*}
\lim_{t \rightarrow \infty} \sigma(t) = x.
\end{align*}
Fix a point $o \in \Omega$ and a sequence $x_n \in \Omega$ such that $x_n \rightarrow x$. Then let $\sigma_n:[0,T_n] \rightarrow \Omega$ be a geodesic such that $\sigma_n(0)=o$ and $\sigma_n(T_n) = x_n$. Now we can pass to a subsequence so that $\sigma_n$ converges locally uniformly to a geodesic ray $\sigma : [0,\infty) \rightarrow \Omega$. Then by Corollary~\ref{cor:geod_loc_2} 
\begin{align*}
\lim_{t \rightarrow \infty} \sigma(t) =  \lim_{n \rightarrow \infty} \sigma_n(T_n) = x.
\end{align*}
Hence $\Phi$ is onto. 

\textbf{Continuous:} Suppose $\xi_n \rightarrow \xi$ in $\Omega \cup \Omega(\infty)$. If $\xi \in \Omega$ then clearly $\Phi(\xi_n) \rightarrow \Phi(\xi)$. So we can assume that $\xi_n \in \Omega(\infty)$. Since $\Omega \cup \partial\Omega$ is compact, it is enough to show that every convergent subsequence of $\Phi(\xi_n)$ converges to $\Phi(\xi)$. So we may also assume that $\Phi(\xi_n) \rightarrow x$ for some $x \in \partial \Omega$. 

Now fix $o \in \Omega$ and  let $\sigma_n:[0,T_n) \rightarrow \Omega$ be a geodesic with $\sigma_n(0)=o$ and 
\begin{align*}
\lim_{t \rightarrow T_n} \sigma_n(t)= \Phi(\xi_n).
\end{align*}
Notice that $T_n$ could be $\infty$. Now we can pick $T_n^\prime \in (0, T_n)$ such that 
\begin{align*}
\lim_{n \rightarrow \infty} \sigma_n(T_n^\prime) = \lim_{n \rightarrow \infty} \Phi(x_n) = x.
\end{align*}
Now by the definition of topology on $\Omega \cup \Omega(\infty)$, if $\sigma$ is the limit of a convergent subsequence $\sigma_{n_k}$ of $\sigma_n$ then $\xi = [\sigma]$. Moreover, by Corollary~\ref{cor:geod_loc_2}
\begin{align*}
 \lim_{t \rightarrow \infty} \sigma(t) =\lim_{k \rightarrow \infty} \sigma_{n_k}(T_{n_k}^\prime) = x.
 \end{align*}
Thus $\Phi(\xi) = x$. So $\Phi$ is continuous. 
 
 \textbf{Injective:} Suppose for a contradiction that $\Phi(\xi_1)=\Phi(\xi_2)$ for some $\xi_1 \neq \xi_2$ in $\Omega \cup \Omega(\infty)$. Since $\Phi|_{\Omega} = id$ we must have that $\xi_1,\xi_2 \in \Omega(\infty)$. Now let $\sigma_1, \sigma_2$ be geodesic representatives of $\xi_1, \xi_2$. Then 
 \begin{align*}
 \lim_{t \rightarrow \infty} \sigma_1(t) = \Phi(\xi_1) = \Phi(\xi_2)= \lim_{t \rightarrow \infty} \sigma_2(t).
 \end{align*}
This implies, by Proposition~\ref{prop:gromov_prod_infinite} and Theorem~\ref{thm:global_loc}, that 
\begin{align*}
\lim_{t \rightarrow \infty} (\sigma_1(t) | \sigma_2(t))_o = \infty.
\end{align*}
But by~\cite[Chapter III.H Lemma 3.13]{BH1999} this happens only if $\xi_1 = [\sigma_1]=[\sigma_2]=\xi_2$ which is a contradiction. 

\end{proof}

\bibliographystyle{alpha}
\bibliography{complex_kob}

\end{document}